\documentclass[11pt,a4paper]{amsart}
\usepackage[utf8]{inputenc}
\usepackage[english]{babel}
\usepackage[T1]{fontenc}
\usepackage{amsmath}
\usepackage{amsthm}
\usepackage{amsfonts}
\usepackage{amssymb}
\usepackage{lmodern}
\usepackage{mathrsfs}
\usepackage{physics}
\usepackage[colorlinks=true,citecolor=blue,linkcolor=black]{hyperref}
\usepackage{graphicx}
\usepackage[left=2.5cm,right=2.5cm,top=2.5cm,bottom=2.5cm]{geometry}
\usepackage{enumerate}
\usepackage{tikz-cd}
\usepackage{bm}
\usetikzlibrary{backgrounds}
\usetikzlibrary{calc}
\usetikzlibrary{hobby}
\usetikzlibrary{decorations.markings}
\usepackage{caption}
\usepackage{subcaption}
\usepackage{array}
\usepackage{float}

\numberwithin{equation}{section}

\DeclareMathOperator{\diag}{diag}
\DeclareMathOperator{\ram}{ram}
\DeclareMathOperator{\slope}{slope}
\DeclareMathOperator{\Hom}{Hom}
\DeclareMathOperator{\End}{End}

\DeclareMathOperator{\Sto}{\mathbb{S}to}
\DeclareMathOperator{\Card}{Card}
\DeclareMathOperator{\Id}{Id}

\DeclareMathOperator{\Aut}{Aut}
\DeclareMathOperator{\Irr}{Irr}
\DeclareMathOperator{\GrAut}{GrAut}
\DeclareMathOperator{\lcm}{lcm}
\DeclareMathOperator{\SL}{SL}

\DeclareMathOperator{\GL}{GL}

\newcommand{\N}{\mathbb{N}}
\newcommand{\Z}{\mathbb{Z}}

\newcommand{\C}{\mathbb{C}}
\newcommand{\Pbb}{\mathbb{P}}
\newcommand{\A}{\mathbb{A}}
\newcommand{\Lbb}{\mathbb{L}}

\newcommand{\Ical}{\mathcal{I}}

\newcommand{\Lcal}{\mathcal{L}}

\newcommand{\Acal}{\mathcal{A}}
\newcommand{\Ccal}{\mathcal{C}}

\newcommand{\MDR}{\mathcal M_{\mathrm{DR}}}
\newcommand{\MB}{\mathcal M_{\mathrm{B}}}
\newcommand{\MDol}{\mathcal M_{\mathrm{Dol}}}

\newcommand{\cir}[1]{\langle #1 \rangle}
\newcommand{\sslash}{\mathbin{/\mkern-6mu/}}

\begin{document}


\renewcommand{\proofname}{Proof}
\renewcommand{\Re}{\operatorname{Re}}
\renewcommand{\Im}{\operatorname{Im}}
\renewcommand{\labelitemi}{$\bullet$}

\newtheorem{theorem}{Theorem}[section]
\newtheorem{proposition}[theorem]{Proposition}
\newtheorem{lemma}[theorem]{Lemma}
\newtheorem{corollary}[theorem]{Corollary}

\theoremstyle{definition}
\newtheorem{definition}[theorem]{Definition}
\newtheorem{example}{Example}[section]
\theoremstyle{remark}
\newtheorem{remark}{Remark}[section]
\newtheorem*{claim}{Claim}

\title{Diagrams and irregular connections on the Riemann sphere}

\author{Jean Douçot}
\address{University of Lisbon, Group of mathematical physics, Lisboa, Portugal}
\email{jmdoucot@fc.ul.pt}

\begin{abstract}
We define a diagram associated to any algebraic connection on a vector bundle on a Zariski open subset of the Riemann sphere, extending the definition of Boalch--Yamakawa to the general case featuring several irregular singularities, possibly ramified.  
We prove that the diagram is invariant under the symplectic automorphisms of the Weyl algebra, encompassing the Fourier--Laplace transform. 
As an application, we establish several new cases of
the observation that different Lax representations of a given Painlevé-type equation 
may be read off directly from the diagram, 
corresponding to connections with different formal data, 
usually on different rank bundles. 
\end{abstract}

\maketitle

\section{Introduction}

This work takes place in line with a series of studies by several authors establishing links between symplectic moduli spaces of meromorphic connections on the Riemann sphere and graphs (or doubled quivers). Our main horizon, as outlined in  
\cite{boalch2018wild},
is to develop a theory of Dynkin-like diagrams for the (wild) nonabelian Hodge spaces of Riemann surfaces, in the perspective of their classification.\\

First recall that the study of moduli spaces in 2d gauge theory can be traced back Riemann's definition of the {\em monodromy representation} of a linear differential equation. This leads to a purely topological description of a special class of algebraic differential equations, the linear connections with regular singularities. In turn this leads to the definition of the character variety or Betti moduli space:
$$\MB = \Hom(\pi_1(\Sigma^\circ),G)/G$$
of any smooth complex algebraic curve $\Sigma^\circ$.
Here $G=\GL_n(\mathbb C)$, and the points of $\MB$ correspond to 
(polystable) linear connections with regular singularities, 
on rank $n$ algebraic vector bundles on $\Sigma^\circ$, as in  \cite{deligne1970equations} for example.

For applications to nonlinear differential equations (such as the Painlev\'e equations) the most interesting case is 
when $\Sigma^\circ$ has genus zero.
In that case one can define an additive moduli space $\MDR^*$
whose points correspond to isomorphism classes of Fuchsian systems
$$d-Adz,\qquad A=\frac{A_1}{z-a_1}+\cdots+\frac{A_m}{z-a_m}, \quad A_i\in \End(\mathbb C^n)$$
so that that the Riemann--Hilbert map, taking the monodromy representation,  is a holomorphic map
$$\nu:\MDR^*\ \to \ \MB$$
between two algebraic varieties of the same dimension.
In brief $\nu$ is induced from a $G$-equivariant map from the (additive) space of coefficients $\{A_i\}\cong \End(\mathbb C^n)^m$ to the (multiplicative) space of monodromy data $\cong \GL_n(\mathbb C)^m$.
If we fix the adjoint orbits of the residues or 
conjugacy classes of
the local monodromies then one obtains symplectic moduli spaces.

The simplest case is to take rank $n=2$ with $m=3$ poles at finite distance, and then the symplectic moduli spaces have complex dimension two and are phase spaces for the Painlev\'e VI differential equation.
This case was famously related to the affine Dynkin diagram of type 
$D_4$ by Okamoto \cite{okamoto1986studies, okamoto1992painleve} 
who showed that the Painlev\'e VI 
equations admit a symmetry group isomorphic to
the (extended) affine Weyl group of type $\widehat D_4$.

This relation between graphs and connections was better understood
and generalised  by 
Crawley-Boevey \cite{crawley2003matrices,crawley2006multiplicative} who showed 
(building on work of Nakajima, Kronheimer and Kraft--Procesi)  how to construct such spaces directly from graphs/doubled quivers.
In effect they showed that if $n=2, m=3$ then:

$\bullet$ 1) the symplectic additive moduli spaces $\MDR^*$ of Fuchsian systems are isomorphic to  the Nakajima quiver varieties of the doubled quiver $\widehat D_4$, and

$\bullet$ 2) one can define a notion of ``multiplicative quiver variety'' so that the  symplectic character varieties $\MB$ are examples of multiplicative quiver varieties attached to $\widehat D_4$.

Moreover they did this for all $n,m$ not just $n=2,m=3$, thus defining a graph for any Fuchsian system.
The graphs that occur are the star-shaped graphs, consisting of 
$m+1$ legs (type $A$ Dynkin graphs) glued 
to a single node at one end.
They then used the Kac--Moody root system attached to the graph 
in the study of the corresponding moduli spaces (the tame Deligne--Simpson problems).

In particular this suggests a better way to parameterise the 
choice of data needed to determine the spaces: 
rather than choose $\Sigma^\circ$ and the local conjugacy classes to get a space, one now chooses a star-shaped graph and some data on the graph (a dimension vector $\mathbf d$ and a scalar at each node).
For example the complex dimension of the quiver variety is given by the formula
\begin{equation} 
2-(\mathbf d,\mathbf d)
\end{equation}
where $(\,\cdot\,,\,\cdot\,)$ is the bilinear form determined 
by the Kac--Moody Cartan matrix of the graph. 
This is useful as the same space occurs in several different ways as a moduli space of connections, that can be read off from the graph 
(for example see \cite{boalch2012simply} \S11.1 for three readings of $\widehat D_4$,  of ranks $2,3,5$).
It also explains the link between affine Dynkin diagrams (which have null roots $(\mathbf d,\mathbf d)=0$) and moduli spaces of dimension two, leading to the idea (\cite{boalch2008irregular} p.12) that the next simplest classes of examples come from hyperbolic diagrams.

In turn this leads to the idea (\cite{boalch2012simply, boalch2018wild}) 
that we should be  thinking  of 
$\MB$ as an abstract space, a global analogue of a Lie group, 
and then study the representations of it, and that the graph plays the role of Dynkin diagram (to parameterise the spaces and their representations).\\

This story has been significantly deepened in recent years, to encompass many more moduli spaces, based
on the following developments (amongst others):

$\bullet$ a) any algebraic connection over $\Sigma^\circ$ has a purely topological description, even if it has irregular singularities, 
via the Riemann--Hilbert--Birkhoff correspondence.
This originates with Stokes and Birkhoff and was extended to full generality by Sibuya \cite{sibuya1977stokes} and Malgrange \cite{malgrange1979remarques}, and then rephrased in various more convenient/intrinsic  ways by Jurkat, Deligne, Martinet--Ramis, Loday-Richaud, Boalch 
(see the overview in \cite{boalch2021topology}).
For example the surface group 
$\pi_1(\Sigma^\circ)$ can be generalised to the wild surface group,
yielding explicit presentations.

$\bullet$ b) The Stokes data in a) can be used to build generalisations of the character varieties, the wild 
character varieties. For any rank $n$ 
they were constructed algebraically, in full generality, 
in the sequence of works
\cite{boalch2007quasi, boalch2009through, boalch2014geometry, boalch2015twisted} and they come in two flavours: 
the Poisson wild character varieties 
$\MB(\bm \Sigma)$ and the symplectic 
wild character varieties $$\MB(\bm \Sigma,\bm{\mathcal C})\subset \MB(\bm\Sigma).$$
This gives  a purely algebraic approach to  Boalch's 
holomorphic symplectic manifolds  first constructed analytically (``irregular Atiyah--Bott'') \cite{boalch2001symplectic}.
The choices $\bm\Sigma,\bm\Ccal$ here will be made precise below: in brief 
we fix some boundary data consisting of some Stokes circles and 
formal monodromy conjugacy classes at each marked point.

$\bullet$ c) The nonabelian Hodge correspondence was 
extended \cite{biquard2004wild} to incorporate these new symplectic wild character varieties, showing they enjoy the richer property of being  new examples of (complete) hyperk\"ahler manifolds 
and that there are 
diffeomorphisms 
\begin{equation}\label{non_abelian_hodge_space}
\MDol\cong \MDR\xrightarrow{\nu}\MB
\end{equation}
with moduli spaces of meromorphic Higgs bundles and meromorphic connections.
This leads to the notion of {\em nonabelian Hodge space},
 a single differentiable manifold 
with the extra structures following from the 
identifications in \eqref{non_abelian_hodge_space} (see \cite{boalch2018wild}).
The classical case of Corlette, Donaldson, Hitchin 
and Simpson corresponds 
to the case where the boundary data is trivial, 
and the tame case is due to \cite{konno1993construction, nakajima1996hyperkahler} 
building on Simpson's bijective correspondence 
\cite{simpson1990harmonic},  
and Biquard's analytic package \cite{biquard1991fibres}.

Thus in brief we can choose 
data $(\bm\Sigma,\bm\Ccal)$ consisting of a Riemann surface 
and some boundary data,  and this determines some 
incredibly rich geometric objects, as in \eqref{non_abelian_hodge_space}.
As before, in genus zero there is 
a corresponding additive moduli space $\MDR^*$ 
of the same dimension and a symplectic map 
$\nu:\MDR^*\to \MB$  to the wild character variety 
(see \cite{boalch2001symplectic,boalch2007quasi}). 
And again the simplest examples of wild character varieties are of complex dimension $2$, and occur in the theory of Painlev\'e equations, this time for Painlev\'e I-V. 
Note there are many applications. For example 
the original Seiberg--Witten
integrable systems  \cite{seiberg1994electric,seiberg1994monopoles}, 
used to {solve} 4d superYang-Mills 
($\mathrm{SU}_2$ with $N_f=0,1,2,3,4$ flavours),  
are examples of integrable systems obtained as 
autonomous limits of  Painlev\'e equations  $3,5,6$, 
and involve irregular connections/Higgs fields for 
all the basic (asymptotically free) cases with $N_f<4$.\\

Thus it is natural to try to extend the story of graphs/doubled quivers to encompass some of these more general moduli spaces. 
Before stating our main result we will define 
the notion of a diagram, as a 
generalisation of a graph:

\begin{definition}[\cite{boalch2020diagrams}]
A \textit{diagram} is a pair $\Gamma=(N,B)$ where $N$ is a finite set (the set of nodes) and $B=(B_{ij})_{i,j\in N}$ is a symmetric square matrix with integer values, 
such that $B_{ii}$ is even for any $i\in N$.
A \textit{dimension vector} $\mathbf d$ for $\Gamma$ is an element of $\mathbb Z_{\geq 0}^{N}$. 
\end{definition}

We view the integers $B_{ij}$ as giving edge 
multiplicities between nodes (or edge loops if $i=j$).
A graph is thus the special case of a diagram with 
$B_{ij}\ge 0$, and it is simple (or {\em simply-laced})
if $B_{ij}\in \{0,1\}$ and $B_{ii}=0$.

Let $\Sigma$ be a genus 0 compact Riemann surface, and let us choose a point $\infty\in \Sigma$. Given these choices, in this work:

\begin{itemize}

\item For any algebraic connection $(E,\nabla)$ on a Zariski open subset of $\Sigma$, we define a diagram $\Gamma(E,\nabla)$, and

\item For any choice of minimal \textit{marking} 
 of $(E,\nabla)$, we define a dimension vector $\mathbf d$ for $\Gamma(E,\nabla)$.

\end{itemize}

See Def.\ref{def:def_diagram_introduction} for the direct definition of the diagram and Def. \ref{def:marking} for the definition of marking. 
This generalises  the previous definitions of 
\cite{crawley2006multiplicative, boalch2008irregular,boalch2012simply, boalch2020diagrams} for connections who belong to the classes of cases considered in these works, as we will list below. 
None of these previous approaches applied in the case with more than one irregular singularity.

The main properties of the diagrams are as follows. 
If $(E,\nabla)$ is a connection, let $(\bm\Sigma,\bm\Ccal)$ denote its boundary data, and let $\MB(E,\nabla)=\MB(\bm\Sigma,\bm\Ccal)$
be the corresponding symplectic wild character variety.
The diagrams only depend on the boundary data so we will sometimes write $\Gamma(\bm\Sigma,\bm\Ccal)=\Gamma(E,\nabla)$ for the diagram.

\begin{theorem}
If the connection $(E,\nabla)$ is irreducible, the dimension of its wild character variety $\MB(\bm\Sigma,\bm\Ccal)$ is given by the quiver variety dimension formula:
\begin{equation}
\dim_{\mathbb C} \MB(\bm\Sigma,\bm\Ccal)=2-(\mathbf d,\mathbf d)
\label{intro_formula_dimension}
\end{equation}
where $\mathbf d$ is the dimension vector coming from any choice of minimal marking of $(E,\nabla)$, and
$(\,\cdot\,,\,\cdot\,)$ is the bilinear form determined by the
(symmetrized) Cartan matrix $C=2-B$ of the diagram $\Gamma(\bm\Sigma,\bm\Ccal)$.
\end{theorem}

Furthermore, we can relate connections to modules over the Weyl algebra
of the affine line $\Sigma\setminus\{\infty\}$.
Then we can consider
the symplectic group 
$\SL_2(\mathbb C)$ of Weyl algebra automorphisms,
containing the Fourier--Laplace automorphism.
This group $\SL_2(\mathbb C)$ then acts
on the set of isomorphism classes of irreducible algebraic connections $(E,\nabla)$ on Zariski open subsets of $\Sigma$ (excluding rank one connections with only a pole of order less than $2$ at infinity). 
In general it changes the rank and pole orders/configurations and gives lots of mysterious isomorphisms between 
completely different wild character varieties. We have:

\begin{theorem}
The diagram is invariant under the action of 
$\SL_2(\mathbb C)$: if $A\in \SL_2(\mathbb C)$, we have
\[
\Gamma(A\cdot (E,\nabla))=\Gamma(E,\nabla).
\]
\end{theorem}

This is a key property and allows us to reduce to 
the 
case establish by 
Boalch--Yamakawa \cite{boalch2020diagrams} 
with at most one irregular singularity.
A surprising fact is that the definition of the diagram
we give is completely direct (independent of any choice of reduction to the setting of \cite{boalch2020diagrams}).\\

Before giving the definition, here are some of the precursors leading to our definition:

1) Okamoto defined affine Weyl groups for all six Painlev\'e equations (see the review \cite{okamoto1992painleve}), and so one can consider their diagrams. Beware that Okamoto and some other authors prefer to parameterise the Painlev\'e equations by a different diagram, the
{\em Okamoto diagram}.
For example Painlev\'e 2 has 
diagram $\widehat A_1$ and Okamoto diagram $\widehat E_7$.
It is not known how to define Okamoto diagrams 
for nonabelian Hodge moduli spaces of dimension $>2$.

2) The star-shaped case was understood by Crawley-Boevey as already discussed, in relation tame/Fuchsian connections.
This work was connected to Okamoto's work in 
\cite{boalch2009quivers}, and exercise 3 of that paper showed that the same story works for some of the non-star-shaped/irregular cases of dimension $2$, related to the cases of Painlev\'e 2, 4, 5: the corresponding additive moduli spaces $\MDR^*$ 
of these Painlev\'e equations (studied earlier in \cite{boalch2001symplectic}) were isomorphic to the Nakajima quiver varieties
for the quiver corresponding to their Okamoto's affine Weyl group, i.e. of types $\widehat A_1,\widehat A_2,\widehat A_3$ respectively.

3) Subsequently Boalch proved \cite{boalch2008irregular, boalch2012simply} 
that the Nakajima quiver variety 
of any complete $k$-partite graph for any $k$ appeared as an additive moduli space $\MDR^*$ of meromorphic connections on the Riemann sphere. 
More generally he defined a new class of graphs, the simply-laced supernova graphs, generalising the stars (gluing legs on to 
a $k$-partite graph), and showed all them were modular, 
i.e. that their Nakajima quiver varieties appeared as additive moduli spaces $\MDR^*$.
(In turn this led to a new theory of multiplicative quiver varieties
\cite{boalch2016global}, different to the classical theory, 
by considering the wild character varieties $\MB$ of these quiver varieties $\MDR^*$.)

4) More generally, beyond the simply-laced case, 
in the appendix of  \cite{boalch2008irregular} Boalch
conjectured that any additive moduli space $\MDR^*$
on the Riemann sphere with at most one untwisted
irregular pole (and any number of tame poles) was isomorphic to a 
Nakajima quiver variety, and gave a prescription 
to define the graph. 
This {\em quiver modularity conjecture} was proved by Hiroe--Yamakawa  \cite{hiroe2014moduli}. 
Their proof is surveyed in Yamakawa's article \cite{yamakawa2020applications}.

5) More generally a different strategy was needed since not all of the additive moduli spaces are quiver varieties  (this was remarked already in \cite{boalch2008irregular} p.3). 
The simplest counterexample is the case of Painlev\'e III, that Okamoto had related to affine $B_2$ 
(note that Nakajima quiver varieties are not defined for 
non-ADE affine Dynkin graphs).
By carefully considering the construction of the general wild character varieties Boalch--Yamakawa \cite{boalch2020diagrams} were led to the 
definition of diagram for any connection on the Riemann sphere with at most one irregular pole (possibly twisted), 
and they showed that 
the quiver dimension formula holds.
Thus, for example, they completed the list of diagrams for the six Painlev\'e equations as follows (where the dashed line indicates a negative 
edge):

\begin{figure}[H]
\centering
\includegraphics[scale=0.7]{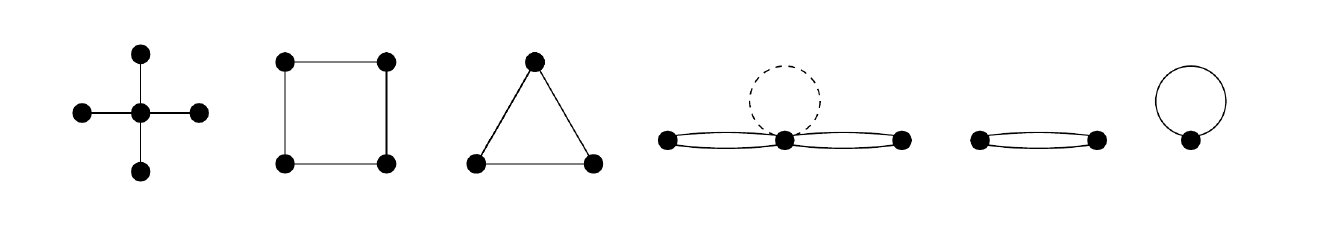}
\label{fig:list_diagrams_BY}
\caption{Diagrams for the Painlevé equations VI, V, IV, III, II, I.}
\end{figure}

Observe that the symmetrized Cartan matrix for 
the Painlev\'e III case is: 
$$C=\begin{pmatrix} 2 & -2 & 0 \\ -2 & 4 & -2 \\ 0 & -2 & 2 \end{pmatrix} = D\Ccal
\quad \text{where}\quad \Ccal=\begin{pmatrix} 2 & -2 & 0 \\ -1 & 2 & -1 \\ 0 & -2 & 2 \end{pmatrix}, D=\diag(1,2,1).$$
The corresponding (non-symmetric) Cartan matrix $\Ccal$ is 
denoted $D_3^{(2)}$ in Kac's notation, 
and is the transpose (Langland's dual) of $\widehat B_2$, so has the same affine Weyl group as that constructed by Okamoto.
The dimension vector in this case is $\mathbf d=(1,1,1)$ so 
$(\mathbf d,\mathbf d)=0$ and the Painlev\'e III moduli space indeed has dimension $2$.

However the Boalch--Yamakawa definition only works for
connections with at most one irregular pole
(and so they were only able to access the Painlev\'e III spaces 
via their alternate/non-standard representation with one twisted irregular pole, known as degenerate PV, rather than via the standard Lax representation with two poles of order two).
This leads to the present work, giving a direct definition that works in general.
In particular our general definition applied to the standard  Painlev\'e III linear system will give the same diagram 
as that of Boalch--Yamakawa above. 

More generally, our framework allows us to extend to new cases, including all Lax representations for 4-dimensional Painlevé-type equations listed in \cite{kawakami2018degeneration}, the observation that for most instances where several different representations are known for one given Painlevé-type equation, these representations correspond to different readings of the same diagram.
\\

To conclude this outline, let us briefly comment on related recent work by other authors. In \cite{hiroe2017linear}, Hiroe defines quivers associated to unramified irregular connections on with an arbitrary number of irregular singularities, and shows that the additive de Rham moduli spaces are isomorphic to open dense subsets of the corresponding quiver varieties, and uses this to solve the additive irregular Deligne-Simpson problem in this case. Our approach is many in respects different from his since he does not use the Fourier transform to define the quiver, and our diagrams are not the same as his (see remark \ref{rem:hiroe_diagrams}). Furthermore, particular cases of our diagrams have been studied by Xie \cite{xie20213d} in the context of 3d mirror symmetry for some Argyres-Douglas theories.

\subsection{The construction}

Here is a summary of our main definition 
(see the body of the article for full details).
Let $(E,\nabla)\to U$ be an algebraic connection on a vector bundle on a Zariski open subset $U\subset \mathbb P^1$ of the Riemann sphere $\mathbb P^1=\mathbb C\cup \infty$. 
The diagram $\Gamma(E,\nabla)=(N,B)$ of the connection $(E,\nabla)$
is defined by 
constructing a core diagram with nodes $N_c\subset N$
and then gluing a single leg  onto each core node 
 (a leg is a type $A$ Dynkin graph). 
As in \cite{crawley2006multiplicative,boalch2020diagrams} the legs are each determined by a 
conjugacy class in a general linear group, so the 
main part of the construction is to define the core diagram
(with nodes $N_c\subset N$), and a conjugacy class for each core node.

The set of core nodes $N_c$ is the disjoint union of three finite sets:
\begin{align*}
N_c=&\{\text{Stokes circles of $(E,\nabla)$ at $\infty$}\} \  \cup \\
&\{\text{wild Stokes circles of } (E,\nabla) \text{ not at $\infty$}\} \ \cup\\
&\{\text{tame Stokes circles of } (E,\nabla)\text{ not at $\infty$ with 
nontrivial formal monodromy}\}.
\end{align*}
To understand this recall that the formal 
structure of $(E,\nabla)$ at each point 
$a\in \mathbb P^1\setminus U$ is given by an irregular class, a finite 
multiset of Stokes circles 
$$\Theta_a=\sum n_i\cir{q_i}$$ at $a$, 
plus a (formal monodromy) conjugacy class
$\Ccal_i\subset \GL_{n_i}(\mathbb C)$ for each Stokes circle 
$\cir{q_i}$ in $\Theta_a$ (see \S\ref{sec:formal_data} for more details).

The matrix of edge multiplicities 
between the core nodes
is then defined as follows. 
For any Stokes circle $i=\cir{q}$, 
let $\alpha_i=\text{Irr}(i), \beta_i=\ram(i)\in \mathbb N$ 
denote its irregularity and 
ramification numbers (see \S\ref{sec:formal_data}),  
so that $i$ has slope $\alpha_i/\beta_i$.
For any pair of Stokes circles $i,j$ at the same point of $\mathbb P^1$, 
the irregular class $\Hom(i,j)$ is well defined (and of rank $\beta_i\beta_j$).
Then we define,
following \cite{boalch2020diagrams}, 
an integer $B^\infty_{ij}$ as follows:
\begin{align}
B^\infty_{ij}&=A_{ij}-\beta_i\beta_j,\qquad \text{if $i\neq j$} \label{eq: BY1}\\
B^\infty_{ii}&=A_{ii}-\beta_i^2+1,\quad \,\text{otherwise} \label{eq: BY2}
\end{align}
where $A_{ij}=\Irr(\Hom(i,j))$
is the  irregularity of the irregular 
class $\Hom(i,j)$.

Then our main definition is:

\begin{definition}
\label{def:def_diagram_introduction}
Suppose $i,j\in N_c$ are core Stokes circles, and write
$a_i=\pi(i),a_j=\pi(j)\in \mathbb P^1$ for the underlying 
points of the Riemann sphere.
Define $B_{ij}\in \mathbb Z$ as follows:
\begin{enumerate}
\item If $a_i=a_j=\infty$ then $B_{ij}=B^{\infty}_{ij}$,
\item If $a_i=a_j\neq \infty$ then
$B_{ij}=B_{ji}=B^{\infty}_{ij}-\alpha_i\beta_j-\alpha_j\beta_i$,
\item If $a_i=\infty$ and $a_j\neq \infty$
then $B_{ij}=\beta_i(\alpha_j+\beta_j)$,
\item If $a_i\neq \infty$ and $a_j =\infty$
then $B_{ij}=B_{ji}=\beta_j(\alpha_i+\beta_i)$,
\item If $a_i\neq \infty,\; a_j\neq \infty$ and
$a_i\neq a_j$ then $B_{ij}=B_{ji}=0$.
\end{enumerate}
\end{definition}

This completes the definition of the core diagram. 
Each node $i\in N_c$ is equipped with a conjugacy class $\breve\Ccal_i$ (and thus a leg, as in \S\ref{sec:def_diagram_infinity}) as follows:  
if $i$ is wild or at $\infty$ then 
$\breve\Ccal_i=\Ccal_i$ (the formal 
monodromy conjugacy class of $(E,\nabla)$ around 
the Stokes circle $i$).
If $i$ is tame and at finite distance then 
$\breve\Ccal_i$ is the {\em child}
of $\Ccal_i$ in the sense that 
if $1+ab\in \Ccal_i$ 
with $a,b$ linear maps, with $b$ surjective and $a$ injective, then 
$(1+ba)^{-1}\in \breve\Ccal_i$ 
(see e.g.  \cite{boalch2016global} Appx. for this terminology).

Taking the legs of these conjugacy classes for all core nodes, and gluing them to the core,  defines 
the full diagram (see \S\ref{sec:def_diagram_infinity} below).
If all the circles at finite distance are tame (so $\alpha_i=0$ for all those core nodes, and the
formulae (2), (3), (4) simplify) then this 
definition specialises to that of 
Boalch--Yamakawa \cite{boalch2020diagrams}.

We will prove directly that this does indeed define a diagram (i.e. $B_{ii}$ is an even integer) in \S \ref{section_invariance of the diagram}, independently of the proof suggested in \cite{boalch2020diagrams} for their
special case.

\subsection{Organisation of the paper}

Here is a brief outline of the structure of the article. In section 2 we recall some facts about formal data of meromorphic connections, which we formulate in terms of local systems on a large collection of circles corresponding to the exponential factors of the connection. There are some subtleties involving the relation between connections and modules over the Weyl algebra, which lead to the introduction of the modified formal data. In section 3, we discuss the action of $\SL_2(\C)$. We review the stationary phase formula which relates the formal data of a connection to those of its Fourier transform. In section 4, we study the diagram associated to a connection with just one singularity at infinity defined in \cite{boalch2020diagrams} and show that it is invariant under generic symplectic transformations. This is the crucial fact that will allow to get a well-defined diagram in the general case. To this purpose, we find an explicit formula for the number of edges and loops of the diagram which is of independent interest. We also determine the monomials appearing in the Legendre transform of any exponential factor.  
The definition of the diagram in the general setting is obtained in section 5, where we also discuss how the different readings of the diagram are generalized to our setting.
In section 6, we look at the dimension of the wild character variety, and show that it is given from the graph by the formula \eqref{intro_formula_dimension}. Finally, in section 7 we discuss several examples of diagrams, many of them related to Painlevé equations and show that in several new cases several known Lax representations for Painlevé-type equations correspond to different readings of the diagram. To facilitate the reading, a few somewhat lengthy proofs are relegated to appendices.

\subsection*{Acknowledgements} I thank my advisor Philip Boalch for his constant support and many useful discussions.

\section{Formal data of irregular connections}
\label{sec:formal_data}

In this section, we describe the data encoding the formal type of an irregular connection on the Riemann sphere at each of its singular points. Among the different existing approaches to the formal classification of meromorphic connections, the one we will use is the geometric description of \cite{boalch2015twisted} close to the point of view of Deligne and Malgrange \cite{deligne1978letter,malgrange1991equations}, in terms of graded local systems on the circles of directions around the singular points.
We refer the reader to \cite{boalch2015twisted} for more details.

\subsection{The exponential local system}

Let $\Sigma$ be an algebraic curve, and $a\in \Sigma$.  We consider the real blow-up $\pi_a : \widehat{\Sigma}_a\to\Sigma$ of $\Sigma$ at $a$, such that $\partial_a:=\pi^{-1}(a)$ is the circle of directions around $a$.

For $a\in \Sigma$, let $z_a$ be a local coordinate at $a$. We define a local system of sets (i.e. a cover) $\Ical$ on the circle $\partial_a$ in the following way: if the open subset $U\subset \partial_a$ is a germ of angular sector at $a$, the sections of $\Ical_a$ on $U$ are functions of the form
\[ q_U(z_a)=\sum_{i=1}^s b_i z_a^{-i/r},
\]
where $r,s\in \N$, $b_i\in \C$ for $i=1,\dots,s$, for some determination on $U$ on the $r$-th root $z^{1/r}$. It is also possible to give an intrinsic definition of $\Ical_a$ independent of the choice of a local coordinate, see \cite[Remark 3]{boalch2015twisted}.

As a topological space, the cover $\Ical_a$ is a disjoint union of circles. To  any polynomial $q\in z_a^{-1/r}\C[z_a^{-1/r}]$ in $z_a^{1/r}$ corresponds a connected component of $\Ical_a$, homeomorphic to a circle, that we will denote $\cir{q}_a$, or simply $\cir{q}$ when there will be no ambiguity about the point $a$ which is considered. The function $z_a\mapsto q(z_a)$ on the punctured disk $D^\times_a$ around $a$ is multivalued: if $r$ is the smallest possible denominator for the exponents appearing in $q$, the cover $\cir{q}\to \partial_a$ is of order $r$; each leaf of the cover above $U\subset \partial_a$ corresponding to a determination of the $r$-th root $z_a^{1/r}$. The  number $r=:\ram(q)$ is the \textit{ramification index} of $q$, its degree $s$ as a polynomial in $z_a^{-1/r}$ is its \textit{irregularity} $\Irr(q):=s$. The quotient $s/r$ is the \textit{slope} of $q$. The circle $\cir{0}$ is the \textit{tame circle}. 

If $I\in \pi_0(\Ical_a)$ is a connected component of $\Ical_a$, there is no unique $q$ such that $I=\cir{q}_a$: if $\omega$ is a $r$-th root of unity with $r=\ram(q)$, one has $\cir{q(z_a)}=\cir{q(\omega z_a)}$. The set of polynomials $q'$ such that $I=\cir{q'}_a$ is the Galois orbit of $q$ under the action of the group $\mu_r$ of $r$-th roots of unity. We thus have  
\[\Ical_a =\bigsqcup \cir{q}_a. \]
Let us fix a direction $p\in \partial_a$, and denote by $(\Ical_a)_p:=\pi_a^{-1}(p)$ the fibre of $\Ical_a$ at $p$. The monodromy of $\Ical_a$ is an automorphism $\rho_a : (\Ical_a)_p\to (\Ical_a)_p$. 

All this was local at a point $a\in\Sigma$. At the global level, we define the global exponential local system to be
\[ \Ical = \bigsqcup_{a\in \Sigma} \Ical_a.
\]
As a topological space, $\Ical$ is thus a large collection of disjoint circles.

\subsection{Graded local systems}

\begin{definition}
An $\Ical_a$-\textit{graded local system} is a local system $V^0\to \partial_a$ of finite-dimensional vector spaces endowed with a point-wise grading
\[ V^0_p=\bigoplus_{i\in (\Ical_a)_p} V^0_{p,i},
\]
where $V^0_p$ is the fibre of $V^0$ at $p$, for each point $p\in \partial_a$, such that the grading is locally constant.
\end{definition}

If we fix a base point $p\in \partial_a$, let $\rho\in \Aut((\Ical_a)_p)$ be the monodromy of $\Ical_a$ and $\widehat{\rho}\in \GL(V^0_p)$ the monodromy of $V$, this implies that 
\[ \widehat{\rho}(V^0_{p,i})=V^0_{p,\rho(i)}.
\]
The data of its monodromy determines the local system $V^0$ up to isomorphism. More precisely, the isomorphism class of an $\Ical_a$-graded local system $V^0$ corresponds to the conjugacy class of $\widehat{\rho}$ under graded automorphisms of $V^0_p$. 

Let $V^0\to \partial_a$ be an $\Ical_a$-graded local system. For $i\in (\Ical_a)_p$, let $n_i:=\dim V^0_{i,p}$.  If $I=\cir{q}$ is a circle in $\Ical_a$ with ramification order $r$, the fibre ${\cir{q}}_p$ of the restriction $\cir{q}\to \partial_a$ has $r$ elements. Let $V^0_{\cir{q}}\to \partial$ be the restriction of $V^0$ to its $\cir{q}$-graded part. Then all graded pieces of $V^0_{\cir{q}}\to \partial$ have the same dimension $n_{\cir{q}}$, that is $n_{\cir{q}}=n_i$ for all $i\in (V^0_{\cir{q}})_p$. The integer $n_{\cir{q}}$ is the \textit{multiplicity} of $\cir{q}$. 

\begin{definition}
A (local) irregular class at $a$ is a function $\pi_0(\Ical_a)\to \N$.
\end{definition} 

The \textit{irregular class} of $V^0$ is the function 
\begin{align*}
\Theta_a(V^0) : \pi_0(\Ical_a)&\to \N,\\
\cir{q}&\mapsto n_{\cir{q}}.
\end{align*}
The \textit{active circles} the circles $\cir{q}$ such that $n_{\cir{q}}\neq 0$. 

The classical Fabry-Hukuhara-Turritin-Levelt theorem on the formal classification of meromorphic connections is formulated in this language in the following way: 

\begin{theorem}[\cite{deligne1978letter,malgrange1982classification}]
The category of connections on the formal punctured disk is equivalent to the category of $\Ical_a$-graded local systems.
\end{theorem}

The basic idea behind this is that the active circles $\cir{q_i}$ of the $\Ical_a$-graded local system associated to the formal connection correspond to the exponential terms $e^{q_i}$ in the horizontal sections of the Hukuhara-Turritin normal form. \\

It is possible to view an $\Ical_a$-graded local system $V^0\to \partial$ as a local system on $\Ical_a$, such that the following diagram commutes:
\[
\begin{tikzcd}
V^0 \arrow[r] \arrow[rd] & \Ical \arrow [d,"\pi"]\\
 & \partial
\end{tikzcd}
\]
The isomorphism class of a local system $V^0\to \Ical_a$ with irregular class $\Theta_a$ is given by the collection $\bm{\Ccal}=(\Ccal_{\cir{q}})$, where $\Ccal_{\cir{q}}\subset \GL_{n(\cir{q})}(\C)$ of the conjugacy classes of its monodromies around its active circles.\\

In turn, global formal data are obtained by putting together all local formal data. Let $(E,\nabla)$ be an algebraic on a Zariski open subset of $\Sigma$. From the formal classification at each of its singularities, $(E,\nabla)$ defines a global formal local system $V^0\to \Ical$, with support on the singular points. Its isomorphism class is given by the pair $(\bm\Theta,\bm{\Ccal})$, with $\bm\Theta:\pi_0(\Ical)\to \N$ the global irregular class, and $\bm{\Ccal}$ the collection of the formal monodromies at each active circle. We will refer to the pair $(\bm\Theta,\bm\Ccal)$ as the  (global) \textit{formal data} of the connection $(E,\nabla)$.

\subsection{Modified formal data}

In the rest of article, we fix a choice of a genus 0 compact Riemann surface $\Sigma$, and a point $\infty\in \Sigma$. We will also at times fix a choice of isomorphism between $\Sigma$ and $\mathbb P^1=\C  \cup\{\infty\}$. To arrive at our construction of the diagrams, we will need to use the $\SL_2(\C)$ action and in particular the Fourier-Laplace transform. However, the category which the Fourier-Laplace transform acts on is not a category of connections on the affine line, but the category of modules on the Weyl algebra $A_1=\C[z]\cir{\partial_z}$. This is the source of an important subtlety, pertaining to the fact that slightly more formal data are needed to describe a $A_1$-module than a connection. We take this into account by introducing \textit{modified formal data}. In this paragraph, we first define modified formal data for connections on the affine line, then discuss (modified) formal data of $A_1$-modules.

Let us first define modified local formal data for connections. We consider the local situation at a singularity $a\in \Sigma$. Let $V^0_{\cir{0}_a}$ be a local system on the tame circle $\cir{0}_a$, with monodromy $T$. We associate to it another local system $\breve{V}^0_{\cir{0}_a}\to \cir{0}_a$ defined by $\breve{V}^0_{\cir{0}_a}:=\Im(T-\Id)$. Its rank is $\rank(T-\Id)$ and it has monodromy $T_{|\Im(T-\Id)}$. Notice that the data of $\breve{V}^{0}_{\cir{0}_a}$ together with the rank of $V^0_{\cir{0}_a}$ is enough to reconstruct the isomorphism class of $V^0_{\cir{0}_a}$. 

Now, let $V^0\to \Ical_a$ be a formal local system at $a$. We define the \textit{modified formal local system} $\breve{V}^0\to \Ical$ associated to $V^0$ to be the local system on $\Ical_a$ such that $\breve{V}^0_{\cir{0}_a}$ is obtained from $V^0_{\cir{0}_a}$ as described above, and $\breve{V}^0_{\cir{q}_a}=V^0_{\cir{q}_a}$ for any other connected component $\cir{q}_a$. 

Let us now turn to the global situation, for $\Sigma=\Pbb^1$. If $V^0\to\Ical_a$ is a global formal local system, applying at each point at finite distance the previous construction defines a modified global formal local system $\breve{V}^0\to \Ical$.

\begin{definition}
Let $(E,\nabla)$ an algebraic connection on a Zariski open subset of $\Sigma$, $V^0\to \Ical$ its formal local system and $(\bm\Theta,\bm{\Ccal})$ its (global) formal data. Its global modified formal data $(\bm{\breve\Theta},\bm{\breve\Ccal})$ correspond to the isomorphism class of the global modified local system $\breve{V}^0\to \Ical$ obtained by taking the (local) modified formal local system at each point $a\neq\infty$ at finite distance, and keeping the non-modified formal local system $V^0\to \Ical_\infty$ at infinity.
\end{definition}

Notice that $\breve{V}^0$ satisfies the condition $\rank(\breve{V}^0_a)\leq \rank{\breve{V}^0_\infty}$ since for each $a\neq \infty$ passing to the modifed formal local system lowers the rank. If $\breve{V}\to \Ical$ is a formal local system, with isomorphism class $(\bm{\breve\Theta},\bm{\breve\Ccal})$, we say that $V$, or $(\bm{\breve\Theta},\bm{\breve\Ccal})$ is \textit{compatible} if it satisfies this condition. If there exists a connection $(E,\nabla)$ on a Zariski open subset of the affine line such that $(\bm{\breve\Theta},\bm{\breve\Ccal})$ is its modified formal data, we say that $(\bm{\breve\Theta},\bm{\breve\Ccal})$ is \textit{effective}. In particular $(\bm{\breve\Theta},\bm{\breve\Ccal})$ cannot be effective if it is not compatible.

If we know the modified formal local system $\breve{V}^0$ of a connection $(E,\nabla)$, since we keep the non-modified local system $V^0_\infty\to\Ical_\infty$ at infinity, we know the rank of $(E,\nabla)$. It is given by
\[ \rank(E,\nabla)=\sum_{\cir{q}\in \pi_0(\Ical_\infty)} n_{\cir{q}}\ram(q).
\]

In turn, this enables us to reconstruct from $\breve{V}^0\to\Ical$ the isomorphism class of the non-modifed local system $V^0\to \Ical$. The data of $(\bm{\breve\Theta},\bm{\breve\Ccal})$ is thus equivalent to the data of $(\bm\Theta,\bm{\Ccal})$.\\

Let us now briefly discuss (modified) formal data of $A_1$-modules, following \cite[chapter IV]{malgrange1991equations}. If $M$ is an $A_1$-module, its formal data consist of:
\begin{itemize}
\item At infinity, a formal local system $V^0_\infty\to \Ical_\infty$.

\item For any point at finite distance $a\neq \infty$:
\begin{itemize}
\item A formal local system ${V^0_{\cir{q}}}\to \cir{q}_a$ over each irregular circle at $a$ (with only a finite number of circles having nonzero multiplicity).
 \item A quadruple $(V^0_{\cir{0}_a}, \breve V^0_{\cir{0}_a}, u_a,v_a)$, where $V^0_{\cir{0}_a}$ and $\breve V^0_{\cir{0}_a}$ are local systems over the tame circle $\cir{0}_a$, and $u_a:V^0_{\cir{0}_a}\to \breve V^0_{\cir{0}_a}$, $v_a:\breve V^0_{\cir{0}_a}\to V^0_{\cir{0}_a}$ are linear maps, such that $1+v_au_a$ and $1+u_av_a$ are respectively the monodromies of $V^0_{\cir{0}_a}$ and $\breve V^0_{\cir{0}_a}$.
\end{itemize} 
\end{itemize}

Now, we define the modified formal data of $M$ by forgetting $V^0_{\cir{0}_a}$, $u_a$, $v_a$, for all $a\neq \infty$, more precisely in a similar way as above we define a modified formal local system $\breve V^0\to \Ical$ by setting, for any circle $\cir{q}\in \Ical$, $\breve V^0_{\cir{q}}:=\breve V^0_{\cir{q}}$ if $\cir{q}$ is not a tame circle $\cir{0}_a$ with $a\neq \infty$ and taking $\breve V^0_{\cir{0}_a}$ otherwise. This yields (global) modified formal data $(\bm{\breve\Theta}, \bm{\breve\Ccal})$ attached to $M$.\\

The link between the notions of modified formal data for connections and $A_1$-modules is as follows. If $(E,\nabla)$ is an algebraic connection on on a Zariski open subset $U=\A^1\smallsetminus \{\mathbf{a}\}$ of the affine line, the \textit{minimal extension} of $(E,\nabla)$ is an $A_1$-module $M$, such that its modified formal data in the sense of $A_1$-modules coincide with those of $(E,\nabla)$ in the sense of connections (see \cite[§5.5]{arinkin2008fourier}). 

\section{Action of $\SL_2(\C)$}

We now discuss the action of the Fourier-Laplace transform on modules on the Weyl algebra, and review the stationary phase formula relating the formal data of an $A_1$-module and its Fourier transform. 

\subsection{The stationary phase formula}

The Fourier transform is the automorphism of the Weyl algebra $A_1=\C[z]\cir{\partial_z}$ defined by $z\mapsto -\partial_z, \; \partial_z\mapsto z$. It induces a transformation on modules on the Weyl algebra: the Fourier-Laplace transform $F\cdot M$ of a  $\C[z]\langle \partial_z\rangle$-module $M$ is the $\C[\xi]\langle \partial_\xi\rangle$-module obtained by setting: $z\mapsto -\partial_\xi, \; \partial_z\mapsto \xi$ (we use a dual variable $\xi$ for the image in this paragraph).  

The stationary phase formula \cite{fang2009calculation,sabbah2008explicit} states that the modified formal data of the Fourier transform $M'=F\cdot M$ are determined by the modified formal data of $M$. 

\begin{theorem}
\label{th_stationary_phase} There exists a bijection, that we also denote by $F$, from the set of all modified formal data to itself, such that if $M$ is a module on the Weyl algebra, $(\bm{\breve\Theta},\bm{\breve\Ccal})$ its modified formal data, $M'$ is the Fourier transform of $M$ and $(\bm{\breve\Theta'},\bm{\breve \Ccal}')$ its modified formal data, the following diagram commutes:
\[
\begin{tikzcd}
M \arrow[r,"F"] \arrow[d] & M' \arrow[d]\\
(\bm{\breve\Theta},\bm{\breve\Ccal}) \arrow[r,"F"] & (\bm{\breve\Theta'},\bm{\breve \Ccal}')
\end{tikzcd}
\]
\end{theorem}

This map $F$ is the formal Fourier transform. More specifically, the active circles of $\bm{\breve{\Theta}'}$ are related to the active circles of $\bm{\breve{\Theta}}$ by a Legendre transform.

\begin{remark}
\label{rem: Fourier_module_connections}
Under some conditions, the Fourier-Laplace transform induces a transformation on connections \cite[§2.2]{arinkin2010rigid}. Precisely, let $(E,\nabla)$ be an \textit{irreducible} algebraic connection on a Zariski open subset $\Sigma^\circ\subset \C$ let $M$ be its minimal extension, and $M'=F\cdot M$. Then, if we assume that we are not in the case when it has rank one and its only singularity is a second order pole at infinity, there exists an irreducible connection $(E',\nabla')$ on a Zariski open subset ${\Sigma^\circ}'\subset \C$ such that $M'$ is the minimal extension of $(E,\nabla)$. In that case we may call it as in \cite[§2.2]{arinkin2010rigid} the Fourier transform of $(E,\nabla)$. In particular, this implies that if $(\bm{\breve\Theta},\bm{\breve\Ccal})$ is the set of formal data of such an irreducible connection, then this is also the case of $F\cdot(\bm{\breve\Theta'},\bm{\breve\Ccal'})$. 
\end{remark}

Let us describe the formal Fourier transform and express it in our framework of local systems on $\Ical$. To this end, we formulate the Legendre transform as a homeomorphism sending the collection of circles $\Ical$ to the dual collection $\Ical'$, obtained by replacing the variable $z$ by the variable $\xi$. The Legendre transform takes slightly different forms depending on the circles in $\pi_0(\Ical)$. It is therefore necessary to distinguish several types of circles as follows:
\begin{definition}
Any circle in $\Ical$ belongs to one of the following five families:
\begin{enumerate}
\item The \textit{pure circles} at infinity, of the form $\cir{\alpha z}_\infty$, with $\alpha\in\C$. There are of slope $1$ if $\alpha\neq 0$, and 0 otherwise. 
\item Other circles of slope $\leq 1$ at infinity, of the form $\cir{\alpha z +q}_\infty$, with $\alpha\in\C$, and $q\neq 0$ of slope $<1$, 
\item Circles $\cir{q}_\infty$ of slope $>1$ at infinity,
\item Irregular circles at finite distance $\cir{q}_a$, with $q\neq 0$, $a\in\C=\Pbb^1\smallsetminus\{\infty\}$.
\item The tame circles $\cir{0}_a$, $a\in\C$.
\end{enumerate}
\end{definition}
We will denote by $\Ical_1,\dots,\Ical_5\subset \Ical$ the corresponding collections of circles. In a similar way, we denote by $\Ical'_1,\dots,\Ical'_5\subset \Ical'$ the dual collections. The Legendre transform $\Lcal$ yields homeomorphisms 
\begin{align*}
\Lcal &:\Ical_1\to \Ical_5',\\
\Lcal &:\Ical_2\to \Ical_4'\\
\Lcal &:\Ical_3\to \Ical_3',\\
\Lcal &:\Ical_4\to \Ical_2',\\
\Lcal &:\Ical_5\to \Ical_1',
\end{align*}
that we will all denote by $\Lcal$ with a slight abuse of notation. 

Let us first describe the Legendre transform in some detail for circles of type 4 at $z=0$, that is irregular circles at $0$. The other cases are easily deduced from this one by changes of variable.
In this paragraph, we set $\partial:=\partial_0$ the circle of directions at $0\in\Pbb^1_z$, $\Ical:=\Ical_0$ the corresponding exponential local system, $\partial'$ the circle of directions at $\xi=\infty$ in $\Pbb^1_{\xi}$, and $\Ical'\to\partial'$ the corresponding exponential local system. 

\begin{definition}
Let $U\subset\partial$ be a germ of sector, and $q_U\in \Ical(U)$ that is non-zero. For $z\in U$, we set 
\[ \phi(z):=\frac{dq}{dz}(z).
\]
Choosing $U$ small enough, $\phi$ induces a biholomorphism between $U$ and a sector $U'\subset \partial'$. For $\xi\in U'$, we set
\[ \widetilde{q}_{U'}(\xi):=q_U(\phi^{-1}(\xi))-\xi \phi^{-1}(\xi).
\]
This yields a section $\widetilde{q}_{U'}\in \Ical'(U')$, called the Legendre transform of $q_U$. 
\end{definition}

The Legendre transform yields an homeomorphism between $\Ical$ and $\Ical'$, the inverse homemorphism being given by $q_U(z)=\widetilde{q}_{U'}(\phi(z))+\phi(z)z$. It follows from a straightforward computation that the image of an irregular circle $I=\cir{q}$ by the Legendre transform is a circle $I'=\cir{\widetilde{q}}$ with ramification order $r+s$ and slope $s/(r+s)$. 

The general case of circles of type 2,3,4 is similar to the case of a circle of type 4 at $z=0$. Let us summarize the properties of the Legendre transform for the different types of circles.

\begin{itemize}

\item Type 4: if $a\in \Pbb^1\setminus \{\infty\}$ is a point at finite distance, the Legendre transform $\Lcal$ yields an homeomorphism between a circle of type 4 of the form $\langle q\rangle_a$ with slope $\alpha/\beta\neq 0$ at $a$ and the circle of type 2 $\langle -a\xi+\widetilde{q} \rangle_{\infty}$ with $\widetilde{q}$ of slope $\alpha/(\alpha+\beta)<1$, once again determined by the same system of equations as before. The circle $\langle -a\xi+\widetilde{q} \rangle_{\infty}$ thus has slope $1$ if $a\neq 0$, and slope $<1$ if $a=0$.
\[
\begin{tikzcd}
\cir{q}_0 \arrow[r,"\Lcal","1:1"'] \arrow[d,"\pi","\beta:1"'] & \cir{\widetilde{q}}_\infty \arrow[d,"\pi", "\alpha+\beta :1" ']\\
\partial_0  & \partial'_\infty
\end{tikzcd}
\] 
This is obtained in a straightforward way from the special case $a=0$ that we just detailed, by using the relation $z_a=z-a$ between the local coordinate $z_a$ at $a$ and the coordinate $z$ on $\Pbb^1$. The factor $\widetilde{q}$ is given by the same system of equations as previously, with $z$ replaced by $z_a$.

\item Type 2: conversely, if $\cir{-ax+q}_\infty$, with $a\in \C$ and $q_1$ of slope $\alpha/\beta<1$ is a circle of type 2, the Legendre transform induces an homeomorphism between $\cir{q}_\infty$ and a circle of type 4 $\cir{\widetilde{q}}_{a}$ at the point $a\in \Pbb^1\setminus \{\infty\}=\C$, with slope $\alpha/(\beta-\alpha)$. We thus have
\[
\begin{tikzcd}
\cir{q}_\infty \arrow[r,"\Lcal","1:1"'] \arrow[d,"\pi","\beta:1"'] & \cir{\widetilde{q}}_\alpha \arrow[d,"\pi", "\beta-\alpha :1" ']\\
\partial_0  & \partial'_\infty
\end{tikzcd}
\]

\item Type 3: If $\langle q\rangle$ is a circle at infinity of type type 3, of slope $\alpha/\beta>1$, $\Lcal$ induces a homeomorphism of $\langle q\rangle$ on the circle $\langle \widetilde{q}\rangle$ of slope
$\alpha/\alpha-\beta>1$, of type 3. The situation is the following: 
\[
\begin{tikzcd}
\cir{q}_\infty \arrow[r,"\Lcal","1:1"'] \arrow[d,"\pi","\beta:1"'] & \cir{\widetilde{q}}_\infty \arrow[d,"\pi", "\alpha-\beta :1" ']\\
\partial_0  & \partial'_\infty
\end{tikzcd}
\]
\end{itemize}

We now deal with circles of type 1 and 5. Let $a\in\Pbb^{1}\setminus\{\infty\}$, and consider the circles $\cir{0}_a$ and $\cir{-a\xi}_\infty$. We define an homeomorphism $\phi$ between the circles of directions $\partial_a$ and $\partial_\infty$ as follows, for any argument $\theta\in \partial_a$, we draw the half-line $l_{\theta}$ starting  from $a$ with direction $\theta$, and let $\phi(\theta)\in \partial_\infty$ be the direction at which $l_\theta$ approaches $z=\infty$. Notice that this inverses the sense of rotation, in agreement with Malgrange \cite[e.g. p. 98]{malgrange1991equations}. Replacing $z$ by the dual coordinate $\xi$ to identify $\partial_\infty$ to $\partial'_\infty$, we get an homeomorphism $\phi': \partial_a\to \partial'_\infty$ which lifts to an homeomorphism $\Lcal : \cir{0}_a \to \cir{-a\xi}_\infty$, as pictured on the diagram

\[
\begin{tikzcd}
\cir{0}_a \arrow[r,"\Lcal","\cong" '] \arrow[d,"\pi"] & \cir{-a\xi}_\infty \arrow[d,"\pi"]\\
\partial_0 \arrow[r,"\phi'"']  & \partial'_\infty
\end{tikzcd}
\]

The Legendre transform induces a transformation of the local systems on $\Ical$: if $V\to\Ical$ is a local system, we define $\widetilde{V}:=\Lcal_*(V)$ such that the following diagram commutes:
\[
\begin{tikzcd}
V \arrow[r,"\Lcal_*"] \arrow[d] & \widetilde{V} \arrow[d]\\
\Ical \arrow[r,"\Lcal"] & \Ical'
\end{tikzcd}
\]

The previous construction accounts for what happens to the exponential factors in the Laplace method, but the square root term in the Gaussian integral remains to be taken into account. A way to do this is to define for any circle $\cir{q}$ of type 2,3,4 a local system $W_{\cir{\widetilde{q}}}\to \cir{\widetilde{q}}$ of one-dimensional vector spaces. In the case where $\cir{q}$ is a circle of type 2 at zero, the definition can formulated as follows.

\begin{definition}
Let $\cir{q}_0$ be an irregular circle at $z=0$, $\cir{\widetilde{q}}_\infty$ its image by the Legendre transform. Let $r=\ram(q)$ and $s=\Irr(q)$. Let $\zeta$ a variable such that $\zeta^{r+s}=1/\xi$ so that the Legendre transform $\widetilde{q}$ is a polynomial in $\zeta^{-1}$, and the circle $\cir{\widetilde{q}}$ can be parametrized by the directions in the disk $D_\zeta$. We define $W_{\cir{\widetilde{q}}}\to \cir{\widetilde{q}}$ to be the rank one local system with étalé space $W_{\cir{\widetilde{q}}}$ is  $\cir{\widetilde{q}}\times \C$, and a flat section of which given by any determination of the germ of the function $\zeta\mapsto (\partial^2 q/\partial z^2)^{-1/2} (\zeta)$ on the disk $D_\zeta$.
\end{definition}

\begin{lemma}
If $q$ has slope $\alpha/\beta$, with $\beta=\ram(q)$, then the monodromy of $W_{\cir{\widetilde{q}}}$ is $(-1)^\alpha$. 
\end{lemma}

\begin{proof}
One has $q(z)\sim z^{-\alpha/\beta}$, hence $\partial^2 q/\partial z^2\sim z^{-(\alpha+2\beta)/\beta}$. In terms of the coordinate $\zeta$, the Legendre transform implies $\zeta\sim z^\beta$, hence $\partial^2 q/\partial z^2\sim \zeta^{\alpha+2\beta}$, hence $(\partial^2 q/\partial z^2)^{-1/2}\sim \zeta^{-(\alpha+2\beta)/2}$. Therefore, when going around $\zeta=0$ in $D_\zeta$,  $(\partial^2 q/\partial z^2)^{-1/2}$ gets multiplied by $\exp(-2i\pi\times \frac{\alpha+2\beta}{2})=(-1)^\alpha$. Notice this does not depend on the choice of determination of the square root $(\partial^2 q/\partial z^2)$, so that $W_{\cir{\widetilde{q}}}$ is well defined up to isomorphism. 
\end{proof}

The definition of $W_{\cir{q}}$ in the general case is similar, and the lemma remains true. 

\begin{definition}
Let $\breve{V}$ be a local system on $\Ical$, with isomorphism class $(\bm{\breve\Theta},\bm{\breve\Ccal})$. We set $F \cdot \breve V:=\Lcal_*(\breve V)\otimes W$. It is a local system on $\Ical'$. Let $(\bm{\breve\Theta'},\bm{\breve\Ccal'})$ be its isomorphism class. We define $F\cdot (\bm{\breve\Theta},\bm{\breve\Ccal}):=(\bm{\breve\Theta'},\bm{\breve\Ccal'}).$
\end{definition}

The operation $F$ on modified formal local systems on $\Ical$ is the counterpart at the level of formal data of the Fourier-Laplace transform of $A_1$-modules, that is it is such that the theorem \ref{th_stationary_phase} is true.

\begin{proof}
It seems that this result was first understood by Malgrange. The case of irregular circles has been independently proven by Fang \cite{fang2009calculation} and Sabbah \cite{sabbah2008explicit}. An alternative proof was also given later by Graham-Squire \cite{graham2013calculation}. To see that our formulation is indeed equivalent to the one in \cite{sabbah2008explicit}, it suffices to take the monodromy of the local systems. The case of regular circles can be extracted from Malgrange \cite[p.129]{malgrange1991equations}: the Fourier transform exchanges the local system of regular microsolutions at finite distance and the local system of regular solutions at infinity.
\end{proof}

\begin{remark}
If $(\bm{\breve\Theta},\bm{\breve\Ccal})$ is effective, the remark \ref{rem: Fourier_module_connections} guarantees that $F\cdot (\bm{\breve\Theta},\bm{\breve\Ccal})$ is also effective.
\end{remark}

If $F\cdot (\bm{\breve\Theta},\bm{\breve\Ccal})=(\bm{\breve\Theta'},\bm{\breve\Ccal'})$, the irregular class $\bm{\breve\Theta'}$ only depends on $\bm{\breve\Theta}$, so we set $F\cdot \bm{\breve\Theta}:=\bm{\breve\Theta'}$. We will also set $F\cdot\cir{q}:=\Lcal (\cir{q})\in \pi_0(\Ical)$ for any circle $\cir{q}$. Moreover, $\cir{q}$ is an active circle of $\bm{\breve{\Theta}}$ if and only if $F\cir{q}$ is an active circle of $\bm{\breve{\Theta}'}$. Explicitly, if $\bm{\breve{\Theta}}=n_1\cir{q_1}+ \dots n_r \cir{q_r}$, we have 
\[
\bm{\breve{\Theta}'}=F\cdot \bm{\breve{\Theta}}=n_1 F\cdot \cir{q_1}+\dots +n_r F\cdot \cir{q_r}.
\]

From the stationary phase formula we obtain a formula for the rank of the Fourier-Laplace transform of an irreducible connection.

\begin{lemma}
Let $(E,\nabla)$ an irreducible connection on a Zariski open subset of the affine line with modified formal data $(\bm{\breve\Theta},\bm{\breve\Ccal})$. The rank of $F\cdot (E,\nabla)$ is given by
\begin{equation}
\sum_k \left( m_k + \sum_{i=1}^{r_k} n^{(k)}_i (\alpha^{(k)}_i+\beta^{(k)}_i) \right) + \sum_{i,slope >1} n^{(\infty)}_i (\alpha^{(\infty)}_i - \beta^{(\infty)}_i),
\end{equation}
where the second sum is on all active circles $\cir{q^{(\infty)}_i}_\infty$ at infinity with slope $\alpha^{(\infty)}_i/\beta^{(\infty)}_i>1$. 
\end{lemma}

\begin{proof}
From the stationary phase formula the active circles of $V'$ at infinity are the following: 
\begin{itemize}
\item Images by Legendre transform of irregular circles $\cir{q^{(k)_i}}_{a_k}$, $k=1,\dots,m$ i.e circles $\cir{-a_k\xi +\widetilde{q}^{(k)_j}}_\infty$, with ramification $ \alpha^{(k)}_i+\beta^{(k)}_i$ and multiplicity $n^{(k)}_i$.
\item Images of the tame circles $\cir{0}_{a_k}$, $k=1,\dots,m$: pure circles $\cir{-a_k \xi}_\infty$, with multiplicity $m_k$.
\item Images of circles of slope $>1$ at infinity among the $\cir{q^{(\infty)}_i}, i=1,\dots, r_\infty$, of the form $\cir{\widetilde{q}^{(\infty)}_i}$, with slope $\alpha^{(\infty)}_i - \beta^{(\infty)}_i$ and multiplicity $n^{(\infty)}_i$. 
\end{itemize}
Adding the contribution of those circles to the rank gives the result.
\end{proof}

\begin{remark}
Notice that this is consistent with the formula given by Malgrange \cite[p. 79]{malgrange1991equations} for the rank of the Laplace transform of modules over the Weyl algebra.
\end{remark}

\subsection{Symplectic transformations}

The Fourier-Laplace transform is part of a larger group of transformations acting on modules over the Weyl algebra. Indeed, to any matrix
\[ A=\begin{pmatrix}
a & b\\
c & d
\end{pmatrix}
\]
in $\SL_2(\C)$, we can associate an automorphism of the Weyl algebra $A_1=\C[z]\cir{\partial_z}$ given by $z\mapsto az+b\partial_z, \partial_z \mapsto cz+d\partial_z$. This induces an action of the group $\SL_2(\C)$ of symplectic transformations on modules over the Weyl algebra (see \cite{malgrange1991equations}). 

The group of symplectic transformations is generated by three types of elementary transformations:
\begin{itemize}
\item The Fourier-Laplace transform $F$, corresponding to the matrix $\begin{pmatrix}
0 & 1\\
-1 & 0
\end{pmatrix}.
$
\item Twists at infinity $T_{\lambda}$, for $\lambda\in \C$, corresponding to the matrix $\begin{pmatrix}
1 & \lambda\\
0 & 1
\end{pmatrix}.
$
\item Scalings $S_\lambda$, for $\lambda\in \C^*$, corresponding to the matrix $\begin{pmatrix}
\lambda^{-1} & 0\\
0 & \lambda
\end{pmatrix}.
$
\end{itemize}

The geometric interpretation of twists and scalings is the following. The twist $T_\lambda$  corresponds to taking the tensor product with the rank one module $(\C[z],\partial_z+\lambda z)$, and the scaling $S_\lambda$ corresponds to do the change of variable $z\mapsto z/\lambda$ on $\Pbb^1$. 

As for the Fourier transform, any element of $\SL_2(\C)$ induces a transformation on modules over the Weyl algebra, and on irreducible connections on Zariski open subsets of the affine line which are not rank one connections with only a singularity at infinity of order less than two. Moreover, there exists again a corresponding formal transformation (that we will also denote by $A$) on modified formal data, such that the diagram
\[
\begin{tikzcd}
M \arrow[r,"A"] \arrow[d] & M' \arrow[d]\\
(\bm{\breve\Theta},\bm{\breve\Ccal}) \arrow[r,"A"] & (\bm{\breve\Theta'},\bm{\breve\Ccal}')
\end{tikzcd}
\]
commutes. To show this, it is enough to check this for elementary transformations. We have already dealt with the case of the Fourier transform, and for the twists and scalings we have:

\begin{proposition} Let $M$ a module over the Weyl algebra, $\breve V\to \Ical$ its modified formal local system, and $(\bm{\breve\Theta},\bm{\breve\Ccal})$ its isomorphism class.
\begin{itemize}
\item Let $t_\lambda : \Ical_\infty\to \Ical_\infty$ be the homeomorphism of $\Ical_{\infty}$ defined by $q_U(1/z)\mapsto q_U(1/z)+\frac{\lambda}{2}z^2$, where $q_U$ is a section of $\Ical_\infty$ over an open sector $U\subset \partial_\infty$. We extend $t_\lambda$ to an homeomorphism of $\Ical$ by having it act trivially on $\bigsqcup_{a\in \C}\Ical_a$. Then the modified formal local system associated to $T_\lambda \cdot (E,\nabla)$ is isomorphic to $(t_\lambda)_* \breve V$. 
\item For $\lambda\in \C^*$, let $s_\lambda:\Ical \to \Ical$ be the homeomorphism of $\Ical$ defined in the following way: for $a\in \C$, if $q_{U_a}(z-a)$ is a section of $\Ical_a$ on the open sector $U\subset \partial_a$, its image by $s_\lambda$ is the section $q_{U_a/\lambda}(\lambda(z-a/\lambda))$ over the sector $U_a/\lambda\subset \partial_{a/\lambda}$ which is the image of $U$ by $z\mapsto z/\lambda$. If $q_U(1/z)$ is a section of $\Ical_{\infty}$, its image is $q_{U/\lambda}(\frac{1}{(\lambda z)})$. The modified formal local system associated to $S_\lambda \cdot (E,\nabla)$ is isomorphic to $(s_\lambda)_*(\breve V)$.
\end{itemize}
\end{proposition}

We thus define $T_\lambda\cdot (\bm{\breve\Theta},\bm{\breve\Ccal})$ as the isomorphism class of $(t_\lambda)_*(\breve V)$ for $\lambda\in \C$, and $S_\lambda\cdot (\bm{\breve\Theta},\bm{\breve\Ccal})$ as the isomorphism class of $(s_\lambda)_*(\breve V)$ for $\lambda\in \C^*$.

\begin{proof} The twist $T_{\lambda}$ consists of tensoring $M$ with the rank one connection $(\mathcal{O},d+\lambda z dz)$, having a third order pole at infinity and no other singularity. Therefore, $T_{\lambda}$ only modifies the formal data of $M$ at infinity. Since $\lambda z dz=-\lambda \frac{dz_\infty}{z_\infty^3}$, if $A$ is the matrix of the formalization at infinity $\widehat{M}_\infty$ in some basis, i.e. $\widehat{M}_\infty$ corresponds $\partial_{z_\infty}+A$ in this basis, then the matrix of the formalization of $T_\lambda\cdot M$ at infinity in the same basis is $A-\frac{\lambda}{z_\infty^3}$. Integrating this, we obtain that the exponential factors at infinity of $T_\lambda \cdot M$ are obtained from the exponential factors of $M$ by adding $\frac{\lambda}{2 z_\infty^2}$. The formal monodromies do not change.

The scaling $S_\lambda$ corresponds to making the change of variable $z'=\lambda z$ The singularities of $S_\lambda\cdot M$ are thus the images of the ones of $M$ by $z\mapsto z/\lambda$, and the exponential factors of $S_\lambda\cdot M $ are obtained from the ones of $M$ by expressing them as a function of $z'$, i.e. by substituting $z$ with $\lambda z'$.
\end{proof}

As for the Fourier transform, if $A\cdot (\bm{\breve\Theta},\bm{\breve\Ccal})=(\bm{\breve\Theta'},\bm{\breve\Ccal'})$ the irregular class $\bm{\breve\Theta'}$ only depends on $\bm{\breve\Theta}$, so we set $A\cdot \bm{\breve\Theta}=\bm{\breve\Theta'}$. For any circle $\cir{q}$, there is also a well-defined circle $A\cdot\cir{q}\in \pi_0(\Ical)$ such that $\cir{q}$ is an active circle of $\bm{\breve{\Theta}}$ if and only if $A\cdot\cir{q}$ is an active circle of $\bm{\breve{\Theta}'}$. If $\bm{\breve{\Theta}}=n_1\cir{q_1}+ \dots n_r \cir{q_r}$, we have 
\[
\bm{\breve{\Theta}'}=A\cdot \bm{\breve{\Theta}}=n_1 A\cdot \cir{q_1}+\dots +n_r A\cdot \cir{q_r}.
\]
The circle $A\cdot\cir{q}$ can be explicitly determined in practice by factorizing $A$ as a product of elementary operations.

\section{Invariance of the diagram for one irregular singularity}
\label{section_invariance of the diagram}

In this section, we discuss the properties of the diagram associated by \cite{boalch2020diagrams} to a connection of the Riemann sphere with only one irregular singularity at infinity. We show that the diagram is invariant under Fourier-Laplace transform, and under $\SL_2(\C)$ transformations. To do this, we compute an explicit formula for the number of edges and loops of the diagram, as well as a formula for the terms appearing in the Legendre transform of an exponential factor.

\subsection{Diagram for one irregular singularity}
\label{sec:def_diagram_infinity}

Let $(E,\nabla)$ be an irreducible algebraic connection on the affine line. The only singularity is at infinity. Its formal local system $V^0\to \Ical$ (which is also the modified formal local system) has support on $\Ical_{\infty}$. Let $(\Theta_\infty,\bm{\Ccal_\infty})$ denote its formal data. Let us fix a direction $d\in \partial_\infty$. We set $G:=\GL(V_d)$ and $H:=\GrAut(V^0_d)$ the set of graded automorphisms of $V_d$. Recall \cite{boalch2015twisted} that the conjugacy class $\Ccal_\infty$ of the monodromy of $V^0\to \partial_\infty$ is an element of a twist $H(\partial)$ of $H$. The wild character variety associated to $(E,\nabla)$ is in this case a multiplicative symplectic quotient
\begin{equation}
\mathcal{M}_B(E,\nabla)=\mathcal{M}_B(\Theta_\infty,\bm{\Ccal}_\infty)=\Hom_\mathbb{S}(V^0)\sslash_{\mathcal{C}}H, 
\end{equation}
where 
\begin{equation}
\Hom_\mathbb{S}(\Theta_\infty)=\mathcal{A}(V^0)\sslash G,
\end{equation}
is a twisted quasi-Hamiltonian $H$-space, itself obtained by symplectic reduction with respect to $G$ from the twisted quasi-Hamiltonian $G\times H$-space 
\begin{equation}
\mathcal{A}(V^0)=H(\partial)\times G \times \prod_{d\in \mathbb{A}}\Sto_d(V^0).
\end{equation}

We denote by $\cir{q_1}, \dots,\cir{q_r}$ the active circles, and $\beta_i:=\ram(q_i)$. 

Let us recall the definition of the diagram $\Gamma^\infty(E,\nabla)$ associated in \cite{boalch2020diagrams} to any such connection $(E,\nabla)$. The diagram has the following structure: it consists of a core diagram $\Gamma^\infty_c(E,\nabla)=\Gamma_c(\Theta_\infty)$, which only depends on the irregular class $\Theta_\infty$, to which are then glued  \textit{legs} encoding the conjucacy classes $\bm \Ccal_\infty$.

\begin{definition} The \textit{core diagram} $\Gamma_c^\infty(\Theta_\infty)$ associated to the irregular class $\Theta_\infty$ has a set of nodes $N_c$ given by the set of active circles $\cir{q_1},\dots,\cir{q_r}$, and for $i,j=1,\dots,r$, the number of arrows $B_{ij}$ between $\cir{q_i}$ and $\cir{q_j}$ is given by 
\begin{itemize}
\item If $i\neq j$ then 
\begin{equation}
B_{ij}=A_{ij}-\beta_i\beta_j,
\end{equation}
where $A_{ij}=\Irr(\Hom(\cir{q_i},\cir{q_j})$.
\item If $i=j$, the number of oriented loops at $\cir{q_i}$ is given by
\begin{equation}
B_{ii}=A_{ii}-\beta_i^2+1.
\end{equation}
\end{itemize}
\end{definition}
The numbers of edges/loops may be negative. One has $B_{ij}=B_{ji}$ and we will show that $B_{ii}$ is always an even number, so we can group the arrows two by two to get an unoriented diagram. 
The motivation for this definition comes from counting the number of appearances of blocks between graded parts of $V^0$ associated to the  different active circles in the explicit presentation of the wild character variety. The positive terms in $B_{ij}$ correspond to the matrix blocks in $\mathcal{A}(V^0)$, and the negative terms in the $B_{ij}$ correspond to the relations given by the quasi-Hamiltonian reduction. 

The core diagram only depends on the irregular class $\Theta_\infty$, but it does not take into account the monodromies of the active circles. This can be done  by adding legs to the core diagram encoding the conjugacy classes of the monodromies.  This involves choosing a minimal \textit{marking} $\bm\xi$ of $\bm{\Ccal_\infty}$.

\begin{definition}
\label{def:marking}
If $\mathcal C\in \GL_n(\mathbb C)$  is a conjugacy class, a\textit{ marking} of $\Ccal$ is an ordered let $\xi=(\xi_1,\dots, \xi_k)$ such that the polynomial $\prod_i (X-\xi_i)$ annihilates any element of $\Ccal$. A marking is \textit{minimal} if is corresponds if it is of minimal size, i.e. if it corresponds to the minimal polynomial of $\Ccal$. 
\end{definition}

In turn, a (minimal) marking of $\bm{\Ccal_\infty}$ is the data of a (minimal) marking of the conjugacy class $\Ccal_i\subset \GL_{n_i}(\mathbb C)$ for each active circle $i\in N_c$.

\begin{definition}
\label{def:leg}
Let $(\Ccal, \xi)$ be a pair, with $\mathcal C\in \GL_n(\mathbb C)$ a conjugacy class and $\xi=(\xi_1,\dots, \xi_k)$ a marking of $\Ccal$. The \textit{leg} $\mathbb L_{\Ccal,\xi}$ associated to $(\Ccal, \xi)$ is a linear quiver with $k$ nodes $v_1,\dots,v_k$, and dimension vector $\mathbf d_{\Ccal, \xi}=(d_1,\dots, d_k)$ is the dimension vector for $\mathbb L_{\Ccal,\xi}$ defined by
\[
d_i:=\rank(M-\xi_1)\dots(M-\xi_{i-1})\in \N,
\]
for any $M\in \mathcal C$. 
\end{definition}

Notice that if we only consider minimal markings the leg $\mathbb L_{\Ccal,\xi}$ is independent of the choice of marking.

\begin{definition} 
\label{def:full_diagram_infinity}
 The \textit{full diagram} $\Gamma^{\infty}(\Theta_\infty,\bm{\Ccal}_\infty)$ associated to $(\Theta_\infty,\bm{\Ccal}_\infty)$ is the diagram obtained by gluing for each active circle $i$ to the corresponding vertex of $\Gamma_c^{\infty}(\Theta_\infty,\bm{\Ccal}_\infty)$ the leg $\mathbb{L}_i$ given by any choice of minimal marking of $\Ccal_i$. 
\end{definition}

Furthermore, any choice of a minimal marking $\bm \xi_\infty$ of $\bm \Ccal_\infty$, i.e. the choice of a minimal  marking of $\Ccal_i$ for each active circle $i$ defines a dimension vector $\mathbf{d}\in \Z^N$ for $\Gamma^\infty(E,\nabla)$, where $N$ denotes the set of nodes of the full diagram.\\

This recipe only enables us to define a diagram for connections with only one singularity at infinity. Our goal is to define a diagram for the general case, that is for connections with an arbitrary number of singularities. The task is not straightforward: although it is possible to draw a diagram corresponding to the formal data at each singularity, it is not clear that there is a natural way to somehow glue those diagrams together to define a global diagram having good properties. In other words, the way in which the different singularities should ``interact'' is not clear. 

The idea of our construction, already present in \cite{boalch2012simply}, is thus to use the Fourier transform to reduce to the case where all active circles are at infinity, where we know how to draw the diagram. Indeed, the Fourier transform sends to infinity all active circles at finite distance. However, the Fourier transform also sends to finite distance some of the active circles at infinity, and we are back to the problem we started with. For this reason, we need before taking the Fourier transform to apply some operation to prevent the active circles at infinity to go to finite distance. This can be achieved by using more general $\SL_2(\C)$ transformations on modules over the Weyl algebra. As we shall see, any matrix in an open dense subset of $\SL_2(\C)$ sends all active circles to infinity. 

We thus want to define the diagram associated to a connection as the diagram associated to its image under a generic $\SL_2(\C)$ transformation. For this to be well--defined however, the diagram must not depend on the choice of $\SL_2(\C)$ transformation to bring all active circles to infinity, in other words the diagram has to be invariant under $\SL_2(\C)$ transformations. The goal of this section is to show that it is indeed the case. The main step will be to show that the diagram $\Gamma^\infty(\Theta_\infty,\bm\Ccal_\infty)$ is invariant under Fourier transform. 

\subsection{An explicit formula for the number of edges}
\label{subsection_explicit_formula}

Let $\cir{q}$ and $\cir{q'}$ be two exponential factors at infinity, with ramification orders $\beta$, and $\beta'$, with slopes $\alpha/\beta$ and $\alpha'/\beta'$. We set
\[ q=\sum_{j=0}^{p} b_j z_\infty^{-\alpha_j/\beta},\qquad q'=\sum_{j=0}^{p'} b'_j z_\infty^{-\alpha'_j/\beta'}
\]
the expression for $q$ and $q'$ where we write only the monomials with non-zero coefficients, i.e. $b_j\neq 0$ and $b'_j\neq 0$,  $\alpha_0 > \dots > \alpha'_p$ and $\alpha'_0 > \dots > \alpha'_{p'}$. 
Because of ramification, a monomial appearing both in $q$ and $q'$ can give rise to Stokes arrows associated to the difference $q-q'$, since there are several leaves on the covers $I:=\cir{q}$ and $I':=\cir{q'}$. 
The leaves of $\cir{q}$ correspond to the images $q$ under the action of the Galois group isomorphic to $\Z/\beta\Z$ arising from ramification, i.e. to the polynomials
\[ q_i=\sum_{j=0}^{p}b_j \omega^{-\alpha_j}z_\infty^{-\alpha_j/\beta}, \quad i=0,\dots,\beta-1,\]
with $\omega=e^{2i\pi/\beta}$. In a similar way, the leaves of $\cir{q'}$ correspond to the polynomials
\[ q'_i=\sum_{j=0}^{p'}b'_j \omega'^{-\alpha'_j}z_\infty^{-\alpha'_j/\beta'}, \quad i=0,\dots,\beta'-1,\]
with $\omega'=e^{2i\pi/\beta'}$.

Let $r$ be the biggest integer such that $\sum_{j=0}^{r-1}b_j z_\infty^{-\alpha_j/\beta}$ and $\sum_{j=0}^{r-1}b'_j z_\infty^{-\alpha'_j/\beta'}$ have the same Galois orbit. 

We set 
\[ q_c:=\sum_{j=0}^{r-1}b_j z_\infty^{-\alpha_j/\beta}, \qquad q'_c:=\sum_{j=0}^{r-1}b'_j  z_\infty^{-\alpha'_j/\beta'}, \]
and 
\[ q_d:=\sum_{j=r}^{p}b_j z_\infty^{-\alpha_j/\beta}, \qquad q'_d:=\sum_{j=r}^{p} b'_j z_\infty^{-\alpha'_j/\beta'}, \]
We get a decomposition 
\[ q=q_c+q_d, \qquad q'=q'_c+q'_d,
\]
of $q$ and $q'$ as the sum of a \textit{common part} $q_c$ and a \textit{different part} $q_d$. Replacing $q$ or $q'$ by another element of their Galois orbit if necessary, we may assume that $q_c=q'_c$. In the case where $q$ and $q'$ do not have the same leading term, then $r=0$ and we have $q_c=q'_c=0$ and $q=q_d$, $q'=q'_d$. In particular, this happens when $q$ and $q'$ do not have the same slope. Otherwise, $q$ and $q'$ have a non-zero common part, in particular they have the same slope $\alpha_0/\beta=\alpha'_0/\beta'$. The slope of $q_d$ is $\alpha_{r}/\beta$, and the slope of $q'_d$ is $\alpha'_{r}/\beta'$. 
We use the usual notation $(\cdot,\cdot)$ to denote the greatest common divisor. 
With those notations now set, we are in position to state the formula for the number of edges between $I$ and $I'$.

\begin{lemma} 
\label{nb_aretes_cas_general}
\begin{itemize}
\item Assume that and $\alpha_{r}/\beta\geq \alpha'_{r}/\beta'$. Then the number of edges between $I=\cir{q}$ and $I'=\cir{q'}$ is
\begin{align*}
B_{I,I'}= &(\beta'-(\alpha'_0,\beta'))\alpha_0 +((\alpha'_0,\beta') - (\alpha'_0,\alpha'_1,\beta'))\alpha_1  + \dots +((\alpha'_0,\dots, \alpha'_{r-2},\beta') - (\alpha'_0,\dots,\alpha'_{r-1},\beta'))\alpha_{r-1} \\&+(\alpha'_0,\dots,\alpha'_{r-1},\beta') \alpha_{r} - \beta\beta'.
\end{align*}
\item In particular, if $q$ and $q'$ have no common parts and $\alpha/\beta\geq \alpha'/\beta'$, then 
\begin{align*}
B_{I,I'}= \beta'(\alpha - \beta).
\end{align*}
\end{itemize}

\end{lemma}

We have a similar result for the number of loops at a circle $\cir{q}$. 

\begin{lemma}
Let $q=\sum_{j=0}^{p} b_j z_\infty^{-\alpha_j/\beta}$ be an exponential factor of slope $\alpha_0/\beta>1$ as before.

\begin{itemize}
\item One has
\begin{equation}
B_{I,I}=(\beta-(\alpha_0,\beta))\alpha_0 +((\alpha_0,\beta)-(\alpha_0,\alpha_1,\beta))\alpha_1 +\dots
+((\alpha_0,\dots,\alpha_{p-1},\beta)-(\alpha_0,\dots,\alpha_p,\beta))\alpha_p-\beta^2+1.
\end{equation}
\item In particular, if $(\alpha,\beta)=1$, then we have 
\begin{equation}
B_{I,I}=(\beta-1)(\alpha-\beta-1).
\end{equation}
\end{itemize}
\end{lemma}

The proofs are a bit lenghty and are given in appendix.

\begin{lemma}
The integer $B_{I,I}$ is even. 
\end{lemma}

\begin{proof}
Let us first assume that $\beta$ is odd. Then all greatest common divisors appearing in the formula for $B_{I,I}$ are odd, so differences between two consecutive g.c.d.s are even, and all terms involving those differences are even. The sum of the remaining terms is $-\beta^2+1$ which is even, so the result follows. 

Now, let us consider the case when $\beta$ is even. Let us consider the sequence of greatest common divisors $\beta, (\alpha_0,\beta), \dots, (\alpha_0,\dots,\alpha_p,\beta)$. The first element $\beta$ is even, the last one $(\alpha_0,\dots,\alpha_p,\beta)$ is odd,  and the sequence consists first of even integers until $(\alpha_0,\dots,\alpha_{k_0},\beta)$ where $k_0$ is the smallest index $k$ such that $\alpha_k$ is odd. Starting from this element, all g.c.d.s in the sequence are odd. As a consequence, all the terms involving differences of g.c.d.s in the formula for $B_{I,I}$ are even except $((\alpha_0,\dots,\alpha_{k_0-1},\beta)-(\alpha_0,\dots,\alpha_{k_0},\beta))\alpha_{k_0}$ which is odd. The sum of the remaining terms $-\beta^2+1$ is odd, and the conclusion follows. 
\end{proof}

This proves what was a statement in \cite[p.3]{boalch2020diagrams}, and guarantees that we have a diagram with unoriented edges.

\subsection{Form of the Legendre transform}

In this paragraph, we compute the Legendre transform of an exponential factor $q$ at infinity as explicitly as possible. We determine the monomials appearing (with non-zero coefficients) in the Legendre transform.

\begin{proposition}
\label{forme_de_tilde(q)}
Let $q=\sum_{j=0}^{p} b_j z^{\alpha_j/\beta}$, with $b_j\neq 0$, be an exponential factor at infinity with ramification $\beta$. We set $\alpha:=\alpha_0$, so the slope of $q$ is $\alpha/\beta>1$. Then the exponents of $\xi$ possibly appearing with non-zero coefficients in its Legendre transform $\widetilde{q}$ are of the form $\frac{\alpha-k_1(\alpha-\alpha_1)-\dots-k_p(\alpha-\alpha_p)}{\alpha-\beta}$, with $k_1,\dots,k_p\geq 0$. More precisely, if we set $E:=\{ \gamma\in \N \;|\;\exists k_1,\dots, k_p \in \N, \gamma=k_1(\alpha-\alpha_1)+\dots+k_p(\alpha-\alpha_p) \}$
the Legendre transform has the form
\begin{equation}
\widetilde{q}(\xi)=\sum_{\substack{\gamma \in E \\ \alpha-\gamma >0}}\widetilde{b}_{\gamma}\xi^{\frac{\alpha-\gamma}{\alpha-\beta}},
\end{equation} 
where the sum is restricted to the terms such that the exponent $\frac{\alpha-\gamma}{\alpha-\beta}$ is positive. Furthermore, the coefficients $\widetilde{b}_{(\alpha-\alpha_i)}$ are non-zero for $i\geq 1$. 
\end{proposition}

We refer the reader to the appendix for the proof.

\subsection{Invariance of the diagram under Fourier-Laplace transform}

We now prove the invariance of the diagram under Fourier transform.

\begin{theorem}
Let $\Theta_\infty: \pi_0(\Ical_\infty)\to \N$ be an irregular class at infinity, such that the slopes of all active circles are $>1$. Then the diagram $\Gamma^\infty(\Theta_\infty)$ is invariant under Fourier-Laplace transform. 
\end{theorem} 

\begin{proof} Since the stationary phase formula acts on the conjugacy class of the formal monodromies by multiplying them by $\pm 1$, it doesn't change the length of the legs. We have to show that the core diagram is invariant. The proof comes from combining lemma \ref{nb_aretes_cas_general} giving the number of edges of the core diagram, and lemma \ref{forme_de_tilde(q)} giving the structure of the Legendre transform. This allows us to compute the number of edges in the diagrams $\Gamma_c(\Theta_\infty)$ and $\Gamma_c(\Theta'_\infty)$, where $\Theta'_\infty:=F\cdot \Theta_\infty$, and check they are the same.
Keeping the same notations as before, let $I:=\cir{q}$ and $I':=\cir{q'}$ be two distinct active circles of $\Theta_\infty$. We want to compute the number of edges between their images  $\widetilde{I}:=\cir{\widetilde{q}}$ and $\widetilde{I}':=\cir{\widetilde{q}'}$. Lemma \ref{forme_de_tilde(q)} implies that $\widetilde{q}$ is of the form 
\[ 
\widetilde{q}=\sum_{\gamma\in E}\widetilde{b}_{\gamma}\xi_\infty^{-\frac{\alpha-\gamma}{\alpha-\beta}},
\]
with $E=\{\gamma\in \N \; |\; \gamma<\alpha,\; \exists k_1,\dots,k_p\geq 0, \; \gamma=k_1(\alpha_0-\alpha_1)+\dots+k_p(\alpha_0-\alpha_p)\}$, and similarly
\[ 
\widetilde{q}'=\sum_{\gamma'\in E'}\widetilde{b}'_{\gamma'}\xi_\infty^{-\frac{\alpha'-\gamma'}{\alpha'-\beta'}},
\]
with $E'=\{\gamma'\in \N \; |\; \gamma'<\alpha',\; \exists k_1,\dots,k_p\geq 0, \; \gamma'=k_1(\alpha'_0-\alpha'_1)+\dots+k_p(\alpha'_0-\alpha'_p)\}$. Let us denote by $0=\gamma_0<\dots <\gamma_N$ the distinct elements of $E$ and $0=\gamma'_0 <\dots <\gamma'_{N'}$ those of $E'$. We first determine the common parts and the different parts of $\widetilde{q}$ and $\widetilde{q}'$. The integers $\gamma$ and $\gamma'$ containing only terms common to $q$ and $q'$, i.e. of the form $\gamma=k_1(\alpha_0-\alpha_1)+\dots+k_{r-1}(\alpha_0-\alpha_{r-1})$ and $\gamma'=k_1(\alpha'_0-\alpha'_1)+\dots+k_{r-1}(\alpha'_0-\alpha'_{r-1})$, belong to the part common to $\widetilde{q}$ and $\widetilde{q'}$. Let us denote by $E_c\subset E$ and $E'_c\subset E'$ the corresponding subsets. Furthermore, the respective leading terms of the different parts of $\widetilde{q}_d$ and $\widetilde{q}'_d$ correspond to the first integers $\gamma$ where the factors $(\alpha-\alpha_{r})$ and $(\alpha'-\alpha'_{r})$ appear, which are  $\alpha-\alpha_{r}=:\gamma_{R}$, and $\alpha'-\alpha'_{r+1}=:\gamma'_{R}$ with $R\leq N$. It follows that $\widetilde{q}_d$ is of degree $\frac{\alpha_r}{\alpha-\beta}$ and $\widetilde{q}'_d$ is of degree $\frac{\alpha'_r}{\alpha'-\beta'}$. Since we have assumed $\alpha_r/\beta > \alpha'_r/\beta'$, we have $\frac{\alpha_r}{\alpha-\beta}>\frac{\alpha'_r}{\alpha'-\beta'}$. From lemma \ref{nb_aretes_cas_general}, the number of edges  between $\cir{\widetilde{q}}$ and $\cir{\widetilde{q}'}$ is
\begin{equation}
B_{\widetilde{I},\widetilde{I}'}=((\alpha'-\beta')-p'_0)\alpha_0 + \sum_{k=1}^{R-1} (p'_{k-1}-p'_{k})(\alpha-\gamma_k)+p_R (\alpha-\gamma_{R})-(\alpha-\beta)(\alpha'-\beta'),
\end{equation}
where $p'_k=(\alpha'_0,\dots,\alpha'_0-\gamma_k,\alpha'-\beta')$.

In this expression, let us examine the g.c.d.s $p'_k$ and the differences $p'_{k-1}-p'_k$.
Notice that one has $p'_{k-1}=p'_k$ as soon as any decomposition $\gamma_k=k_1(\alpha'_0-\alpha'_1)+\dots+k_p(\alpha'_0-\alpha'_p)$ of $\gamma_k$ only contains factors $(\alpha-\alpha_j)$ already present in in some $p_l$ with $l\leq k-1$. In other words, we can have $p'_{k-1}>p'_k$ only when in $\gamma_k$ a new factor $(\alpha-\alpha_j)$ appears, that is when $\gamma_k=\alpha-\alpha_j$ for some $j=1,\dots,p$.
The distinct values of $p'_k$ are thus the g.c.d.s $(\alpha'_0,\dots, \alpha'_l,\alpha'-\beta')=(\alpha'_0,\dots ,\alpha'_l ,\beta')$. Therefore, the non-zero terms involving g.c.d.s in the formula for $B_{\cir{\widetilde{q}},\cir{\widetilde{q}'}}$ are exactly the same as those for $B_{I,I'}$. The sum of the remaining terms is $\alpha'\beta-\beta\beta'=\beta(\alpha'-\beta')$ for $B_{I,I'}$, and $(\alpha'-\beta')\alpha_0-(\alpha-\beta)(\alpha'-\beta')=\beta(\alpha'-\beta')$ for $B_{\widetilde{I},\widetilde{I}'}$: they are also the same. Finally, we have the equality $B_{I,I'}=B_{\widetilde{I},\widetilde{I}'}$ of the numbers of edges. 

It remains to deal with the case of loops at $\cir{q}$ and $\cir{\widetilde{q}}$. In a similar way as the previous case, the non-zero differences of g.c.d.s appearing in the formula for $B_{\widetilde{I},\widetilde{I}}$ are exactly the same as those appearing in the formula for $B_{I,I}$. The sum of the remaining terms for $B_{I,I}$ is $\beta(\alpha-\beta)+1$, and this expression is invariant under Fourier transform since $\alpha\mapsto \alpha$ and $\beta \mapsto \alpha-\beta$. We thus have $B_{I,I}=B_{\widetilde{I},\widetilde{I}}$, which completes the proof of the invariance of the diagram.
\end{proof}

\subsection{Invariance of the diagram under symplectic transformations}

The goal of this section is to show: 

\begin{theorem}
Let $\Theta_\infty$ be an irregular class at infinity. There is an open dense subset $S_V\subset \SL_2(\C)$ such that for all $A\in S_V$, $A\cdot \Theta_\infty$ only has a singularity at infinity. Furthermore, we have $\Gamma^\infty(\Theta_\infty)=\Gamma^\infty(A\cdot \Theta_\infty)$.
\end{theorem}

To prove this, the idea is to decompose the elements of $\SL_2(\C)$ as products of elementary operations. 

\begin{lemma} Let $\Theta_\infty$ be an irregular class at infinity. Then 
\begin{itemize}
\item For $\lambda\in \C$, $\Gamma^\infty(T_\lambda\cdot \Theta_\infty )=\Gamma^\infty(\Theta_\infty)$.
\item For $\lambda\in \C^*$, $\Gamma^\infty(S_\lambda\cdot \Theta_\infty)=\Gamma^\infty(\Theta_\infty)$.
\end{itemize}
\end{lemma}

\begin{proof}
Once again in both cases the lengths of the legs do not change, so we just have to show that the core diagrams are the same. The twist $T_\lambda$ induces via $t_\lambda$ a bijection between the active circles of $\Theta_\infty$ and $T_\lambda\cdot \Theta_\infty$.  The differences of exponential factors $q_U(z_\infty)-q'_U(z_\infty)$ are invariant under a twist $T_\lambda$, so the number of Stokes arrows between two active circles and their images by $t_\lambda$ are the same. This proves the first part of the lemma. 
In a similar way, $S_\lambda$ induces via $s_\lambda$ a bijection between the active circles of $\Theta_\infty$ and $S_\lambda \cdot \Theta_\infty$. The degree of a difference $q_U(z_a)-q'_U(z_a)$ of exponential factors at $a\in\C$ is the same as its image $q_U(\lambda(z_{a/\lambda}))-q'_U(\lambda(z_{a/\lambda}))$  by $s_\lambda$. This is also true for exponential factors at infinity. Therefore, the number of Stokes arrows between two active circles and their images by $s_\lambda$ are the same, and the conclusion follows.
\end{proof}

Since we have already seen that the Fourier transform preserves the diagram when all circles are of slope $>1$, i.e. when $F\cdot \Theta_\infty$ only has a singularity at infinity, all elementary transformations preserve the diagram. However, deducing the theorem from this by factoring matrices of $\SL_2(\C)$ as products of elementary transformations is not entirely straightforward. Indeed, if $A\in \SL_2(\C)$ admits a factorization $A=A_k\dots A_1$ as a product of elementary transformations, it may occur that for some $1\leq i\leq k$, $A_i \dots A_1\cdot \Theta_\infty$ has singularities at finite distance. Therefore, we have to show that $A$ always admits a factorization such that this doesn't happen. We will do this in several steps. 

\begin{lemma}
Let $\Theta_\infty$ be an irregular class at infinity. There exists $A\in \SL_2(\C)$ such that all active circles of $A\cdot \Theta_\infty$ have slope no greater than $2$. More precisely, there exists such a matrix of the form $FT_\lambda$, for some $\lambda\in \C$. 
\end{lemma} 

\begin{proof}
Let $\cir{q}_\infty$ be an irregular circle at infinity. Its image by $T_\lambda$ is $\cir{q+\frac{\lambda}{2}z^2}$. The term $\frac{\lambda}{2}z^2$ is of slope $2$. There are thus two possibilities: either $q$ has slope strictly greater than 2 and $\slope(q)=\slope(q+\frac{\lambda}{2}z^2)$, or $q$ has slope less than 2 and $\slope(q+\frac{\lambda}{2}z^2)=2$, if we choose $\lambda$ such that $-\lambda/2$ is not equal to the coefficient of slope 2 of $q$. We then apply the Fourier transform. In the first case, if we denote $\slope(T_\lambda\cdot \cir{q})=p/q>2$, with $p,q\in \N$, one has $p-q>q$, so from the stationary phase formula $\slope(FT_\lambda\cdot\cir{q)}=\frac{p}{p-q}<2$. In the second case we have $\slope(FT_\lambda\cdot\cir{q})=\slope(T_\lambda\cdot \cir{q})=2$. As a consequence, if we choose $\lambda$ such that $-\lambda/2$ is not equal to any of the coefficients of slope 2 of the active circles of $\Theta_\infty$, then $FT_\lambda \cdot \Theta_\infty$ only has active circles at infinity of slope no greater than 2.
\end{proof}

Let $\cir{q}$ be a circle at infinity of slope $\leq 2$. It is clear that this circle remains at infinity when a twist or a scaling is applied. It remains at infinity under Fourier transform if and only if its slope is strictly greater than 1. In particular, it remains at infinity if its coefficient of slope 2 is non-zero. 

\begin{proposition} Let $q$ be an exponential factor at infinity of slope $\leq 2$. We write $q=-\frac{\lambda}{2}z^2+ q_{<2}$ with $q_{<2}$ of slope $<2$. Let $A=\begin{pmatrix}
a & b\\
c & d
\end{pmatrix}\in \SL_2(\C)$. Let $\cir{q'}$ be the image of $\cir{q}_\infty$ by $A$. Then $\cir{q'}$ is also at infinity unless $\slope(q_{<2})\leq 1$ and $c\lambda + d=0$. In this case, writing $q'=-\frac{\lambda'}{2}+ q'_{<2}$ with $q_{<2}$ of slope $<2$, the coefficients $\lambda$ and $\lambda'$ are related by 
\begin{equation}
\lambda'=\frac{a\lambda+b}{c\lambda+d}.
\end{equation}
\end{proposition}

To prove this, we begin with the case of elementary transformations. 

\begin{lemma}
The proposition is true if $A$ is an elementary transformation
\end{lemma}

\begin{proof}
The twist $T_{\mu}=\begin{pmatrix}
1 & \mu\\
0 & 1\\
\end{pmatrix}$
acts on $q$ according to 
\[ T_{\mu}(q)=q-\frac{\mu}{2}z^2=-\frac{\lambda+\mu}{2}z^2+q_{<2}.
\]
This amounts to $\lambda \mapsto \lambda+\mu$, so the property is verified.
The scaling $S_{\mu}=\begin{pmatrix}
\mu^{-1} & 0\\
0 & \mu\\
\end{pmatrix}$ sends $q$ to 
\[ 
S_{\mu}(q)=q(z/\mu)=-\frac{\lambda \mu^2}{2}z^2+q_{<2}(z/\mu).
\]
This amounts to $\lambda \mapsto \lambda\mu^2$, so the property is again verified.
The Fourier transform $F=\begin{pmatrix}
0 & 1\\
-1 & 0\\
\end{pmatrix}$ acts on $q$ according to the Legendre transform: computing it explicitly we find  
\[ F(q)=\Lcal q=\frac{1}{2\lambda}x^2+\widetilde{q}_0,
\]
where $\widetilde{q}_0$ has slope $<2$. This amounts to $\lambda\mapsto -\frac{1}{\lambda}$ and the property is once again verified.
\end{proof}

The idea is then to factorize $A$ as a product of elementary operations, using the following lemma. 

\begin{lemma}
Let $A=\begin{pmatrix}
a & b\\
c & d
\end{pmatrix}\in \SL_2(\C)$, and $h_A : \Pbb^1\to \Pbb^1$ be the corresponding homography $z\mapsto \frac{az+b}{cz+d}$.  
\begin{itemize}
\item If $h_A(\infty)=\infty$, then $A$ admits a factorization as a product of elementary operations of the form 
$A=S_\nu T_\rho$, 
with $\nu\in \C^*, \rho\in \C$. 
\item Otherwise, if $h_A(\infty)\neq \infty$, $A$ admits a factorization of the form 
\begin{equation}
A= T_\mu F S_\nu T_\rho,
\end{equation}
with $\mu\in \C$, $\nu\in \C^*$, $\rho\in \C$. 
\end{itemize}
\end{lemma}

\begin{proof} We have $h_A(\infty)=\frac{a}{c}\in \Pbb^1$, so $h_A$ fixes infinity if and only if $c=0$. This implies $a\neq 0$ and $d=a^{-1}$, and we have $A=S_\nu T_\rho$ with $\nu=a$, $\rho=b/a$. When $h_A(\infty)\neq \infty$,we can reduce to the case where $h_A$ fixes infinity. Let $\mu:= h_A(\infty)=\frac{a}{c}$. We introduce 
\[ B_\mu:=\begin{pmatrix}
0 & -1\\ 1 & -\mu
\end{pmatrix}=F^{-1}T_{-\mu},
\]
such that the corresponding homography $h_{B_\mu}$ sends $\mu$ to $\infty$. Then the homography associated to $B_\mu A$ sends $\infty$ to $\infty$. This implies that $B_\mu A$ has the form
\[ B_\lambda A=\begin{pmatrix}
\nu & \sigma\\
0 & \nu^{-1}
\end{pmatrix},
\]
with $\nu=-c$ and $\sigma=-d$, so that $B_\mu A=S_\nu T_\rho$ with $\rho=\sigma/\mu$. Finally we get $A=B_\mu^{-1}S_\nu T_\rho=T_\mu F S_\nu T_\rho$. 
\end{proof}

\begin{proof}[Proof of the proposition] If $c=0$ then $d\neq 0$ and $c\lambda+d=d\neq 0$. In this case $A$ factorizes as $A=S_\nu T_\rho$, so the circle $\cir{q'}$ is at infinity. Now if $c\neq 0$, from the lemma $A$ factorizes as $A=T_\mu F S_\nu T_\rho$. Let $\cir{q''}:=S_\nu T_\rho \cdot \cir{q}_\infty$, and write $q''=-\frac{\lambda''}{2}+q''_{<2}$ with the slope of $q''_{<2}$ being $<2$. The circle $\cir{q'}$ is at infinity if and only if $\cir{q''}$ is not sent to finite distance by the Fourier transform, i.e. using the stationary phase formula if we don't have $\lambda''=0$ and $\slope(q''_{<2})\leq 1$. From the Legendre transform $\slope(q''_{<2})=\slope(q_{<2})$. We have $S_\nu T_\rho=\begin{pmatrix}
-c & -d\\
0 & -c^{-1}
\end{pmatrix}$, so $\lambda''=-\frac{c \lambda +d}{-c^{-1}}$. Therefore we have $\lambda''\neq 0$ if and only if $c\lambda +d\neq 0$. It follows from this that, if we do not have $c\lambda+d=0$ and $\slope(q_{<2})\leq 1$, then at each step of the factorization the images of $\cir{q}$ remain at infinity. Since at each elementary operation the coefficient $\lambda$ is transformed according to the corresponding homography, the conclusion follows. 

\end{proof}

We are now in position to prove the theorem.

\begin{proof}[Proof of the theorem] Let $\Theta_\infty$ be an irregular class at infinity. Up to applying a well-chosen symplectic transformation, we may assume that all active circles of $\Theta_\infty$ have slope less than $2$. We denote by $-\lambda_i/2, i=1,\dots,r$ the coefficients of $z^2$ of these active circles. It follows from the previous proposition that for $A=\begin{pmatrix}
a & b\\
c & d
\end{pmatrix}\in \SL_2(\C)$, $A\cdot \Theta_\infty$ only has a singularity at infinity if and only if $c\lambda_i+d\neq 0$ for $i=1,\dots,m$. This corresponds to an open dense subset of $\SL_2(\C)$. Furthermore $A$ has a factorization $A= T_\lambda F S_\mu T_\nu$, such that when applying successively the elementary operations, $\Theta_\infty$ always remains with only one singularity at infinity. At each step, the diagram doesn't change, and the conclusion follows.  
\end{proof}

\begin{remark} It follows from this analysis that one may view the active circles at finite distance as circles at infinity of slope $\leq 2$ for which the coefficient of slope 2 is infinite. We may call the copy of $\Pbb^1$ where those coefficients live the \textit{Fourier sphere}, as in \cite{boalch2012simply}. Acting with symplectic transformations on circles at infinity of slope $\leq 2$ amounts to acting on the Fourier sphere with the corresponding homographies. It is then clear that for a generic rotation none of the coefficients will be at infinity, i.e. there will only be one singularity at infinity.

\end{remark}

\section{Diagrams for general connections}

\subsection{Definition of the diagram}

To define the diagram in the the general case where there are several irregular singularities, the idea is to use the Fourier-Laplace transform to reduce to the situation where there is only one singularity at infinity studied in the previous section. The invariance of the diagram under $\SL_2(\C)$ will ensure that this construction is well--defined.

\begin{theorem}
Let $(E,\nabla)$ be an irreducible algebraic connection on a Zariski open subset of $\A^1$, which is not a rank 1 connection with only a singularity at infinity of order less than two. Let $(\bm{\breve\Theta},\bm{\breve\Ccal})$ be its modified formal data. Then for any $A$ in an open dense subset of $\SL_2(\C)$, the formal data $A\cdot (\bm{\breve\Theta},\bm{\breve\Ccal})$ only have support at infinity, and the diagram $\Gamma^\infty(A\cdot (\bm{\breve\Theta},\bm{\breve\Ccal}))$ is independent of $A$. 
\end{theorem}

\begin{proof}
It is easy to find a matrix $A\in \SL_2(\C)$ such that $A\cdot (\bm{\breve\Theta},\bm{\breve\Ccal})$ has support at infinity. For example we may take $A$ of the form $A=F T_\lambda$, with $\lambda\in \C$: it suffices to choose the coefficient $\lambda$ of the twist such that all the coefficients of the terms of slope 2 in the active circles of $T_\lambda \cdot (\bm{\breve\Theta},\bm{\breve\Ccal})$ at infinity are non-zero. Applying then the Fourier transform, this guarantees that the active circles at infinity of $T_\lambda \cdot \bm{\breve\Theta}$ remain at infinity while the active circles at finite distance are sent to infinity. The results of the previous section now imply that $ BA\cdot (\bm{\breve\Theta},\bm{\breve\Ccal})$ has support at infinity  and that $\Gamma^\infty(BA\cdot (\bm{\breve\Theta},\bm{\breve\Ccal}))=\Gamma^\infty(A\cdot (\bm{\breve\Theta},\bm{\breve\Ccal}))$ for $B$ in an open dense subset of $\SL_2(\C)$. 
\end{proof}

\begin{definition}
\label{def: general_diagram_first_definition}
Let $(E,\nabla)$ be an irreducible connection on a Zariski open subset of the affine line, $(\bm\Theta,\bm{\Ccal})$ its formal data, and $(\bm{\breve\Theta},\bm{\breve\Ccal})$ the corresponding modified formal data. We define the diagram associated to $(E,\nabla)$ as 
\[\Gamma(E,\nabla)=\Gamma(\bm\Theta,\bm{\Ccal})=\Gamma(\bm{\breve\Theta},\bm{\breve\Ccal}):=\Gamma^\infty(A\cdot (\bm{\breve\Theta},\bm{\breve\Ccal}))\] for any $A\in \SL_2(\C)$ such that $A\cdot (\bm{\breve\Theta},\bm{\breve\Ccal})$ has support at infinity. 
\end{definition}

\subsection{Explicit formula for the core diagram}

The construction of the diagram given above involves a choice of $\SL_2(\C)$ transformation. The explicit formulas for the number of edges and our understanding of the Legendre
transform allow us to give a direct expression of the core diagram, which we may then take as the definition of the core diagram as in the introduction.

Let $I=\cir{q},I'=\cir{q'}$ be circles at points $a,a'\in\Pbb^1$. Using similar notations as before, we write $\beta:=\ram(q)$, $\beta':=\ram(q')$ and
\[ q=\sum_{i=0}^{k} b_i z_a^{-\alpha_i/\beta},\qquad q'=\sum_{i=0}^{k'} b'_i z_{a'}^{-\alpha'_i/\beta'},
\]
where the coefficients $b_i,b'_i\in\C$ are non-zero, and set $\alpha:=\alpha_0=\Irr(q)$, $\alpha':=\alpha'_0=\Irr(q')$. If $q$ and $q'$ are both at infinity, that is if $a=a'=\infty$, then we know from a previous lemma the number of edges $B_{\cir{q},\cir{q'}}$, which depends only on the degrees of the terms of $q$ and $q'$. Let us denote by $B^{\infty}_{\cir{q},\cir{q'}}$ this formula. 

\begin{theorem}
\label{thm:direct_formula_diagram}
 Let $I=\cir{q}_a$ and $I'=\cir{q'}_{a'}$ be active circles. If $a=a'$, denoting as before by $\alpha_r/\beta$ and $\alpha'_r/\beta'$ the slopes of the different parts of $q$ and $q'$, we assume that $\frac{\alpha_r}{\beta}\geq\frac{\alpha'_r}{\beta'}$. The number of edges $B_{I,I'}$ between the vertices associated to $I$ and $I'$ following:
\begin{enumerate}
\item If $a=a'=\infty$ then $B_{I,I'}=B^{\infty}_{I,I'}$.
\item If $a=\infty$ and $a'\neq \infty$ then $B_{I,I'}=\beta(\alpha'+\beta')$.
\item If $a\neq \infty, a'\neq \infty$ and $a\neq a'$ then $B_{I,I'}=0$.
\item If $a=a'\neq \infty$ then $B_{I,I'}=B^{\infty}_{I,I'}-\alpha\beta'-\alpha'\beta$.
\end{enumerate}
\end{theorem}

\begin{proof} To prove this, we apply a generic $\SL_2(\C)$ twist at infinity acting on exponential factors at infinity as $q\mapsto q+\lambda z^2$, followed by Fourier transform. From our study of the Legendre transform we obtain the form of the images of $q$ and $q'$ under this process, then we apply the formula for the number of edges at infinity. Let us first deal with the cases where $q\neq q'$.
\begin{enumerate}
\item This case is just a consequence of the fact that the twist and the Fourier transform leave $B^\infty_{\cir{q},\cir{q'}}$ invariant.
\item Let $q_1$ and $q'_1$ be the images of $q$, $q'$. If $\frac{\alpha}{\beta}\leq 2$, then $q_1$ has ramification $\beta$ and slope $2=2\beta/\beta$. If $\frac{\alpha}{\beta}>2$, then $q_1$ has ramification $\alpha-\beta$ and slope $\frac{\alpha}{\alpha-\beta}$. On the other hand, $q'_1$ is of the form
\[ q'_1=-a'z+ b'z^{\alpha'/(\alpha'+\beta')}+\dots
\]
for some $b'\in\C$, $b'\neq 0$, so $q'_1$ has ramification $\alpha'+\beta'$ and slope 1 (if $a\neq 0$), or $\frac{\alpha'}{\alpha'+\beta'}$ (if $a=0$). When $\frac{\alpha}{\beta}\leq 2$, it follows from the formula for $B^\infty$ that 
\[ B_{\cir{q},\cir{q'}}=B^{\infty}_{\cir{q_1},\cir{q'_1}}=(\alpha'+\beta')(2\beta-\beta)=\beta(\alpha'+\beta').
\]
When $\frac{\alpha}{\beta}>2$ the same formula gives
\[ B_{\cir{q},\cir{q'}}=B^{\infty}_{\cir{q_1},\cir{q'_1}}=(\alpha'+\beta')(\alpha-(\alpha-\beta))=\beta(\alpha'+\beta').
\]
so both cases yield the expected result.

\item In this case $q_1$ and $q'_1$ are of the form
\[ q_1=-az+ bz^{\alpha/(\alpha+\beta)}+\dots,\qquad q'_1=-a'z+ b'z^{\alpha'/(\alpha'+\beta')}+\dots,
\]
that is $q_1$ and $q'_1$ have respective ramification orders $\alpha+\beta$ and $\alpha'+\beta'$, have both slope 1 and non common part. The formula for $B^\infty$ thus gives
\[ B_{\cir{q},\cir{q'}}=B^{\infty}_{\cir{q_1},\cir{q'_1}}=(\alpha'+\beta')((\alpha+\beta)-(\alpha+\beta))=0.
\]
\item If $a=a'$, $q_1$ and $q'_1$ are of the form
\[ q_1=-az+\tilde{q}, \qquad q'_1=-az+\tilde{q}',
\]
with $\tilde{q}$ having ramification $\alpha+\beta$ and slope $\frac{\alpha}{\alpha+\beta}$, and $\tilde{q}'$ having ramification $\alpha'+\beta'$ and slope $\frac{\alpha'}{\alpha'+\beta'}$. 
Let us first assume that $q$ and $q'$ have a common part. Then so do $q_1$ and $q'_1$. The slopes or their different parts are $\frac{\alpha_r}{\alpha+\beta}\geq \frac{\alpha'_r}{\alpha'+\beta'}$. The form of $q_1$ and $q'_1$ is given by the lemma giving the form of the Legendre transform. The slopes or their different parts are $\frac{\alpha_r}{\alpha+\beta}\geq \frac{\alpha'_r}{\alpha'+\beta'}$ and the formula for $B^\infty_{\cir{q_1},\cir{q'_1}}$ gives
\begin{align*}
B_{\cir{q},\cir{q'}}&=B^{\infty}_{\cir{q_1},\cir{q'_1}}\\
&=((\alpha'+\beta')-(\alpha'_0,\alpha'+\beta'))\alpha_0+\dots +(\alpha'_0,\dots,\alpha'_{r-1},\alpha'+\beta')\alpha_r-(\alpha+\beta)(\alpha'+\beta')\\
&=((\alpha'+\beta')-(\alpha'_0,\beta'))\alpha_0+\dots +(\alpha'_0,\dots,\alpha'_{r-1},\beta')\alpha_r-(\alpha\alpha'+\beta\beta'+\alpha'\beta+\alpha\beta')\\
&=B^\infty_{\cir{q},\cir{q'}}-\alpha\beta'-\alpha'\beta\\
&=B^\infty_{\cir{q},\cir{q'}}-2\alpha\beta',
\end{align*}
where in the last step we have used that $\frac{\alpha}{\beta}=\frac{\alpha'}{\beta'}$ since $q$ and $q'$ have a common part.
Now assume that $q$ and $q'$ have no common part. Then the slope of $\tilde{q}$ and $\tilde{q}'$ are respectively $\frac{\alpha}{\alpha+\beta}\geq \frac{\alpha'}{\alpha'+\beta'}$ (the order is preserved by the Legendre transform), so from the formula for $B^\infty_{\cir{\tilde{q}},\cir{\tilde{q}'}}$ we have
\[ B_{\cir{q},\cir{q'}}=-\beta(\alpha'+\beta').
\]
On the other hand, we have $B^\infty_{\cir{q},\cir{q'}}=\beta'(\alpha-\beta)$, so we again have $B^{\infty}_{\cir{q},\cir{q'}}=B^\infty_{\cir{q},\cir{q'}}-\alpha\beta'-\alpha'\beta$ as expected.
\end{enumerate}
To complete the proof, it remains to deal with the case of loops. We check that one has $B_{\cir{q},\cir{q}}=B^\infty_{\cir{q},\cir{q}}-2\alpha\beta$ when $a\neq \infty$, so the formula for $B_{\cir{q},\cir{q'}}$ remains valid when $q=q'$.
\end{proof}

Now, let us consider an algebraic connection $(E,\nabla)$ on a Zariski open subset of $\Sigma$. If we choose an isomorphism $\phi$ between $\Sigma$ and $\mathbb P^1=\mathbb C\cup \{\infty\}$ such that $\phi(\infty)=\infty$, this allows us to view $(E,\nabla)$ as an algebraic connection on a Zariski open subset of the affine line. In turn if $(E,\nabla)$ is irreducible (and not a rank one connection with only a singularity at infinity of order less than 2), definition \ref{def: general_diagram_first_definition} yields a diagram $\Gamma(E,\nabla)$. 
Notice that theorem \ref{thm:direct_formula_diagram} implies the diagram we obtain doesn't depend on the choice of isomorphism $\phi$. Moreover, the formula for the number of edges of the core diagram still makes sense even if $(E,\nabla)$ is not irreducible. We can thus make a slightly more general definition.

\begin{definition}

Let $(E,\nabla)$ be an algebraic connection on a Zariski open subset of $\Sigma$, $(\bm\Theta,\bm{\Ccal})$ its formal data, and $(\bm{\breve\Theta},\bm{\breve\Ccal})$ the corresponding modified formal data. Let $N_c$ be the set of active circles of $\bm{\breve\Ccal}$. We define a diagram
$\Gamma(E,\nabla)=\Gamma(\bm\Theta,\bm{\Ccal})=\Gamma(\bm{\breve\Theta},\bm{\breve\Ccal})$ as follows: it consists of a core diagram $\Gamma_c(E,\nabla)=\Gamma_c(\bm{\breve\Theta})$ with set of nodes $N_c$ and adjacency matrix given by the formulas of theorem \ref{thm:direct_formula_diagram}, to which are glued legs defined by any choice of minimal marking of $\bm{\breve\Ccal}$. Furthermore, any choice of such a marking defines a dimension vector $\mathbf d$ for $\Gamma(E,\nabla)$.
\end{definition} 

\begin{remark}
In \cite{boalch2020diagrams}, the authors also define a diagram $\Gamma(E,\nabla)$ for a connection with one irregular singularity at infinity together with regular singularities at finite distance. Our definition of the diagram coincides with the one of \cite{boalch2020diagrams} in this case as well, provided that in the latter framework one chooses minimal \textit{special} markings for the conjugacy classes of the monodromies at the regular singularities, where a special marking means a marking $\xi=(\xi_1,\dots, \xi_k)$ such that $\xi_1=1$. This comes from the fact that to pass from the formal data to the modified formal data, we kill the eigenspace for the eigenvalue 1 of the monodromy of the regular part.
\end{remark}

\subsection{Fundamental representations of the diagrams}
\label{section_fundamental_representations}

In this section, we discuss how a diagram can be ``read'' in different ways corresponding to connections on bundles of different ranks with different formal data, which we call the fundamental representations of the diagram. This generalizes the $k$-partite cases considered in \cite{boalch2012simply}.

The idea is the following. We have seen that the diagram associated to a meromorphic connection is invariant under the action of $\SL_2(\C)$. Each orbit under this action contains different formal data, with in general different numbers of singularities.  

Let $\cir{q}_\infty$ be the circle at infinity associated with an exponential factor $q$. We have seen that $\cir{q}$ can be sent to finite distance by a symplectic transformation if and only if it is of the form 
\begin{equation}
q=\lambda z^2+ q',
\label{circle_goes_at_finite_distance}
\end{equation}
where $q'$ is of slope $\leq 1$. Furthermore, when this is the case, the position of the singularity at finite distance that we obtain is determined by the coefficient of $z$ in $q$, i.e. the number $\mu\in \C$ such that 
\begin{equation}
q=\lambda z^2 +\mu z+q'',
\label{circle_finite_distance_bis}
\end{equation}
where $q''$ has slope $<1$. 

Let us consider formal data at infinity $(\Theta_\infty,\bm{\Ccal}_\infty)$ at infinity. Let $N\subset \pi_0(\Ical_\infty)$ be the subset of active exponents, i.e. the support of $\Theta_\infty$. $N$ is the set of vertices of the core diagram associated to $\Theta_\infty$. We can partition $N$ according to the coefficients of $z^2$ and $z$ in the exponential factors. Let $N_\infty\subset N$ be the set of exponential factors in $N$ which are not of the form \eqref{circle_goes_at_finite_distance} for any $\lambda\in \C$. For $\lambda\in \C$, let $N_\lambda\subset N$ be the set of exponential factors of the form \eqref{circle_goes_at_finite_distance}. For $\lambda,\mu\in\C$, let $N_{\lambda,\mu}\subset N_\lambda$ be the set of exponential factors of the form \eqref{circle_finite_distance_bis}.

There is a finite number of coefficients $\lambda\in\C$ such that $N_\lambda$, is non empty, which we denote $\lambda_1,\dots,\lambda_k$. We set $N_i:=N_{\lambda_i}$ for $i=1,\dots,r$. Then, for each $i$, there is again a finite number of coefficients $\mu\in\C$ such that $N_{\lambda_i,\mu}$ is non-empty, which we denote $\mu_{i,1},\dots,\mu_{k,s_i}$, and we set $N_{i,j}:=N_{\lambda_i,\mu_{i,j}}$. For each pair $i,j$, let $t_{i,j}:=\Card N_{i,j}$, and $q_{i,j,k}$, $k=1,\dots,t_{i,j}$ be the exponential factors of slope $<1$ such that the elements of $N_{i,j}$ are 
\begin{equation}
\lambda_i z^2 + \mu_{i,j} z+ q_{i,j,k}. 
\end{equation}
Let $\beta_{i,j,k}=\ram q_{i,j,k}$ be the ramification order of $q_{i,j,k}$, and $\alpha_{i,j,k}/\beta_{i,j,k}$ its slope. We thus have the following partition of the set of active circles:
\[
N=N_\infty \cup N_1 \cup \dots N_r.
\]

There are $r+1$ fundamental representations of diagram: the generic representation, corresponding to $\Theta$, and $r$ other representations, depending on which of the sets $N_1,\dots, N_r$ we choose to send to finite distance by a symplectic transformation. The generic representation corresponds to symplectic transformations $A\in \SL_2(\C)$ such that all circles remain at infinity under $A$. From the previous discussion of the action of $\SL_2(\C)$ on the coefficients of $z^2$ in the exponential factor, if
\[ A=\begin{pmatrix}
a&b\\
c&d
\end{pmatrix},
\]
the condition for this is $2c\lambda-d\neq 0$. 

For each $i=1,\dots,r$ the $i$-th representation of the diagram corresponds to acting on $\Theta_\infty$ with $A\in \SL_2(\C)$ such that $2c\lambda_i-d=0$. Such a symplectic transformation sends all circles in $N_i$ to finite distance. The formal data that we obtain have $s_i=\text{Card}(N_i)$ poles at finite distance corresponding to the coefficients $\mu_{i,1},\cdots,\mu_{i,s_i}$ (and whose position depends on $A$). For each $j=1,\dots,s_i$, the vertices in $N_{i,j}$ then correspond to the active circles of $A\cdot \Theta_\infty$ at the pole associated to $\mu_{i,j}$. It follows from the Legendre transformation that the image  $A\cdot\cir{q_{i,j,k}}$ of the circle $\cir{q_{i,j,k}}$ has ramification order $\beta_{i,j,k}-\alpha_{i,j,k}$ and slope $\alpha_{i,j,k}/(\beta_{i,j,k}-\alpha_{i,j,k})$. 
This generalizes the different readings of the diagram considered in \cite{boalch2008irregular,boalch2012simply,boalch2016global}. The situation considered in those works is the simply laced case where all circles at infinity are of the form 
\[ q=\lambda z^2 +\mu z,
\]
with $\lambda,\mu\in\C$. In particular $N_\infty$ is empty. In this situation, each circle can be sent by a symplectic transformation to a tame circle at finite distance.

\section{Dimension of the wild character variety}

\subsection{The wild character variety}

Let $(E,\nabla)$ be an irreducible algebraic connection on a Zariski open subset of $\Sigma$, and let $(\bm\Theta, \bm\Ccal)$ be its formal data. The connection defines a point in the wild character variety $\MB(E,\nabla)=\MB(\bm\Theta, \bm\Ccal)$. Let us quickly review the quasi-Hamiltonian description of $\MB(E,\nabla)$ in the general case with several singularities.

Let us set $G:=\GL_n(\C)$ where $n$ is the rank of $(E,\nabla)$, and keeping the previous notations let $a_1,\dots, a_m\in \Sigma$ be its singular points. Let $V^0\to \Ical$ be the (non-modified) local system on $\Ical$ giving rise to $V$. For $k=1,\dots,m$, we consider the restriction $V^0_{a_k}\to \partial_{a_k}$ of the local system system $V^0$. Let us choose a direction $d\in\partial_{a_k}$, and denote by $\rho_k\in \GL(V^0_d)$ the monodromy of $V^0_{a_k}$. As before, we denote by $\cir{q^{(k)}_1},\dots \cir{q^{(k)}_{r_k}}\in \pi_0(\Ical_{a_k})$ the active circles at $a_k$, $n^{(k)}_i\in \N$ their respective multiplicities, and $\beta^{(k)}_j$ their respective ramification orders. Let $H_k:=\GrAut(V^0_{a_k})\simeq\prod_{i=1}^{r_k} \GL_{n^{(k)}_i}(\C)^{b^{(k)}_i}$. The monodromy $\rho_k$ is an element of the $H_k$-torsor $H_k(\partial_{a_k})$ (see \cite{boalch2015twisted}). Explicitly, an element of $H_k(\partial_{a_k})$ is of the form
\[ h^{(k)}=\diag(h^{(k)}_1,\dots,h^{(k)}_{r_k}),
\]
where $h^{(k)}_i$ is a square diagonal block matrix of size $\beta^{(k)}_i$ of the form
\[ h^{(k)}_i=
\begin{pmatrix}
0 & \hdots & 0 &  * \\
 *& &  & 0\\
 & \ddots & & \vdots\\
 & & * & 0
\end{pmatrix},
 \]
with each block in $\GL_{n^{(k)}_i}(\C)$. 

The wild character variety $\MB(E,\nabla)=\mathcal{M}_B(\bm\Theta,\bm{\Ccal})$ is obtained as a symplectic reduction of the twisted quasi-Hamiltonian manifold 
\[  \Hom_{\mathbb{S}}(V)\simeq \Acal(V^0_{a_1}) \circledast \dots \circledast \Acal(V^0_{a_m})\sslash G, 
\]
where $\circledast$ is the quasi-Hamiltonian fusion operation. Here, each piece $\Acal(V^0_{a_k})$ is a twisted quasi-Hamiltonian $H_k\times G$-space: 
\[ \Acal(V^0_{a_k})=H(\partial_{a_k})\times G \times \prod_{d\in \A_{a_k}}\Sto_d,
\]
where $\A_{a_k}$ is the set of singular directions of the connection at $a_k$, and $\Sto_d$ is the Stokes group associated to the singular direction $d\in \A_{a_k}$. Hence,
$\Hom_{\mathbb{S}}(E,\nabla)$ is a twisted quasi-Hamiltonian $\mathbf{H}\times G$-space, where 
$\mathbf{H}=H_1\times \dots\times H_m$. 

The formal monodromies $\rho(a_k)\in H(\partial_{a_k})$ determine twisted conjugacy classes $\mathcal{C}(\partial_{a_k})\subset H(\partial_{a_k})$ which do not depend on the choice of direction $d\in \partial_{a_k}$ and the choice of isomorphism $V^0_{d}\simeq \C^{\sum_j n^{(k)}_i \beta^{(k)}_i}$. Finally, the wild character variety is the twisted quasi-Hamiltonian reduction of $\Hom_{\mathbb{S}}(E,\nabla)$ at the twisted conjugacy class $\bm{\mathcal{C}}:=\prod_{k=1}^m \mathcal{C}(\partial_{a_k}) $, i.e. 
\begin{equation}
\mathcal{M}_B(\bm\Theta,\bm{\Ccal})=\Hom_{\mathbb{S}}(E,\nabla)\sslash_{\bm{\mathcal{C}}}  \mathbf{H}.
\end{equation}

\subsection{Formula for the dimension}

If $\Gamma$ is a diagram, possibly with loops, negative edges or negative loops, $N$ its set of vertices, and $B\in M_{N\times N}(\mathbb{Z})$ its adjacency matrix, the Cartan matrix of$\Gamma$ is defined by
\[ C=2 \Id-B.
\]

We have the following result: 

\begin{theorem}
\label{theorem_dimension}
Let $(E,\nabla)$ be an irreducible algebraic connection on a Zariski open subset of $\Sigma$. Let $N$ be the set of vertices of the diagram $\Gamma(E,\nabla)$, and $(\cdot,\cdot)$ the bilinear form on $\mathbb R^N$ defined by the Cartan matrix of $\Gamma(E,\nabla)$. The dimension of the wild character variety $\mathcal{M}_B(E,\nabla)$ is given by
\begin{equation}
\dim \mathcal{M}_B(E,\nabla)=2-(\mathbf{d},\mathbf{d}),
\end{equation}
where $\mathbf{d}\in \Z_{\geq 0}^N$ is the dimension vector of $\Gamma(E,\nabla)$ defined by any choice of marking of $(E,\nabla)$. 
\end{theorem}

The proof is given in appendix C.

\begin{corollary}
If $(\bm\Theta,\bm{\Ccal})$ are the formal data of an irreducible connection $(E,\nabla)$, and $A\in \SL_2(\C)$, then the character varieties $\MB(\bm\Theta,\bm{\Ccal})$ and $\MB(A\cdot (\bm\Theta,\bm{\Ccal}))$ have the same dimension. 
\end{corollary}

\begin{proof}
By construction of the diagram $V$ and $A\cdot (\bm\Theta,\bm{\Ccal})$ have the same diagram $\Gamma$. Therefore we have $\dim \MB(\bm\Theta,\bm{\Ccal})=2-(\mathbf{d},\mathbf{d})=\dim \MB(A\cdot (\bm\Theta,\bm{\Ccal}))$. 
\end{proof}

\section{Examples of diagrams}

Let us now give a few examples of diagrams. We first review a few known cases to see how we recover the diagrams from our more general approach, then discuss some new cases, in particular cases arising from Painlevé-type equations.

\subsection*{Complete bipartite case} Let us consider the case of a connection with  a second order poles at infinity, together with simple poles $a_1,\dots, a_m$ at finite distance, with a global modified irregular class of the form
\[
\bm{\breve\Theta}=k_1 \cir{\alpha_1 z}_\infty+  \dots + k_n\cir{\alpha_n z}_\infty + l_1\cir{0}_{a_1}+\dots +l_m\cir{0}_{a_m},
\]
with $k_i\in \N$, $l_j\in \N$, and $\alpha_i\in \C$ pairwise distinct, $a_i\in \C$ pairwise distinct. There are $n$ active circles of slope 1 at infinity. The diagram is the following: 

\begin{center}
\begin{tikzpicture}
\tikzstyle{vertex}=[circle,fill=black,minimum size=5pt,inner sep=0pt]
\foreach \i in {0,...,3}
\foreach \j in {0,...,2}
{\node[vertex] (A\i) at (0,-1.5+\i) {};
\node[vertex] (B\j) at (3,-1+\j) {};
\draw (A\i)--(B\j);}
\draw (0,1.5) node[left] {$\cir{0}_{a_1}$};
\draw (0,-1.5) node[left] {$\cir{0}_{a_m}$};
\draw (3,1) node[right] {$\cir{\alpha_1 x}_\infty$};
\draw (3,-1) node[right] {$\cir{\alpha_n x}_\infty$};
\end{tikzpicture}
\end{center}

It is a complete bipartite graph, one part having $n$ vertices and one part having $m$ vertices. If $n=1$ or $m=1$ this is the star-shaped case.

\subsection*{Simply laced case} 
An important case is the one where there is only one unramified irregular pole at infinity of order 3, together with simple poles at finite distance. This is the case extensively studied in \cite{boalch2008irregular,boalch2012simply,boalch2016global}, giving rise in those works to simply-laced supernova graphs, with core any complete multipartite graph. We can check that our definition gives the same diagrams.

In this case, the circles at infinity are unramified and have slope $\leq 2$, i.e. they are of the form $\cir{q_{i,j}}_\infty$ with $q_{i,j}=\lambda_i z^2+ \mu_{i,j} z$, for $i=1,\dots k$, $j=1,\dots s_i$, where the coefficients $\lambda_1,\dots,\lambda_k$ are all different, as well as $\mu_{i,1},\dots,\mu_{i,s_i}$ for all $i$. The simple poles correspond to tame circles $\cir{0}_{a_1},\dots, \cir{0}_{a_t}$, where $a_1,\dots a_t\in \C$. Let $I_0:=\{ \cir{0}_{a_1},\dots, \cir{0}_{a_t}\}$ and for $i=1,\dots k$,  $I_i:=\{\cir{\lambda_i z^2 +\mu_{i,1} z},\dots,\cir{\lambda_i z^2 +\mu_{i,s_i} z}\}$. The sets $I_0,\dots, I_k$ constitute a partition of the set of active circles. The core diagram has the following structure two active circles are either linked by no edge in the diagram when they belong to the same set $I_k$, otherwise they are linked by exactly one edge. The core diagram that we obtain is thus a $k+1$-partite graph. The diagram coincides with the one considered in \cite{boalch2012simply,boalch2016global}. 

The standard rank two Lax representations of the Painlevé V and Painlevé IV equation fit into this setting. The representation for Painlevé V has one irregular singularity at infinity of order 2, together with two simple poles at finite distance. The modified irregular class is of the form
\[
\bm{\breve\Theta}=\cir{\alpha z}_\infty + \cir{\beta z }_\infty + \cir{0}_a +\cir{0}_b,
\]
with $\alpha\neq \beta\in \C$, $a\neq b\in \C$. This gives the following diagram:

\begin{center}
\begin{tikzpicture}[scale=0.6]
\tikzstyle{vertex}=[circle,fill=black,minimum size=8pt,inner sep=0pt]
\node[vertex] (A) at (1,1){} ;
\node[vertex] (B) at (-1,1){} ; 
\node[vertex] (C) at (-1,-1){};
\node[vertex] (D) at (1,-1){};
\draw (A)--(B)--(C)--(D)--(A);
\end{tikzpicture}
\end{center}

The standard representation for Painlevé IV has one irregular singularity at infinity of order 3, together with one simple pole at finite distance. 
The modified irregular class is of the form
\[
\bm{\breve\Theta}=\cir{\alpha z^2}_\infty + \cir{\beta z^2}_\infty + \cir{0}_a ,
\]
with $\alpha\neq \beta\in \C$, $a \in \C$. This gives the following diagram:

\begin{center}
 \begin{tikzpicture}[scale=0.6]
\tikzstyle{vertex}=[circle,fill=black,minimum size=8pt,inner sep=0pt]
\node[vertex] (A) at (-30:1.3){} ;
\node[vertex] (B) at (90:1.3){} ; 
\node[vertex] (C) at (210:1.3){};
\draw (A)--(B);
\draw (A)--(C);
\draw (B)--(C);
\end{tikzpicture}
\end{center}

\subsection*{Other supernova graphs} More general supernova graphs correspond to the case where there is one irregular singularity at infinity with all active circles being unramified, together with simple poles at infinity. This is the case considered in \cite[appendix C]{boalch2008irregular}  and \cite{hiroe2014moduli}. If the active circles at infinity are $\cir{q_1},\dots,\cir{q_r}$ we have $B_{ij}=\deg(q_i-q_j)-1$. 

The standard Painlevé II Lax pair fits into this framework. The (modified or not) irregular class is of the form $\bm\Theta=\Theta_\infty=\cir{\alpha z^3}_\infty+\cir{\beta z^3}_{\infty}$ with $\alpha\neq \beta$. The diagram is

\begin{center}
 \begin{tikzpicture}[scale=0.6]
\tikzstyle{vertex}=[circle,fill=black,minimum size=8pt,inner sep=0pt]
\node[vertex] (A) at (0,0){} ;
\node[vertex] (B) at (2,0){} ; 
\draw[double distance=2 pt] (A)--(B);
\end{tikzpicture}
\end{center}

\subsection*{Painlevé III} The standard Painlevé III Lax pair corresponds to a rank 2 connection with 2 irregular singularities. The active circles are two irregular circles of slope 1 at infinity and 2 irregular circles of slope 1 at $z=0$, say $\cir{\lambda_1 z}_\infty$, $\cir{\lambda_2 z}_\infty$ $\cir{\mu_1 z^{-1}}_0$, $\cir{\mu_2 z^{-1}}_0$, with $\lambda_1\neq \lambda_2$, $\mu_1\neq \mu_2$. With no loss of generality, up to applying a twist by a rank one connection on the trivial bundle on $\Pbb^1$, we may assume that $\mu_2=0$, so that we have the active circle $\cir{0}_0$, and that the tame circle has trivial formal monodromy in the non-reduced formal local system, so that it has multiplicity zero in the associated modified formal local system. The diagram associated to the connection is given by 
\begin{center}
\begin{tikzpicture}[scale=0.6]
\tikzstyle{vertex}=[circle,fill=black,minimum size=8pt,inner sep=0pt]
\tikzstyle{vertex_nofill}=[draw,circle,minimum size=8pt,inner sep=0pt]
\node[vertex] (A) at (0,0){} ;
\node[vertex] (B) at (2,0){} ; 
\node[vertex] (C) at (4,0){};
\draw[double distance=2 pt] (A)--(B)--(C);
\draw[dashed] (B) to[out=45, in=0] (2,1.3) to[out=180,in=135] (B);
\draw (-0.3,0) node[left]{$\cir{\lambda_1 z}_\infty$};
\draw (2,0) node[below]{$\cir{\mu_1 z^{-1}}_0$};
\draw (4.3,0) node[right]{$\cir{\lambda_2 z}_\infty$};
\end{tikzpicture}
\end{center}
with dimension vector $\mathbf{d}=(1,1,1)$. This is exactly the same diagram as the one in \cite{boalch2020diagrams}.

\subsection*{Degenerate Painlevé V}

The representation of Painlevé III known as degenerate Painlevé V corresponds as for Painlevé V to a rank 2 connection with one order two pole at infinity and two simple poles at finite distance \cite[p. 34]{joshi2007linearization}. The difference is that the leading term of the connection is nilpotent (this is the meaning of the word degenerate in this context) which implies there is just one ramified active circle at infinity of slope $1/2$. The modified irregular class is of the form:

\[
\bm{\breve\Theta}=\cir{\alpha z^{1/2}}_\infty +\cir{0}_a+ \cir{0}_b,
\]
and this gives the diagram

\begin{center}
\begin{tikzpicture}[scale=0.6]
\tikzstyle{vertex}=[circle,fill=black,minimum size=8pt,inner sep=0pt]
\tikzstyle{vertex_blue}=[circle,fill=blue,minimum size=8pt,inner sep=0pt]
\node[vertex] (A) at (0,0){} ;
\node[vertex] (B) at (2,0){} ; 
\node[vertex] (C) at (4,0){};
\draw[double distance=2 pt] (A)--(B)--(C);
\draw[dashed] (B) to[out=45, in=0] (2,1.3) to[out=180,in=135] (B);
\draw (-0.3,0) node[left]{$\cir{0}_{a}$};
\draw (2,0) node[below]{$\cir{\alpha z^{1/2}}_\infty$};
\draw (4.3,0) node[right]{$\cir{0}_{a}$};
\end{tikzpicture}
\end{center}
computed in \cite{boalch2020diagrams}. This is the same diagram as for  the standard Painlevé III Lax pair. The two irregular classes indeed correspond to different representations of the diagram.

\subsection*{Degenerate Painlevé III}

The degenerate Painlevé III system admits, as for the standard Painlevé III representation, a representation corresponding to a rank 2 connection with two order two poles \cite{ohyama2006studiesV}. The difference is that the leading term of the connection matrix at one of the poles is nilpotent. Again, this corresponds to having one active circle with slope $1/2$. The modified irregular class has the following form:
\[
\bm{\breve\Theta}=\cir{z^{1/2}}_\infty+\cir{\alpha z^{-1}}_0+\cir{\beta z^{-1}}_0
\]
with $\alpha,\beta\in \C$, $\alpha\neq \beta$.

Up to performing a twist at zero, we may assume that $\beta=0$, so that we have the tame circle at 0, and that this circle has trivial formal monodromy, so that the corresponding piece of the modified formal local system has multiplicity 0. This leads to the following diagram, with one negative loop at each vertex, and 4 edges between the two vertices. 

\begin{center}
\begin{tikzpicture}[scale=0.8]
\tikzstyle{vertex}=[circle,fill=black,minimum size=8pt,inner sep=0pt]
\node[vertex] (A) at (-2,0){} ; 
\node[vertex] (C) at (2,0){};
\draw (C)-- node[midway,above]{$4$} (A);
\draw[dashed] (A) to[out=-135, in=-90] (-3.4,0) to[out=90,in=135] (A);
\draw[dashed] (C) to[out=-45, in=-90] (3.4,0) to[out=90,in=45] (C);
\draw (-2.3,-0.5) node[below] {$\cir{\alpha z^{-1}}_0$};
\draw (2.3,-0.5) node[below] {$\cir{z^{1/2}}_\infty$};
\draw (4,0) node {$-1$};
\draw (-4,0) node {$-1$};
\end{tikzpicture}
\end{center}

The Cartan matrix is 
\[
C=\begin{pmatrix}
4 & -4\\
-4 & 4
\end{pmatrix}.
\]

\subsection*{Doubly degenerate Painlevé III}

The doubly degenerate Painlevé III system admits a Lax representation corresponding to two irregular singularities of order two, each with nilpotent leading term leading to an active circle of slope $1/2$ \cite{ohyama2006studiesV}. The modified irregular class is of the form

\[
\bm{\breve\Theta}=\cir{\alpha z^{1/2}}_\infty +\cir{\beta z^{-1/2}}_0
\]
with $\alpha,\beta$ nonzero complex numbers. This gives the following diagram, where the number on the edges correspond to their multiplicities.

\begin{center}
\begin{tikzpicture}[scale=0.8]
\tikzstyle{vertex}=[circle,fill=black,minimum size=8pt,inner sep=0pt]
\node[vertex] (A) at (-2,0){} ; 
\node[vertex] (C) at (2,0){};
\draw (C)-- node[midway,above]{$6$} (A);
\draw[dashed] (A) to[out=-135, in=-90] (-3.4,0) to[out=90,in=135] (A);
\draw[dashed] (C) to[out=-45, in=-90] (3.4,0) to[out=90,in=45] (C);
\draw (-2.3,-0.5) node[below] {$\cir{z^{-1/2}}_0$};
\draw (2.3,-0.5) node[below] {$\cir{z^{1/2}}_\infty$};
\draw (4,0) node {$-1$};
\draw (-4,0) node {$-3$};
\end{tikzpicture}
\end{center}

The Cartan matrix is 
\[
C=\begin{pmatrix}
 8 & -6\\
 -6& 4
\end{pmatrix},
\]

\subsection*{Flaschka-Newell Lax pair for Painlevé II} The Flaschka-Newell Lax pair for Painlevé II \cite{flaschla1980monodromy}, also known as degenerate Painlevé IV, corresponds to a rank 2 connection with one irregular singularity at infinity with one active circle of slope $3/2$, together with one simple pole at finite distance. The modified irregular class is of the form
\[
\bm{\breve\Theta}=\cir{\alpha z^{3/2}}_\infty + \cir{0}_0.
\]
with $\alpha\neq 0$. This gives the diagram \begin{center}
 \begin{tikzpicture}[scale=0.6]
\tikzstyle{vertex}=[circle,fill=black,minimum size=8pt,inner sep=0pt]
\node[vertex] (A) at (0,0){} ;
\node[vertex] (B) at (2,0){} ; 
\draw[double distance=2 pt] (A)--(B);
\end{tikzpicture}
\end{center}
that is the same diagram as for Painlevé II, as found in \cite{boalch2020diagrams}. Again, the two irregular classes correspond to different fundamental representations of the diagram.

\subsection*{Painlevé I} The standard Lax pair for the Painlevé I equation corresponds to a rank 2 connection with just one irregular singularity at infinity, with one active circle of slope $5/2$. The modified irregular class is of the form
\[
\bm{\breve\Theta}=\cir{\alpha z^{5/2}}_\infty,
\]
with $\alpha\neq 0$. The corresponding diagram has one vertex and one loop. 

\begin{center}
\begin{tikzpicture}[scale=0.6]
\tikzstyle{vertex}=[circle,fill=black,minimum size=8pt,inner sep=0pt]
\node[vertex] (A) at (0,0){} ;
\draw (A) to[out=45, in=0] (0,1.3) to[out=180,in=135] (A);
\end{tikzpicture}
\end{center}

\subsection*{H3 surfaces}

The  Painlevé equations correspond to the simplest examples of non-trivial wild character varieties, since their moduli spaces are 2-dimensional. The corresponding nonabelian Hodge spaces $\mathfrak{M}$
are thus complete hyperkähler manifolds of real dimension four, so are examples of H3 surfaces in the terminology of \cite{boalch2018wild} to designate the two-dimensional wild character varieties. Apart from the cases corresponding to Painlevé equations, there are 3 other known H3 surfaces whose standard representation correspond to a fuchsian connection with 3 regular singularities. Since they have no deformation parameters, they do not give rise to an isomonodromy system.
Thanks to our more general theory, we are now able to associate a diagram to all H3 surfaces. They are represented in fig. \ref{table_H3}. The reader may check that in each case we have $2-(\mathbf{d},\mathbf{d})=2$ as expected. Notice that for Painlevé II, IV, V, VI, the diagram is exactly the affine Dynkin diagram corresponding to its Okamoto symmetries \cite{okamoto1986studies,okamoto1987studiesII,okamoto1986studiesIII,okamoto1987studiesIV,ohyama2006studiesV},

\begin{figure}[h]
\begin{center}
\begin{tabular}{|c|c|}
\hline
Space & Diagram \\
\hline
$\widehat{E}_8$ & \begin{tikzpicture}[scale=0.5]
\tikzstyle{vertex}=[circle,fill=black,minimum size=8pt,inner sep=0pt]
\node[vertex] (A) at (-2,0){};
\node[vertex] (B) at (0,0){};
\node[vertex] (C) at (2,0){};
\node[vertex] (D) at (4,0){};
\node[vertex] (X) at (2,2){}; 
\node[vertex] (E) at (6,0){};
\node[vertex] (F) at (8,0){};
\node[vertex] (G) at (10,0){};
\node[vertex] (H) at (12,0){};
\draw (A)--(B)--(C)--(D)--(E)--(F)--(G)--(H);
\draw (C)--(X);
\draw (A)++(0,-0.3) node[below left]{$2$};
\draw (B)++(0,-0.3) node[below]{$4$};
\draw (C)++(0,-0.3) node[below]{$6$};
\draw (X)++ (-0.3,0) node[left]{$3$};
\draw (D)++(0,-0.3) node[below]{$5$};
\draw (E)++(0,-0.3) node[below]{$4$};
\draw (F)++(0,-0.3) node[below]{$3$};
\draw (G)++(0,-0.3) node[below]{$2$};
\draw (H)++(0,-0.3) node[below]{$1$};
\end{tikzpicture} \\
$\widehat{E}_7$ & \begin{tikzpicture}[scale=0.5]
\tikzstyle{vertex}=[circle,fill=black,minimum size=8pt,inner sep=0pt]
\node[vertex] (A) at (-2,0){};
\node[vertex] (B) at (0,0){};
\node[vertex] (C) at (2,0){};
\node[vertex] (D) at (4,0){};
\node[vertex] (X) at (4,2){}; 
\node[vertex] (E) at (6,0){};
\node[vertex] (F) at (8,0){};
\node[vertex] (G) at (10,0){};
\draw (A)--(B)--(C)--(D)--(E)--(F)--(G);
\draw (D)--(X);
\draw (A)++(0,-0.3) node[below left]{$1$};
\draw (B)++(0,-0.3) node[below]{$2$};
\draw (C)++(0,-0.3) node[below]{$3$};
\draw (X)++ (-0.3,0) node[left]{$2$};
\draw (D)++(0,-0.3) node[below]{$4$};
\draw (E)++(0,-0.3) node[below]{$3$};
\draw (F)++(0,-0.3) node[below]{$2$};
\draw (G)++(0,-0.3) node[below]{$1$};
\end{tikzpicture}\\
$\widehat{E}_6$ & \begin{tikzpicture}[scale=0.4]
\tikzstyle{vertex}=[circle,fill=black,minimum size=8pt,inner sep=0pt]
\node[vertex] (A) at (-2,0){};
\node[vertex] (B) at (0,0){};
\node[vertex] (C) at (2,0){};
\node[vertex] (X) at (2,2){}; 
\node[vertex] (Y) at (2,4){};
\node[vertex] (D) at (4,0){};
\node[vertex] (E) at (6,0){};
\draw (A)--(B)--(C)--(D)--(E);
\draw (C)--(X)--(Y);
\draw (A)++(0,-0.3) node[below left]{$1$};
\draw (B)++(0,-0.3) node[below]{$2$};
\draw (C)++(0,-0.3) node[below]{$3$};
\draw (X)++ (-0.3,0) node[left]{$2$};
\draw (Y)++ (-0.3,0) node[left]{$1$};
\draw (D)++(0,-0.3) node[below]{$2$};
\draw (E)++(0,-0.3) node[below]{$1$};
\end{tikzpicture}\\
$\widehat{D}_4$ & \begin{tikzpicture}[scale=0.35]
\tikzstyle{vertex}=[circle,fill=black,minimum size=8pt,inner sep=0pt]
\node[vertex] (A) at (0,0){} ;
\node[vertex] (B) at (0,2){} ; 
\node[vertex] (C) at (0,-2){};
\node[vertex] (D) at (2,0){};
\node[vertex] (E) at (-2,0){};
\draw (A)--(B);
\draw (A)--(C);
\draw (A)--(D);
\draw (A)--(E);
\draw (A)node[below left]{$2$};
\end{tikzpicture} \\
$\widehat{A}_3=\hat{D}_3$ & \begin{tikzpicture}[scale=0.6]
\tikzstyle{vertex}=[circle,fill=black,minimum size=8pt,inner sep=0pt]
\node[vertex] (A) at (1,1){} ;
\node[vertex] (B) at (-1,1){} ; 
\node[vertex] (C) at (-1,-1){};
\node[vertex] (D) at (1,-1){};
\draw (A)--(B)--(C)--(D)--(A);
\end{tikzpicture}\\
$\widehat{D}_2$ & \begin{tikzpicture}[scale=0.6]
\tikzstyle{vertex}=[circle,fill=black,minimum size=8pt,inner sep=0pt]
\node[vertex] (A) at (0,0){} ;
\node[vertex] (B) at (2,0){} ; 
\node[vertex] (C) at (4,0){};
\draw[double distance=2 pt] (A)--(B)--(C);
\draw[dashed] (B) to[out=45, in=0] (2,1.3) to[out=180,in=135] (B);
\end{tikzpicture} \\
$\widehat{D}_1$ & \begin{tikzpicture}[scale=0.5]
\tikzstyle{vertex}=[circle,fill=black,minimum size=8pt,inner sep=0pt]
\node[vertex] (A) at (-2,0){} ; 
\node[vertex] (C) at (2,0){};
\draw (C)-- node[midway,above]{$4$} (A);
\draw[dashed] (A) to[out=-135, in=-90] (-3.4,0) to[out=90,in=135] (A);
\draw[dashed] (C) to[out=-45, in=-90] (3.4,0) to[out=90,in=45] (C);
\draw (4,0) node {$-1$};
\draw (-4,0) node {$-1$};
\end{tikzpicture}\\
$\widehat{D}_0$ & \begin{tikzpicture}[scale=0.5]
\tikzstyle{vertex}=[circle,fill=black,minimum size=8pt,inner sep=0pt]
\node[vertex] (A) at (-2,0){} ; 
\node[vertex] (C) at (2,0){};
\draw (C)-- node[midway,above]{$6$} (A);
\draw[dashed] (A) to[out=-135, in=-90] (-3.4,0) to[out=90,in=135] (A);
\draw[dashed] (C) to[out=-45, in=-90] (3.4,0) to[out=90,in=45] (C);
\draw (4,0) node {$-1$};
\draw (-4,0) node {$-3$};
\end{tikzpicture}\\
$\widehat{A}_2$ & \begin{tikzpicture}[scale=0.6]
\tikzstyle{vertex}=[circle,fill=black,minimum size=8pt,inner sep=0pt]
\node[vertex] (A) at (-30:1.3){} ;
\node[vertex] (B) at (90:1.3){} ; 
\node[vertex] (C) at (210:1.3){};
\draw (A)--(B);
\draw (A)--(C);
\draw (B)--(C);
\end{tikzpicture} \\
$\widehat{A}_1$ & \begin{tikzpicture}[scale=0.6]
\tikzstyle{vertex}=[circle,fill=black,minimum size=8pt,inner sep=0pt]
\node[vertex] (A) at (0,0){} ;
\node[vertex] (B) at (2,0){} ; 
\draw[double distance=2 pt] (A)--(B);
\end{tikzpicture} \\
$\widehat{A}_0$ & \begin{tikzpicture}[scale=0.6]
\tikzstyle{vertex}=[circle,fill=black,minimum size=8pt,inner sep=0pt]
\node[vertex] (A) at (0,0){} ;
\draw (A) to[out=45, in=0] (0,1.3) to[out=180,in=135] (A);
\end{tikzpicture}\\
\hline
\end{tabular}
\end{center}
\caption{Diagrams associated to all known H3 surfaces. The names of the surfaces are as in \cite{boalch2018wild}. All unspecified multiplicities are equal to 1.}
\label{table_H3}
\end{figure}

\subsection*{h-Painlevé systems}  
It is possible to define a notion of h-Painlevé systems  $hP_k^{(n)}$, where h stands for higher or hyperbolic or Hilbert. Higher Painlevé systems $hP_k^{(n)}$ were introduced by Boalch \cite{boalch2008irregular,boalch2012simply}: for each integer $n$ there is a rank $2n$ Lax representation giving rise to a $2n$-dimensional higher Painlevé moduli space. All higher Painlevé systems of a given number have the same diagram, but the dimension vector is a function of $n$. More precisely, the diagram associated to a given higher Painlevé system is obtained by taking the original diagram with all multiplicities scaled by $n$, then adding a leg of length one, with its end vertex having multiplicity 1, as in fig. \ref{fig_higher_painleve_boalch} for higher Painlevé IV,V,VI. The higher Painlevé diagrams are hyperbolic Dynkin diagrams since they are obtained by adding a vertex to an affine Dynkin diagram. In this sense they can be seen as the next simplest examples after the affine case. The higher Painlevé  moduli spaces are (conjecturally) related to Hilbert schemes on $n$ points on the corresponding H3 surface (see \cite{boalch2012simply}). The same recipe for Painlevé III yields the $hP_{III}^{(n)}$ diagram.

\begin{figure}[h]
\begin{center}
\begin{tikzpicture}[scale=0.9]
\begin{scope}[scale=0.45]
\tikzstyle{vertex}=[circle,fill=black,minimum size=6pt,inner sep=0pt]
\node[vertex] (A) at (0,0){} ;
\node[vertex] (B) at (0,2){} ; 
\node[vertex] (C) at (0,-2){};
\node[vertex] (D) at (2,0){};
\node[vertex] (E) at (-2,0){};
\node[vertex] (F) at (4,0){};
\draw (A)--(B);
\draw (A)--(C);
\draw (A)--(D);
\draw (A)--(E)--(F);
\draw (A) node[below left]{$2n$};
\draw (B) node[above]{$n$};
\draw (C) node[below]{$n$};
\draw (D) node[below]{$n$};
\draw (E) node[left]{$n$};
\draw (F) node[right]{$1$};
\end{scope}
\begin{scope}[shift={(3.6,0)},scale=.45]
\tikzstyle{vertex}=[circle,fill=black,minimum size=6pt,inner sep=0pt]
\node[vertex] (A) at (0,0){} ;
\node[vertex] (B) at (4,0){} ; 
\node[vertex] (C) at (2,2){};
\node[vertex] (D) at (2,-2){};
\node[vertex] (E) at (6,0){};
\draw (B)--(E);
\draw (A)--(C)--(B)--(D)--(A);
\draw (A) node[left]{$n$};
\draw (B) node[below right]{$n$};
\draw (2.3,2) node[right]{$n$};
\draw (2.3,-2) node[right]{$n$};
\draw (E) node[right]{$1$};
\end{scope}
\begin{scope}[shift={(8,0)},scale=.45]
\tikzstyle{vertex}=[circle,fill=black,minimum size=6pt,inner sep=0pt]
\node[vertex] (A) at (2,0){} ;
\node[vertex] (B) at (120:2){} ; 
\node[vertex] (C) at (-120:2){};
\node[vertex] (D) at (4,0){};
\draw (A)--(B)--(C)--(A);
\draw (A)--(D);
\draw (A) node[below right]{$n$};
\draw (B) node[above]{$n$};
\draw (C) node[below]{$n$};
\draw (D) node[right]{$1$};
\end{scope}
\begin{scope}[shift={(11,0)},scale=.5]
\tikzstyle{vertex}=[circle,fill=black,minimum size=6pt,inner sep=0pt]
\tikzstyle{vertex_blue}=[circle,fill=blue,minimum size=6pt,inner sep=0pt]
\node[vertex] (A) at (0,0){} ;
\node[vertex] (B) at (2,0){} ; 
\node[vertex] (C) at (4,0){};
\node[vertex] (D) at (6,0){};
\draw[double distance=2 pt] (A)--(B)--(C);
\draw (C)--(D);
\draw[dashed] (B) to[out=45, in=0] (2,1.3) to[out=180,in=135] (B);
\draw (0,-0.2) node[below]{$n$};
\draw (2,-0.2) node[below]{$n$};
\draw (4,-0.2) node[below]{$n$};
\draw (6,-0.2) node[below]{$1$};
\end{scope}
\end{tikzpicture}
\end{center}
\caption{Diagrams for higher Painlevé systems $hP_{VI}^{(n)}$, $hP_{V}^{(n)}$, $hP_{IV}^{(n)}$, and $hP_{III}^{(n)}$.}
\label{fig_higher_painleve_boalch}
\end{figure}

\subsection*{Some 4-dimensional cases}

Another point of view on higher dimensional isomonodromy systems consists in looking at degenerations of an equation coming via isomonodromy from a fuchsian connection. This approach is used in \cite{kawakami2018degeneration} to list some representations of some 4-dimensional isomonodromy systems: they are defined there as the systems coming via degeneration from the isomonodromic deformation equations associated to connections with regular singularities giving rise to 4-dimensional moduli spaces. Whereas for the usual 2-dimensional Painlevé equations there is only one possible choice of formal data for fuchsian connections with non-trivial admissible deformations and 2-dimensional moduli spaces, in the 4-dimensional case there are 4 possible choices. The full degeneration scheme is represented at \cite[p. 40]{kawakami2018degeneration}. Each of the 4 fuchsian formal data gives rise via degeneration to a family of 4-dimensional Painlevé type equations. 

The two points of view on higher-dimensional isomonodromy systems are actually not independent. The elements of the fourth degeneration family of 4-dimensional Painlevé-type equations in the degeneration scheme of \cite{kawakami2018degeneration}  (called there matrix Painlevé systems) correspond to the higher Painlevé equations in the sense of \cite{boalch2012simply} in the 4-dimensional case, and the master diagram $hP_6^{(2)}$ at the top of the coalescence cascade first appeared in the context of Painlev\'e systems in the list of 4-dimensional hyperbolic examples in \cite[p. 12]{boalch2008irregular}. This notion of higher Painlevé systems is however more general than just this 4-dimensional case. 

In figure \ref{our_diagrams_4d_painleve} are drawn the diagrams associated by our approach to the 4-dimensional isomonodromy systems listed in \cite{kawakami2018degeneration} whose Lax representations feature several irregular singularities. These diagrams are to be contrasted with the shapes in \cite[p.235]{hiroe2013classification}. Remarkably, for all instances in the degeneration scheme of \cite{kawakami2018degeneration} where there are different Lax representations for the same Painlevé-type system, they correspond to different readings of the same diagram. 

\begin{figure}[h]
\begin{center}
\begin{tabular}{|m{5cm}|m{5cm}|m{6cm}|}
  \hline
  Name in \cite{kawakami2018degeneration}& Mod. irreg. class & Diagram \\
  \hline
  $H_{\text{Gar}}^{2+2+1}$ & $\cir{\alpha z}_\infty+\cir{\beta z}_\infty + \cir{\gamma z^{-1}}_0 + \cir{0}_1$ & \begin{tikzpicture}[scale=0.4]
\tikzstyle{vertex}=[circle,fill=black,minimum size=8pt,inner sep=0pt]
\tikzstyle{vertex_blue}=[circle,fill=blue,minimum size=8pt,inner sep=0pt]
\node[vertex] (A) at (0,-2){} ;
\node[vertex] (C) at (4,-2){};
\node[vertex] (D) at (2,0){};
\node[vertex] (E) at (2,-4){};
\draw[double distance=2 pt] (A)--(D);
\draw[double distance=2 pt] (C)--(D);
\draw (A)--(E)--(C);
\draw[dashed] (D) to[out=45, in=0] (2,1.3) to[out=180,in=135] (D);
\draw (2.2,0) node[right]{$1$};
\draw (-0.3,-2) node[left]{$1$};
\draw (2,-4.2) node[below]{$1$};
\draw (4.3,-2) node[right]{$1$};
\end{tikzpicture} \\
   $H_{\text{Gar}}^{3+2}$  & $\cir{\alpha z^2}_\infty +\cir{\beta z^2}_\infty +\cir{\gamma z^{-1}}_0$ &\begin{tikzpicture}[scale=0.4]
\tikzstyle{vertex}=[circle,fill=black,minimum size=8pt,inner sep=0pt]
\node[vertex] (A) at (0,-2){} ;
\node[vertex] (C) at (4,-2){};
\node[vertex] (D) at (2,0){};
\draw (A)--(C);
\draw[double distance=2 pt] (A)--(D);
\draw[double distance=2 pt] (C)--(D);
\draw[dashed] (D) to[out=45, in=0] (2,1.3) to[out=180,in=135] (D);
\draw (2.2,0) node[right]{$1$};
\draw (-0.3,-2) node[left]{$1$};
\draw (4.3,-2) node[right]{$1$};
\end{tikzpicture} \\
  $H_{\text{FS}}^{A_3}$ &  $2\cir{\alpha z}_\infty+\cir{\beta z}_\infty + 2\cir{\gamma z^{-1}}_0 $ & \begin{tikzpicture}[scale=0.4]
\tikzstyle{vertex}=[circle,fill=black,minimum size=8pt,inner sep=0pt]
\node[vertex] (A) at (0,0){} ;
\node[vertex] (C) at (4,0){};
\node[vertex] (D) at (2,0){};
\node[vertex] (A1) at (-2,0){};
\node[vertex] (D1) at (2,-2){};
\draw[double distance=2 pt] (A)--(D);
\draw[double distance=2 pt] (C)--(D);
\draw (A)--(A1);
\draw (D)--(D1);
\draw[dashed] (D) to[out=45, in=0] (2,1.3) to[out=180,in=135] (D);
\draw (2.2,0) node[below right]{$2$};
\draw (-0.3,0) node[below]{$2$};
\draw (4.3, 0) node[right]{$1$};
\draw (A1) node[left]{$1$};
\draw (D1) node[right]{$1$};
\end{tikzpicture} \\
  $H_{\text{Gar}}^{3/2+1+1+1}$  & $\cir{\alpha z^{1/2}}_\infty + \cir{0}_a +\cir{0}_b +\cir{0}_c$ &
\begin{tikzpicture}[scale=0.4]
\tikzstyle{vertex}=[circle,fill=black,minimum size=8pt,inner sep=0pt]
\tikzstyle{vertex_blue}=[circle,fill=blue,minimum size=8pt,inner sep=0pt]
\node[vertex] (A) at (0,-2){} ;
\node[vertex] (B) at (2,-2){} ; 
\node[vertex] (C) at (4,-2){};
\node[vertex] (D) at (2,0){};
\draw[double distance=2 pt] (A)--(D);
\draw[double distance=2 pt] (B)--(D);
\draw[double distance=2 pt] (C)--(D);
\draw[dashed] (D) to[out=45, in=0] (2,1.3) to[out=180,in=135] (D);
\draw (2.2,0) node[right]{$1$};
\draw (-0.3,-2) node[left]{$1$};
\draw (2,-2.2) node[below]{$1$};
\draw (4.3,-2) node[right]{$1$};
\end{tikzpicture}\\
  $H_{\text{Ss}}^{D_4}$ & $ 3\cir{\alpha z}_\infty+\cir{\beta z}_\infty + 2 \cir{\gamma z^{-1}}_0$ &
\begin{tikzpicture}[scale=0.4]
\tikzstyle{vertex}=[circle,fill=black,minimum size=8pt,inner sep=0pt]
\tikzstyle{vertex_blue}=[circle,fill=blue,minimum size=8pt,inner sep=0pt]
\node[vertex] (A) at (0,0){} ;
\node[vertex] (B) at (2,0){} ; 
\node[vertex] (C) at (4,0){};
\node[vertex] (D) at (6,0){};
\node[vertex] (E) at (8,0){};
\draw[double distance=2 pt] (A)--(B)--(C);
\draw (C)--(D)--(E);
\draw[dashed] (B) to[out=45, in=0] (2,1.3) to[out=180,in=135] (B);
\draw (0,-0.2) node[below]{$1$};
\draw (2,-0.2) node[below]{$2$};
\draw (4,-0.2) node[below]{$3$};
\draw (6,-0.2) node[below]{$2$};
\draw (8,-0.2) node[below]{$1$};
\end{tikzpicture}\\
$H_{III}^{\text{Mat}}(D_6)$ & $ 2\cir{\alpha z}_\infty+ 2\cir{\beta z}_\infty + 2\cir{\gamma z^{-1}}_0$ &
\begin{tikzpicture}[scale=0.4]
\tikzstyle{vertex}=[circle,fill=black,minimum size=8pt,inner sep=0pt]
\tikzstyle{vertex_blue}=[circle,fill=blue,minimum size=8pt,inner sep=0pt]
\node[vertex] (A) at (0,0){} ;
\node[vertex] (B) at (2,0){} ; 
\node[vertex] (C) at (4,0){};
\node[vertex] (D) at (6,0){};
\draw[double distance=2 pt] (A)--(B)--(C);
\draw (C)--(D);
\draw[dashed] (B) to[out=45, in=0] (2,1.3) to[out=180,in=135] (B);
\draw (0,-0.2) node[below]{$2$};
\draw (2,-0.2) node[below]{$2$};
\draw (4,-0.2) node[below]{$2$};
\draw (6,-0.2) node[below]{$1$};
\end{tikzpicture}\\
 \hline
\end{tabular}
\end{center}
\caption{Diagrams for the Lax representations of 4-dimensional Painlevé-type equations of \cite{kawakami2018degeneration} featuring several irregular singularities. Notice that the last line is none other than a higher Painlevé III system.}
\label{our_diagrams_4d_painleve}
\end{figure}

\begin{remark}
Notice that the diagram associated to the Painlevé-type equation denoted $H_{III}^{\text{Mat}}(D_6)$ in \cite{kawakami2018degeneration} fits in the higher Painlevé picture in the sense of Boalch: it is $hP_3^{(2)}$, obtained from the Painlevé III diagram by taking all multiplicities equal to 2 and adding a vertex with multiplicity 1. Beware that the usual terminology  ``matrix Painlevé equations'' \cite{balandin1998painleve} differs from that of \cite{kawakami2018degeneration}.   
\end{remark}

\begin{remark}
\label{rem:hiroe_diagrams}
It is interesting to compare our diagrams with those of Hiroe \cite{hiroe2014moduli} and Hiroe and Oshima \cite{hiroe2013classification}. In \cite{hiroe2017linear}, Hiroe defines quivers associated to meromorphic connections with several unramified irregular singularities. When there is only one irregular singularity, these quivers coincide with the ones of \cite{boalch2012simply} and in turn with our diagrams, however when they are several irregular singularities they differ from our diagrams. Hiroe also introduces the notion of \textit{shape}, and in their approach the moduli spaces are classified by these shapes. When there are several irregular singularities, the shape differs from the diagram. It is unclear whether there are instances where different Lax pairs for a Painleve-type equation leads to the same shape in the sense. Futhermore, this approach does not allow to define diagrams for the degenerate and doubly degenerate Painlevé equations.
\end{remark}

\newpage

\appendix

\section{Formula for the number of edges}

We now prove the explicit formula for the number of edges (lemma \ref{nb_aretes_cas_general}). Lets us recall the formula:

\begin{lemma} 
\begin{itemize}
\item Assume that $\alpha_{r}/\beta\geq \alpha'_{r}/\beta'$. Then the number of edges between $\cir{q}$ and $\cir{q'}$ is
\begin{align*}
B_{\cir{q},\cir{q'}}= &(\beta'-(\alpha'_,\beta'))\alpha_0 +((\alpha'_0,\beta') - (\alpha'_0,\alpha'_1,\beta'))\alpha_1  + \dots +((\alpha'_0,\dots, \alpha'_{r-2},\beta') - (\alpha'_0,\dots,\alpha'_{r-1},\beta'))\alpha_{r-1} \\&+(\alpha'_0,\dots ,\alpha'_{r-1},\beta') \alpha_{r} - \beta\beta'
\end{align*}
\item In particular, if $q$ and $q'$ have no common parts and $\alpha/\beta\geq \alpha'/\beta'$, then 
\begin{align*}
B_{\cir{q},\cir{q'}}= \beta'(\alpha - \beta).
\end{align*}
\end{itemize}

\end{lemma}

\begin{proof}
To compute the number of edges between $\cir{q}$ and $\cir{q'}$ we have to determine the local system $\Hom(\cir{q},\cir{q'})$ as well as its irregularity. For this we have to look at the differences $q_i-q'_j$, with $i=0,\dots \beta-1, j=0,\dots, \beta'-1$ between all possible leaves $\cir{q}$ of $\cir{q'}$, find which connected components of $\Ical$ they fall into and what the irregularity of those circles is. The subtlety is that the degree of $q_i-q_j$ depends in general of $i$ and $j$, so that the circles will not have the same irregularity.

Let $\mu$ be the smallest common multiple of $\beta$ and $\widetilde{\beta}$. We set
\[ \mu=k \beta, \qquad \mu=k'\beta'.
\]
Let us also denote by $\delta$ the greatest common divisor of $\beta$ and $\beta'$. We have 
\[ \beta=k' \delta, \qquad \beta'=k\delta.
\]
For $0\leq i \leq r$, let $\gamma_i$ be the integer such that
\[\frac{\alpha_i}{\beta}=\frac{\alpha'_i}{\beta'}=\frac{\gamma_i}{\delta}.\] 
Any difference $q_i-q'_j$ belongs to a circle $I=\cir{q_i-q'_j}$ which is a connected component of $\Hom(\cir{q},\cir{q'})$. However, several differences $q_i-q'_j$ belong to the same connected component: each difference $q_i-q_j$ corresponds to exactly one leaf of such a circle. This implies the following formula:
\begin{equation}
\Irr(\Hom(\cir{q},\cir{q'}))=\sum_{i=0}^{\beta-1}\sum_{j=0}^{\beta'-1}\slope(q_i-q'_j).
\label{Irr_Hom}
\end{equation}
Indeed, regrouping the terms of the sum according to the connected components, to each connected component $I$ of ramification order $r$ and irregularity $s$ correspond $r$ differences $q_i-q'_j$ in the sum, with slope equal to $\slope(q_i-q'_j)=s/r$. The total contribution of those terms is thus equal to $s=\Irr(I)$.

The computation can be simplified by noticing the following fact: for any $k\in \Z$, 
\[
\slope(q_i-q_j)=\slope(q_{i+k}-q'_{j+k}),
\]
where the index $i+k$ is seen as an element of $\Z/\beta \Z$, and $j+k$ is seen as an element of $\Z/\beta' \Z$. Under this shifting action of $\Z$, the set $\Z/\beta \Z\times \Z/\beta' \Z$ is partitioned into $\delta$ orbits, each having cardinal $\mu$. Furthermore, the differences
\[ q_0-q'_j, \; j=0,\dots, \delta-1,
\]
belong to distinct orbits and thus yield one representative of each orbit.
The formula \eqref{Irr_Hom} therefore becomes
\begin{align*}
\Irr(\Hom(\cir{q},\cir{q'}))&=\mu \sum_{j=0}^{\delta-1} \slope(q_0-q_j)\\
&=\sum_{j=0}^{\delta-1} \deg_{z^{-1/\mu}}(q_0-q_j).
\end{align*}

The task is thus reduced to computing the degree of the differences $ q_0-q'_j$, $j=0,\dots,\delta-1$ as polynomials in $z^{-1/\mu}$. It is the exponent of the largest monomial having different coefficients in  $q_0$ and $q'_j$.

The common part $q_c=q'_c$ is given by 
\[ q_c=q'_c=\sum_{i=0}^r b_i z_\infty^{-\alpha_i/\beta}=\sum_{i=0}^r b_i z_\infty^{-\alpha'_i/\beta'}=\sum_{i=0}^r b_i z_\infty^{-\gamma_i/\delta}=\sum_{i=0}^r b_i z_\infty^{-kk'\gamma_i/\mu}.
\]
We will start by determining the number of indices $j=0\in \{0,\dots,\delta-1\}$ such that this degree is the maximal possible degree $kk'\gamma_0$, then the number of indices for which it is $kk'\gamma_1$, etc.

Let us thus begin by computing the number of differences with degree $kk'\gamma_0$. We consider the coefficient of $z_\infty^{kk'\gamma_0/\mu}$ in the difference $q_0-q'_j$, $j=0,\dots,\delta-1$: the corresponding term is
\[ a_0(z_\infty^{kk'\gamma_0/\mu}-e^{2i\pi j\gamma_0/\delta}z_\infty^{kk'\gamma_0/\mu})= a_0 (1-e^{2i\pi j\gamma_0/\delta})z_\infty^{kk'\gamma_0/\mu}.
\] 
The factor $1-e^{2i\pi j\gamma_0/\delta}$ is zero if and only if $j$ is an integer multiple of $\delta/(\gamma_0, \delta)$. There are thus $(\gamma_0,\delta)$ differences $q_0-q'_j$ having a degree strictly less than $kk'\gamma_0$. The $\delta-(\gamma_0,\delta)$ other differences $q_0-q'_j$ have degree $kk'\gamma_0$. Each one of these contributes to $kk'\gamma_0$ Stokes arrows from $\cir{q}$ to $\cir{q'}$. 

Next, we compute the number of differences $q_0-q'_j, j=0,\dots \delta-1,$ whose degree is $kk'\gamma_1$. In $q_0-q_j$ the monomial of degree $kk'\gamma_1$ is 
\[ a_1(1-e^{2i\pi j\gamma_1/\delta})z_\infty^{kk'\gamma_1/\mu}.
\]
As previously, this term is non-zero when $j$ is an integer multiple of $\delta/(\gamma_1,\delta)$. It follows that $q_0-q'_j$ has degree strictly less than $kk'\gamma_1$ when $j$ is both a multiple of $\delta/(\gamma_1 , \delta)$ and of $\delta/(\gamma_0,\delta)$, i.e. is a multiple of their lowest common multiple $\lcm\left(\frac{\delta}{(\gamma_0,\delta)}, \frac{\delta}{(\gamma_1,\delta)}\right)$. Since we have the equality
\[ \lcm\left(\frac{\delta}{(\gamma_0,\delta)}, \frac{\delta}{(\gamma_1,\delta)}\right)=\frac{\delta}{(\gamma_0,\gamma_1,\delta)},
\]
we conclude that the number of differences $q_0-q'_j$, $j=0,\dots,\delta-1$ having degree strictly less than $kk'\gamma_1$ is equal to $(\gamma_0, \gamma_1,\delta)$. Therefore, the number of differences having degree equal to  $kk'\gamma_1$ is $(\gamma_0,\delta)-(\gamma_0,\gamma_1,\delta)$. Each if these differences gives rise to $kk'\gamma_1$ Stokes arrows from $\cir{q}$ to $\cir{q'}$.

The same reasoning can be carried out for the next terms in $q_c=q'_c$. By induction, we find that for $1\leq i\leq r$, the number of differences $q_0-q'_j$, $j=0,\dots,\delta-1$ with degree $kk'\gamma_i$ is equal to the difference $(\gamma_0,\dots ,\gamma_{i-1}, \delta)-(\gamma_0,\dots ,\gamma_{i}, \delta)$, and each of those circles gives rise to $kk'\gamma_i$ Stokes arrows between $\cir{q}$ and $\cir{q'}$.

There only remain $(\gamma_0, \dots ,\gamma_{r-1}, \delta)$ differences having degree strictly less than $kk'\gamma_{r-1}$. For those differences, we have $q_{c,0}-q'_{c,j}=0$, so only the different parts matter: $q_0-q'_j=q_{d,0}-q'_{d,j}$. Now, the leading term of $q_d$ is $b_{r}z_\infty^{-\alpha_{r}/\beta}=b_{r}z_\infty^{-k\alpha_{r}/\mu}
$, so $q_d$ has degree $k\alpha_{r}$ as a polynomial in $z^{-1/\mu}$, and $q'_d$ has degree $k'\alpha'_{r}$ in $z^{-1/\mu}$. It follows that the degree of $q_0-q'_j$ is $\max(k\alpha_{r}, k'\alpha'_{r})$, since  we have assumed that $\alpha_{r}/\beta\geq \alpha'_{r}/\beta'$, the maximum is $k\alpha_{r}$. Each of these thus accounts for $k\alpha_{r}$ Stokes arrows from $\cir{q}$ to $\cir{q'}$.

Adding all contributions, and subtracting the $\beta\beta'$ arrows appearing in the definition of $B_{\cir{q},\cir{q'}}$, we get the following expression for the number of edges between $\cir{q}$ and $\cir{q'}$: 
\begin{align*}
B_{\cir{q},\cir{q'}}= &(\delta-(\gamma_0, \delta))kk'\gamma_0 +((\gamma_0 ,\delta) - (\gamma_0, \gamma_1 ,\delta))kk'\gamma_1 + \dots +((\gamma_0 ,\dots , \gamma_{r-2},\delta) - (\gamma , \dots, \gamma_{r-1} ,\delta))kk'\gamma_{r-1} \\&+(\gamma_0, \dots ,\gamma_{r-1},\delta) k\alpha_{r} - \beta\beta'.
\end{align*}
Since $k\delta=\beta'$, $k\gamma_i=\alpha'_i$, and $k'\gamma_i=\alpha_i$, this yields the desired formula.

The case where $q$ and $q'$ have no common part corresponds to having $r=0$. In the formula, all terms corresponding to the common part disappear, there only remains $B_{\cir{q},\cir{q'}}=\beta'(\alpha-\beta)$.
\end{proof}

The similar formula for the number of loops is 

\begin{lemma}
Let $q=\sum_{j=0}^{p} b_j z_\infty^{-\alpha_j/\beta}$ be an exponential factor of slope $\alpha_0/\beta>1$ as before.

\begin{itemize}
\item One has

\begin{equation}
B_{q,q}=(\beta-(\alpha_0,\beta))\alpha_0 +((\alpha_0, \beta)-(\alpha_0,\alpha_1,\beta))\alpha_1 +\dots +((\alpha_0,\dots, \alpha_{p-1}, \beta)-(\alpha_0,\dots,\alpha_p,\beta))\alpha_p-\beta^2+1.
\end{equation}
\item Otherwise, if $(\alpha,\beta)=1$, then we have 
\begin{equation}
B_{q,q}=(\beta-1)(\alpha-\beta-1).
\end{equation}
\end{itemize}
\end{lemma}

\begin{proof}
The proof is similar to the case of two different circles. As previously we have
\[
\Irr(\End(\cir{q}))=\sum_{i=0}^{\beta-1}\sum_{j=0}^{\beta-1}\slope(q_i-q_j)=\sum_{j=0}^{\beta-1}\slope(q_0-q_j).
\label{Irr_End}
\]
Among those $\beta$ differences, we determine how many have degree $\alpha_0, \dots \alpha_p$, as polynomials in $z^{-1/\beta}$, and the remaining differences will have degree $0$. We find that the number of differences with degree $\alpha_0$ is $\beta-(\alpha_0, \beta)$, the number of differences with degree $\alpha_1$ is $(\alpha_0, \beta)-(\alpha_0,\alpha_1,\beta)$, etc. One has $(\alpha_0,\dots, \alpha_{p-1}, \beta)-(\alpha_0,\dots,\alpha_p,\beta)$ differences with degree $\alpha_p$, and the $(\alpha_0,\dots,\alpha_p,\beta)$ remaining differences belong to connected components that are copies of $\cir{0}$. Each difference of degree $\alpha_i$ accounts for $\alpha_i$ Stokes arrows. The total number of (positive) Stokes arrows is thus
\[ (\beta-(\alpha_0, \beta))\alpha_0 +((\alpha_0,\beta)-(\alpha_0,\alpha_1, \beta))\alpha_1 +\dots +((\alpha_0,\dots,\alpha_{p-1},\beta)-(\alpha_0,\dots\alpha_p, \beta))\alpha_p.
\]
To this we must subtract a number of arrows equal to $\beta(\beta-1)+(\beta-1)=\beta^2-1$, which gives the desired formula. 

When $\alpha$ and $\beta$ are relatively prime, only the first term remains: the number of positive Stokes arrows is $(\beta-1)\alpha$, and the conclusion follows.
\end{proof}

\section{Form of the Legendre transform}

We now prove the formula for the levels of the Legendre transform (lemma \ref{forme_de_tilde(q)}).

\begin{lemma}
Let $q=\sum_{j=0}^{p} b_j z^{\alpha_j/\beta}$, with $b_i\neq 0$,be an exponential factor at infinity. We set $\alpha:=\alpha_0$, so the slope of $q$ is $\alpha/\beta>1$. Then the exponents of $\xi$ possibly appearing with non-zero coefficients in its Legendre transform $\widetilde{q}$ are of the form $\frac{\alpha-k_1(\alpha-\alpha_1)-\dots-k_p(\alpha-\alpha_p)}{\alpha-\beta}$, with $k_1,\dots,k_p\geq 0$. More precisely, if we set $E:=\{ \gamma\in \N \;|\;\exists k_1,\dots, k_p \in \N, \gamma=k_1(\alpha-\alpha_1)+\dots+k_p(\alpha-\alpha_p) \}$
the Legendre transform has the form
\begin{equation}
\widetilde{q}(\xi)=\sum_{\substack{\gamma \in E \\ \alpha-\gamma >0}}\widetilde{b}_{\gamma}\xi^{\frac{\alpha-\gamma}{\alpha-\beta}},
\end{equation} 
where the sum is restricted to the terms such that the exponent $\frac{\alpha-\gamma}{\alpha-\beta}$ is positive. Furthermore, the coefficients $\widetilde{b}_{(\alpha-\alpha_i)}$ are non-zero for $i\geq 1$. 
\end{lemma}

\begin{proof}
The proof consists in computing the Legendre transform directly from the system of equations by which it is defined, in a way explicit enough to find the order of the terms which appear. The first equation of the system is
\begin{equation}
\frac{dq}{dz}=\xi,
\label{x_fct_de_xi_2}
\end{equation}
it is interpreted as determining $\xi$ as a function of $z$. We will show that this implies $\xi$ is of the form
\begin{equation}
z=\sum_{\gamma \in E}c_{\gamma}\xi^{\frac{\beta-\gamma}{\alpha-\beta}}.
\label{forme_de_x(xi)}
\end{equation}
The second equation of the system then yields $\frac{d\widetilde{q}}{d \xi}=-z$, and the lemma follows by integrating.

Equation \eqref{x_fct_de_xi_2} has a solution $z(\xi)$ in the field of Puiseux series in the variable $\xi$ which is unique once we fix a choice of $\alpha-\beta$-th root. To show  \eqref{forme_de_x(xi)}, we thus can take this form of solution as an ansatz, and check that it gives  a unique solution for the coefficients of this expression: this will automatically be the solution $z(\xi)$. So let us assume \eqref{forme_de_x(xi)}. We set
\[
\frac{dq}{dz} = \sum_{j=0}^{p} b'_j z^\frac{\alpha_j-\beta}{\beta}.
\]
This implies
\[
\xi=\frac{d q}{d z} = \sum_{j=0}^{p} b'_j \left(\sum_{\gamma \in E}c_{\gamma}\xi^{\frac{\beta-\gamma}{\alpha-\beta}}\right)^\frac{\alpha_j-\beta}{\beta}.
\]
We develop this expression to identify the coefficients. We have
\begin{align*}
\xi&=\sum_{j=0}^p b'_j \left(\sum_{\gamma}c_{\gamma}\xi^{\frac{\beta-\gamma}{\alpha-\beta}}\right)^\frac{\alpha_j-\beta}{\beta}\\
&=\sum_{j=0}^p b'_j \xi^{-\frac{\alpha_j-\beta}{\alpha-\beta}}\left(\sum_{\gamma}c_{\gamma}\xi^{\frac{-\gamma}{\alpha-\beta}}\right)^\frac{\alpha_j-\beta}{\beta}\\
&= \sum_{j=0}^p b''_j \xi^{1-\frac{\alpha-\alpha_j}{\alpha-\beta}}\left(1+\sum_{\gamma\neq 0}c'_{\gamma}\xi^{\frac{-\gamma}{\alpha-\beta}}\right)^\frac{\alpha_j-\beta}{\beta}\\
&= \sum_{j=0}^p b''_j \xi^{1-\frac{\alpha-\alpha_j}{\alpha-\beta}}\left(\sum_{{(l_{\gamma})}_{\gamma\neq 0}}\prod_{\{\gamma,\; l_{\gamma}\neq 0\}}A^{(j)}_{l_{\gamma}} {c'_{\gamma}}^{l_{\gamma}}\xi^{\frac{-l_\gamma\gamma}{\alpha-\beta}}\right)\\
&=\sum_{j,(l_{\gamma})}b''_j \left(\prod_{\{\gamma,\; l_{\gamma}\neq 0\}}A^{(j)}_{l_{\gamma}} {c'_{\gamma}}^{l_{\gamma}}\right)\xi^{\left(1-\frac{(\alpha-\alpha_j)+\sum_{\gamma} l_{\gamma} \gamma}{\alpha-\beta}\right)}\\
&=\sum_{\delta \geq 0} d_\delta \xi^{1-\frac{\delta}{\alpha-\beta}},
\end{align*}
with
\begin{equation}
d_\delta=\sum_{\substack{j,(l_{\gamma})\\ \sum_{\gamma} l_{\gamma} \gamma+(\alpha-\alpha_j)=\delta}} b''_j \left(\prod_{\{\gamma,\; l_{\gamma}\neq 0\}}A^{(j)}_{l_{\gamma}} {c'_{\gamma}}^{l_{\gamma}}\right).
\label{identification_legendre}
\end{equation}

In the calculation, we have set $b''_j=b'_j c_0^{(\alpha_j-\beta)/\beta}$, $c'_\gamma=c_\gamma/c_0$, and the $A^{(j)}_{l_\gamma}$ are the combinatorial coefficients appearing in the series expansion of the term $(1+\dots)^{(\alpha_j-\beta)/\beta}$. 

The crucial point in this computation is that the integers $\delta=\sum_{\gamma} l_{\gamma} \gamma+(\alpha-\alpha_j)$ still belong to the set of exponents $E$. This is the reason why our ansatz is correct. This enables us to find the coefficients $c_\gamma$ by induction: 
\begin{itemize}
\item For $\delta=0$, we have $d_0=b''_0$. By comparing the terms in $\xi$, we find $1=d_0=b''_0$, then $c_0$. This is where the choice of $\alpha-\beta$-th root takes place. 
\item Let $\delta\in E$. Assume that we know all $c_\gamma$ for $\gamma\in E$ such that $\gamma\leq \delta$. Let $\delta'$ be the smallest element of $E$ strictly greater than $\delta$. Equation \eqref{x_fct_de_xi_2} gives $d_{\delta'}=0$. But $\delta'$ is a sum of terms featuring the $c'_\gamma$ for $\gamma\leq \delta'$, and of a single term featuring $d_{\delta}'$, equal to 
\[ b''_0 A^{(0)}_1 c'_{\delta_1}= \frac{\alpha_0-\beta}{\beta}c'_{\delta'}
\]
since we have  $A^{(0)}_1=\frac{\alpha_0-\beta}{\beta}\neq 0$. All the other terms being known by induction hypothesis, this determines $c'_{\delta'}$ and hence $c_{\delta'}$, in a unique way.  
\end{itemize}
It remains to see that the first subleading term associated to each $\alpha_i$, that is $\tilde{b}_{(\alpha-\alpha_i)}$, is non-zero. For this we look at the coefficient $\delta_{(\alpha-\alpha_j)}$. Equation \eqref{forme_de_x(xi)} implies that $\delta_{(\alpha-\alpha_j)}=0$. On the other hand $\delta_{(\alpha-\alpha_j)}$ is given by \eqref{identification_legendre} with $\delta=\alpha-\alpha_j$. Let us write down this equation more explicitly. There are only two decompositions $\alpha-\alpha_j$ of the form
\[ \alpha-\alpha_j=\sum_{\gamma\in E} l_\gamma \gamma + (\alpha-\alpha_k),
\]
with $l_\gamma\geq 0, k\in \{0,\dots,p\}$:
\[
\alpha-\alpha_j= 0 + (\alpha-\alpha_j),\text{ and } \alpha-\alpha_j= 1\times (\alpha-\alpha_j) + \underbrace{(\alpha-\alpha_0)}_{=0}.
\]
Equation \eqref{identification_legendre} thus yields
\[
0=d_{\alpha-\alpha_j}=b''_j\times 1 + b''_0 A^{(j)}_{\alpha-\alpha_j}c'_{\alpha-\alpha_j}.
\]
Since $b''_0,b''_j$ and $A^{(j)}_{\alpha-\alpha_j}$ are non-zero, this implies that $c'_{\alpha-\alpha_j}\neq 0$, and the conclusion follows.

\end{proof}

\section{Computation of the dimension}

We now turn to the proof of the formula for the dimension of the wild character variety (theorem \ref{theorem_dimension}):

\begin{theorem}
The dimension of the wild character variety $\mathcal{M}_B(\bm\Theta,\bm{\Ccal})$, when it is non-empty, is given by
\begin{equation}
\dim \mathcal{M}_B(\bm\Theta,\bm{\Ccal})=2-(\mathbf{d},\mathbf{d}),
\end{equation}
where $(\cdot,\cdot)$ is the bilinear form associated with the diagram $\Gamma(\bm\Theta,\bm{\Ccal})$, and $\mathbf{d}$ is the dimension vector of $\Gamma(\bm\Theta,\bm{\Ccal})$, for any choice of marking $\bm\xi$. 
\end{theorem}

To prove this, we are going to use the twisted quasi-Hamiltonian description of the wild character variety given in \cite{boalch2015twisted}, to compute its  dimension. Keeping previous notations, the wild character variety $\mathcal{M}_B(V)$ is obtained as a symplectic reduction of the twisted quasi-Hamiltonian space 
\[  \Hom_{\mathbb{S}}(V)\simeq \Acal(V^0_{a_1}) \circledast \dots \circledast \Acal(V^0_{a_m})\sslash G. 
\]
Here, each piece $\Acal(V^0_{a_k})$ is a twisted quasi-Hamiltonian $H(\partial_{a_k})\times G$-space: 
\[ \Acal(V^0_{a_k})=H(\partial_{a_k})\times G \times \prod_{d\in \A_{a_k}}\Sto_d,
\]
where $\A_{a_k}$ is the set of singular directions of the connection at $a_k$, and $\Sto_d$ is the Stokes group associated to the singular direction $d\in \A_{a_k}$. Hence,
$\Hom_{\mathbb{S}}(E,\nabla)$ is a twisted quasi-Hamiltonian $\mathbf{H}\times G$-space, where 
$\mathbf{H}=H_1\times \dots\times H_m$, with $H_i=\GrAut(V^0_{a_i})$.

The formal monodromies $\rho_k\in H(\partial_{a_k})$ determine twisted conjugacy classes $\mathcal{C}(\partial_{a_k})\subset H(\partial_{a_k})$ which do not depend on the choice of direction $d\in \partial_{a_k}$ and the choice of isomorphism $V^0_{d}\simeq \C^{\sum_j n^{(k)}_i \beta^{(k)}_i}$; where the $n_i^{(k)}$ denote the multiplicities of the active circles at $a_i$, and $\beta_i^{(k)}$ their ramifications (the tame circle is included). The wild character variety is the twisted quasi-Hamiltonian reduction of $\Hom_{\mathbb{S}}(E,\nabla)$ at the twisted conjugacy class $\bm{\mathcal{C}}:=\prod_{k=1}^m \mathcal{C}(\partial_{a_k})$, i.e. 
\begin{equation}
\mathcal{M}_B(V)=\Hom_{\mathbb{S}}(E,\nabla)\sslash_{\bm{\mathcal{C}}}  \mathbf{H}.
\end{equation}

This description enables us to compute the dimension of $\mathcal{M}_B(E,\nabla)$. One has 
\begin{equation}
\dim \Acal(V^0_{a_k})=\dim G + \dim H(\partial_{a_k}) + \dim \prod_{d\in \A_{a_k}}\Sto_d.
\label{dimension_A}
\end{equation}
The dimension of $G$ is $\dim G=n^2$. One has $\dim H(\partial_{a_k})=\sum_i \beta^{(k)}_i {n^{(k)}_i}^2$. The dimension of the product of the Stokes groups can be expressed as a function of the number of Stokes arrows between the active circles (see \cite{boalch2014geometry}):

\begin{equation}
\dim \prod_{d\in \A_{a_k}}\Sto_d =\sum_{1\leq i,j \leq r_k} n^{(k)}_i n^{(k)}_j B^+_{i,j},
\end{equation}
where $B^+_{i,j}$ denotes the number of (positive) Stokes arrows between $\cir{q^{(k)}_i}$ and $\cir{q^{(k)}_j}$ in the Stokes diagram corresponding to the formal local system $V^0_{a_i}\to \partial_{a_i}$. 

The dimension of $\Hom_{\mathbb{S}}(V)$ is 
\[ \dim \Hom_{\mathbb{S}}(V)=\sum_{i=1}^m \dim \Acal(V^0_{a_k}) -2\dim G + 2\dim Z(G).,
\]
and the dimension of the wild character variety is then, taking into account the symplectic reduction at the conjugacy classes $\mathcal{C}(\partial_{a_k})$ is given by
\begin{equation}
\dim \mathcal{M}_B(V)= \sum_{k=1}^m \left(\dim \Acal(V^0_{a_k}) + \dim \mathcal{C}(\partial_{a_k}) -2\dim H(\partial_{a_k})\right) -2\dim G + 2\dim Z(G).
\label{dimension_M_B}
\end{equation}

We now are in position to compute the quantities $D:=2-(\mathbf{d},\mathbf{d})$ and 
$D':=\dim \mathcal{M}_B(V)$. Let us introduce the following notations:
\begin{itemize}
\item Let $a_1,\dots, a_m$ be the singularities at finite distance. 
\item For $k=1,\dots,m$, let $\cir{q^{(k)}_1}_{a_k},\dots \cir{q^{(k)}_{r_k}}_{a_k}\in \pi_0(\Ical_{a_k})$ denote the irregular active circles at $a_k$, $n^{(k)}_i\in \N$ their respective multiplicities, and $\beta^{(k)}_j$ their respective ramification orders, and $\alpha^{(k)}_j$ their irregularities, in an order such that the slopes satisfy
$$
\frac{\alpha^{(k)}_1}{\beta^{(k)}_1} \geq \dots \geq \frac{\alpha^{(k)}_{r_k}}{\beta^{(k)}_{r_k}}.
$$
\item For $k=1,\dots,m$, let $m_k$ be the multiplicity of the tame circle $\cir{0}_{a_k}$ in the modified formal local system at $a_k$.
\item At infinity, let $\cir{q^{(\infty)}_1}_\infty,\dots \cir{q^{(\infty)}_{r_\infty}}_\infty\in \pi_0(\Ical_{a_k})$ denote the active circles (the tame circle is included), $\beta^{(\infty)}_j$ their respective ramification orders, and $\alpha^{(\infty)}_j$ their irregularities, in an order such that the slopes satisfy
$$
\frac{\alpha^{(\infty)}_1}{\beta^{(\infty)}_1} \geq \dots \geq \frac{\alpha^{(\infty)}_{r_\infty}}{\beta^{(\infty)}_{r_\infty}}.
$$
\end{itemize}

Here we will assume that all Stokes arrows come from the leading terms in the exponential factors. In the language of the previous section, this means that the different exponential factors have no common part. This entails no loss of generality, indeed in the computation of the number of Stokes arrows we have seen that the terms involving subleading terms do not change under Fourier transform, so that those term will give the same contribution to $2-(\mathbf{d},\mathbf{d})$ and 
$\dim \mathcal{M}_B(\bm\Theta,\bm{\Ccal})$. 

Let us list the number of  loops and edges between the different types of circles. From the formulas for the number of loops and edges, and the formulas the the transformation of slopes under Fourier transform, we get
\begin{itemize}
\item Loops at circles at infinity: $B_{\cir{q^{(\infty)}_i},\cir{q^{(\infty)}_i}}=(\beta^{(\infty)}_i-1)(\alpha^{(\infty)}_i-\beta^{(\infty)}_i-1)$.
\item Edges between different circles at infinity: $B_{\cir{q^{(\infty)}_i},\cir{q^{(\infty)}_j}}=\beta^{(\infty)}_j(\alpha^{(\infty)}_i-\beta^{(\infty)}_i)$ if $i<j$. 
\item Loops at irregular circles at finite distance: $B_{\cir{q^{(k)}_i},\cir{q^{(k)}_i}}=(-\beta^{(k)}_i-1)(\alpha^{(k)}_i+\beta^{(k)}_i-1)$.
\item Edges between different irregular circles at a same pole at finite distance: $B_{\cir{q^{(k)}_i},\cir{q^{(k)}_j}}=-\beta^{(k)}_i(\alpha^{(k)}_j+\beta^{(k)}_j)$ if $i<j$. 
\item Edges between the tame circle $\cir{0}_{a_k}$ and an irregular circle at $a_k$: $B_{\cir{0}_{a_k},\cir{q^{(k)}_i}}=-\beta^{(k)}_i$.
\item Edges between circles at two different poles at finite distance: they are none, i.e. $B_{\cir{q^{(k)}_i},\cir{q^{(l)}_j}}=0$
\item Edges between an irregular circle at finite distance and a circle at infinity: $B_{\cir{q^{(k)}_i},\cir{q^{(\infty)}_j}}=\beta^{(\infty)}_j(\alpha^{(k)}_i+\beta^{(k)}_i)$.
\item Edges between the tame circle $\cir{0}_{a_k}$ and a circle at infinity: $B_{\cir{0}_{a_k},\cir{q^{\infty}_i}}=\beta^{(\infty)}_j$. 
\end{itemize}

For each circle $I$, let $D_{\Lbb_{I}}$ be the contribution to $2-(\mathbf{d},\mathbf{d})$ of the leg $\Lbb_{I}$ associated to $I$. Putting everything together, this gives us the following lengthy expression for $D$: 
\begin{equation}
\begin{split}
2-(\mathbf{d},\mathbf{d})=& 2+\sum_{i^{(\infty)}}\left( -2+(\beta^{(\infty)}_i-1)(\alpha^{(\infty)}_i-\beta^{(\infty)}_i-1)\right){n^{(\infty)}_i}^2 \\
 & +\sum_{i^{(\infty)}<j^{(\infty)}} 2\beta^{(\infty)}_j(\alpha^{(\infty)}_i-\beta^{(\infty)}_i)n^{(\infty)}_i n^{(\infty)}_j
 + \sum_k \sum_{i^{(k)}}\left( -2+(-\beta^{(k)}_i-1)(\alpha^{(k)}_i+\beta^{(k)}_i-1)\right){n^{(k)}_i}^2\\
 &+\sum_k\sum_{i^{(k)}<j^{(k)}} -2\beta^{(k)}_i(\alpha^{(k)}_j+\beta^{(k)}_j)n^{(k)}_i n^{(k)}_j+ \sum_k \sum_{i^{(k)}}-2\beta^{(k)}_i m_k n^{(k)}_i
 \\
& + \sum_k\sum_{i^{(k)}<j^{(\infty)}} 2\beta^{(\infty)}_j(\alpha^{(k)}_i+\beta^{(k)}_i)n^{(k)}_i n^{(\infty)}_j
 + \sum_k \sum_{j^{(\infty)}}2\beta^{(\infty)}_j m_k n^{(\infty)}_j -2\sum_k m_k^2 \\&+\sum_{i^{(\infty)}} D(\Lbb_{\cir{q_{i^{(\infty)}}}}) + \sum_k \sum_{i^{(k)}} D(\Lbb_{\cir{q^{(k)}_i}})+ \sum_k D(\Lbb_{\cir{0}_{a_k}}).
\end{split}
\end{equation}

To compute the dimension of the wild character variety, we have to find the number of (positive) Stokes arrows at each singularity. We have 
\begin{itemize}
\item Number of Stokes arrows between two irregular circles $\cir{q^{(k)}_i}$, $\cir{q^{(k)}_i}$ with $i<j$ at $a_k$: $\beta^{(k)}_j \alpha^{(k)}_i$
\item Number of Stokes loops at the irregular circle $\cir{q^{(k)}_i}$: $\alpha^{(k)}_i(\beta^{(k)}_i-1)$.
\item Number of Stokes arrows between the irregular circle $\cir{q^{(k)}_i}$ and the tame circle $\cir{0}_{a_k}$ at $a_k$: $ \alpha^{(k)}_i$.
\end{itemize}
We also have the similar numbers of arrows at infinity.

Let $n_k$ be the multiplicity of the tame circle $\cir{0}_{a_k}$ in the non-modifed local system $V^0_{a_k}$, i.e. the rank of the regular part of the connection at $a_k$. The dimension of the torsor $H(\partial_{a_k})$ is 
\[ 
\dim H(\partial_{a_k})=\sum_{i^{(k)}} {n^{(k)}_i}^2 \beta^{(k)}_i +{n_k}^2.
\]
Therefore, we get from \eqref{dimension_A} the following expression for $\dim \Acal(V_{a_k}^0)$:
\begin{equation}
\dim \Acal(V_{a_k}^0)=n^2+\sum_{i^{(k)}} ({n^{(k)}_i}^2 \beta^{(k)}_i +{n_k}^2) + \sum_{i^{(k)}} \alpha^{(k)}_i(\beta^{(k)}_i-1){n^{(k)}_i}^2 + \sum_{i^{(k)}<j^{(k)}} 2\beta^{(k)}_j \alpha^{(k)}_i n^{(k)}_i n^{(k)}_j +\sum_{i^{(k)}} 2\alpha^{(k)}_i n^{(k)}_i n_k.
\end{equation}
Let $\Ccal_i^{(k)}$ be the twisted conjugacy class of the monodromy of the irregular circle $\cir{q^{(k)}_i}$, and $\Ccal_k^{\text{reg}}$ the conjugacy class of the monodromy of the tame circle in $V^0_{a_k}$, so that 
\[ \Ccal(\partial_{a_k})\simeq \prod_{i^{(k)}} \Ccal_i^{(k)}\times \Ccal_k^{\text{reg}} \subset H(\partial_{a_k}).
\]
Using $\dim Z(G)=1$, the formula \eqref{dimension_M_B} for the dimension of the wild character variety then becomes
\begin{equation}
\begin{split}
\dim \mathcal{M}_B(V)=&  2-2n^2 \\
&+ \sum_k \left( n^2 -\sum_{i^{(k)}}({n^{(k)}_i}^2 \beta^{(k)}_i +{n_k}^2) + \sum_{i^{(k)}} \alpha^{(k)}_i(\beta^{(k)}_i-1){n^{(k)}_i}^2 \right. \\& \left. + \sum_{i^{(k)}<j^{(k)}} 2\beta^{(k)}_j \alpha^{(k)}_i n^{(k)}_i n^{(k)}_j +\sum_{i^{(k)}} 2\alpha^{(k)}_i n^{(k)}_i n_k \right)\\
&+  n^2 -\sum_{i^{(\infty)}}{n^{(\infty)}_i}^2 \beta^{(\infty)}_i + \sum_{i^{(\infty)}} \alpha^{(\infty)}_i(\beta^{(\infty)}_i-1){n^{(\infty)}_i}^2 \\&  + \sum_{i^{(\infty)}<j^{(\infty)}} 2\beta^{(\infty)}_j \alpha^{(\infty)}_i n^{(\infty)}_i n^{(\infty)}_j \\
&+  \sum_k \left(\sum_{i^{(k)}} \dim \Ccal_i^{(k)} + \dim \Ccal_k^{\text{reg}}\right) + \sum_{i^{(\infty)}} \dim \Ccal_i^{(\infty)} + \dim \Ccal_{\infty}^{\text{reg}}
\end{split}
\end{equation}

We have to compare these two expressions. First, we relate the contributions of the legs to the dimension of the conjugacy classes. 

\subsection*{Dimensions of conjugacy classes} We have to deal with a subtlety: the leg $\Lbb_I$ which we glue to a circle $I$ of ramification order $\beta$ over the pole $a\in \Pbb^1$ of the core of the diagram corresponds to the conjugacy class of the monodromy of the local system $V_I^0\to I$, whereas the twisted conjugacy class $\Ccal(I)\subset H(I)$ is the one of the monodromy of $V_I^0\to \partial_a$, i.e. the local system downstairs on the circle of directions $\partial_a$. 
\[
\begin{tikzcd}
V_I^0 \arrow[r,"\phi"]  \arrow[rd,"\psi"] & I \arrow[d,"\pi"] \\
 & \partial_a 
\end{tikzcd}
\]
Let us fix a direction $d\in \partial_a$, let $d_0,\dots d_{\beta_1}$ be its preimages by the cover $\pi : I\to\partial_a$. The fibre of $V_I^0$ over $d$ is the direct sum 
\[ {V_{I,d}^0}= {V_{I,d_0}^0}\oplus \dots \oplus {V_{I,d_{\beta-1},}^0} 
\]
where ${V_{I,d_i}^0}$ is the fibre of $\phi: V_I^0\to I$ over $d_i$, and the indices $i$ are in $\Z/\beta \Z$. 
The monodromy of $V_I^0\to \partial_a$ is given by the collection of linear applications 
\[ \rho_i : V_{I,d_i}^0 \to V_{I,d_{i+1}}^0, \qquad i=0,\dots,\beta-1.
\]
The monodromy of $V_I^0\to I$ is then given by the product
\[ \rho=\rho_{\beta-1}\dots\rho_0 : {V_{I,d_0}^0}\to {V_{I,d_0}^0}
\]

The group $H_I:=\GL(V_{I,d_0}^0)\times \dots \times \GL(V_{I,d_{\beta-1}}^0)$ acts on the monodromy by 
\[ (k_0,\dots,k_{\beta-1}).\rho_i = k_{i+1}\rho_i k_i^{-1},
\]
where $\mathbf{k}=(k_0,\dots,k_{\beta-1})\in H_{I}$. The twisted conjugacy class $\Ccal_I$ is the orbit of $(\rho_0,\dots,\rho_{\beta-1})$ under this action. This induces on $\rho$ the transformation
\[ \mathbf{k}.\rho=k_0\rho k_0^{-1}.
\]
which is just the action of $\GL(V_{I,d_0}^0)$ by conjugation, so that its orbit, which we will denote $\widetilde{\Ccal}_I$, is the conjugacy class of the monodromy of $V_I^0\to I$. Furthermore, each fibre of the map 
\begin{align*}
 \Ccal_I &\to \widetilde{\Ccal}_I\\
  (\rho_i) &\mapsto \rho
\end{align*}
is isomorphic to $\GL(V_{I,d_1}^0)\times \dots \times \GL(V_{I,d_{\beta-1}}^0)$. 
It follows from this that the dimension of $\Ccal_I$ is 
\begin{equation}
\dim \Ccal_I =n^2_I (\beta-1) +\dim \widetilde{\Ccal}_I,
\label{link_reduced_conjugacy_class}
\end{equation}
where $n_I$ is the dimension of each fibre  ${V_{I,d_i}^0}$.

The legs of the diagram correspond to the reduced conjugacy classes $\widetilde{\Ccal}_I$. The dimension of a conjugacy class is obtained from the associated leg in the following way (see again \cite{boalch2012simply}):

\begin{lemma}
Let $\mathcal{C}\subset \GL_n(\C)$ a conjugacy class, $\mathbb{L}$ the associated leg (after choosing a marking), with $\mathbf{d}$ its dimension vector. The dimension of $\mathcal{C}$ is given by
\begin{equation}
\dim \mathcal{C}=2 n^2-(\mathbf{d},\mathbf{d}).
\end{equation}
\end{lemma}

We can now compare the different terms appearing in the quantities $D=2-(\mathbf{d},\mathbf{d})$ and $D'=\dim \MB(V)$. Let us consider the terms in $D$ involving the multiplicities $m_k$ of the tame circles: their contribution to $D$ is 
\begin{align*}
&\sum_k \sum_{i^{(k)}}-2\beta^{(k)}_i m_k n^{(k)}_i
 + \sum_k \sum_{j^{(\infty)}}2\beta^{(\infty)}_j m_k n^{(\infty)}_j -2\sum_k m_k^2 \\
 =&-2 \sum_k n_k(n-n_k) -2 m_k^2 + 2 \sum_k m_k n\\
 =&\sum_k 2m_k n_k -2 m_k^2,
\end{align*}
where we use that $n=n_k+ \sum_{i^{(k)}} \beta^{(k)}_i n^{(k)}_i$ for each $k$, and the similar formula at infinity. We can recognize that $2m_k n_k -2 m_k^2$ is exactly the contribution of the first edge of the leg associated to a special marking of the conjugacy class of the monodromy of the tame circle $\cir{0}_k$; i.e. the edge corresponding to the eigenvalue $1$. From this we deduce that 
\begin{equation}
D(\Lbb_{\cir{0}_{a_k}})+ 2m_k n_k -2 m_k^2 = \dim \Ccal_k^{\text{reg}}.
\end{equation} 
We also have for each irregular circle at finite distance
\[
D(\Lbb_{\cir{q^{(k)}_i}}) =\dim \widetilde{\Ccal}_i^{(k)},
\]
and similarly for the circles at infinity
\[
D(\Lbb_{\cir{q^{(k)}_i}}) =\dim \widetilde{\Ccal}_i^{(k)}.
\]
This already gives us equalities between some terms appearing in $D$ and $D'$. Let us deal with the remaining terms. First, for each pole $a_k$ at finite distance, we have an equality between the following term in $D$, which is the contribution of the irregular circles at $a_k$
\[
\begin{split}
& \sum_{i^{(k)}}\left( -2+(-\beta^{(k)}_i-1)(\alpha^{(k)}_i+\beta^{(k)}_i-1)\right){n^{(k)}_i}^2
 +\sum_{i^{(k)}<j^{(k)}} -2\beta^{(k)}_j(\alpha^{(k)}_i+\beta^{(k)}_i)n^{(k)}_i n^{(k)}_j\\
&+ \sum_{i^{(k)},j^{(\infty)}} 2\beta^{(\infty)}_j(\alpha^{(k)}_i+\beta^{(k)}_i)n^{(k)}_i n^{(\infty)}_j,
\end{split}
\]
and the following term in $D'$
\[
\begin{split}
& n^2 -\sum_{i^{(k)}}({n^{(k)}_i}^2 \beta^{(k)}_i +{n_k}^2) + \sum_{i^{(k)}} \alpha^{(k)}_i(\beta^{(k)}_i-1){n^{(k)}_i}^2 \\& + \sum_{i^{(k)}<j^{(k)}} 2\beta^{(k)}_j \alpha^{(k)}_i n^{(k)}_i n^{(k)}_j +\sum_{i^{(k)}} 2\alpha^{(k)}_i n^{(k)}_i n_k  + {n_i^{(k)}}^2(\beta^{(k)}_i-1),
\end{split}
\]
where the term ${n_i^{(k)}}^2(\beta^{(k)}_i-1)$ comes from equation \eqref{link_reduced_conjugacy_class}. 
To obtain this equality, we notice that since $n= \sum_{j^{(\infty)}}\beta^{(\infty)}_j n^{(\infty)}_j$,
\[\sum_{i^{(k)},j^{(\infty)}} 2\beta^{(\infty)}_j(\alpha^{(k)}_i+\beta^{(k)}_i)n^{(k)}_i n^{(\infty)}_j = n \sum_{i^{(k)},j^{(\infty)}} (\alpha^{(k)}_i+\beta^{(k)}_i)n^{(k)}_i, 
\]
then we decompose on the both sides the total rank: $n=\sum_{i^{(k)}} \beta^{(k)}_i n^{(k)}_i+n_k$. It is then a somewhat tedious but straightforward check that we have the same terms on both sides. 

It only remains to consider the contributions to $D$ and $D'$ of the circles at infinity: we have an equality between the following term in $D$ 
\[ \sum_{i^{(\infty)}}\left( -2+(\beta^{(\infty)}_i-1)(\alpha^{(\infty)}_i-\beta^{(\infty)}_i-1)\right){n^{(\infty)}_i}^2 
 +\sum_{i^{(\infty)}<j^{(\infty)}} 2\beta^{(\infty)}_j(\alpha^{(\infty)}_i-\beta^{(\infty)}_i)n^{(\infty)}_i n^{(\infty)}_j
\]
and the following term in $D'$: 
\[
\begin{split}
& -2n^2 + n^2 -\sum_{i^{(\infty)}}{n^{(\infty)}_i}^2 \beta^{(\infty)}_i + \sum_{i^{(\infty)}} \alpha^{(\infty)}_i(\beta^{(\infty)}_i-1){n^{(\infty)}_i}^2 \\&  + \sum_{i^{(\infty)}<j^{(\infty)}} 2\beta^{(\infty)}_j \alpha^{(\infty)}_i n^{(\infty)}_i n^{(\infty)}_j +\sum_{i^{(\infty)}} {n^{(\infty)}_i}^2(\beta^{(\infty)}_i-1),
\end{split}
\]
where once again the term ${n^{(\infty)}_i}^2(\beta^{(\infty)}_i-1)$ comes from \eqref{link_reduced_conjugacy_class}. These equalities exhaust all terms appearing in $D$ and $D'$, so this concludes the proof.

\end{document}